\documentclass[11pt,letterpaper,reqno]{amsart}
\usepackage{graphicx}
\usepackage{amsmath,amsthm,amsfonts,amssymb, mathabx, mathrsfs}
\usepackage[abbrev,non-sorted-cites]{amsrefs}
\usepackage{latexsym}
\usepackage{color}
\usepackage{verbatim}
\usepackage{hyperref}
\usepackage{appendix}
\usepackage{slashed}
\usepackage{comment}

\oddsidemargin=.in
\evensidemargin=.in
\numberwithin{equation}{section}
\theoremstyle{plain}
\newtheorem{theorem}{Theorem}[section]
\newtheorem{proposition}[theorem]{Proposition}
\newtheorem{lemma}[theorem]{Lemma}
\newtheorem{corollary}[theorem]{Corollary}
\newtheorem*{thmm*}{Main Theorem}

\theoremstyle{definition}
\newtheorem{definition}[theorem]{Definition}
\newtheorem{assumption}[theorem]{Assumption}

\newtheorem{remark}[theorem]{Remark}
\newtheorem*{ackn*}{Acknowledgments}
\newcommand\dd{{\mathrm d}}
\newcommand\id{{\mathrm{id}}}

\newcommand\sm{\sigma}
\newcommand\dt{\delta}
\newcommand\ep{\epsilon}
\newcommand\vep{\varepsilon}
\newcommand\vph{\varphi}
\newcommand\om{\omega}
\newcommand\ta{\theta}
\newcommand\gm{\gamma}
\newcommand\kp{\kappa}

\newcommand\af{\alpha}
\newcommand\bt{\beta}

\newcommand\ld{\lambda}
\newcommand\up{\upsilon}
\newcommand\vsm{\varsigma}

\newcommand\Sm{\Sigma}
\newcommand\Om{\Omega}
\newcommand\Up{\Upsilon}
\newcommand\Ld{\Lambda}
\newcommand\Dt{\Delta}
\newcommand\Gm{\Gamma}

\newcommand\bphi{\boldsymbol{\phi}}
\newcommand\bvep{\boldsymbol{\varepsilon}}
\newcommand\bd{\mathbf{d}}

\newcommand\ba{\mathbf{a}}
\newcommand\bx{\mathbf{x}}

\newcommand\bm{\mathbf{m}}

\newcommand\bfN{\mathbf{N}}
\newcommand\bfU{\mathbf{U}}

\newcommand\RSU{\mathrm{SU}}
\newcommand\RSO{\mathrm{SO}}
\newcommand\RH{\mathrm{H}}

\newcommand\BC{\mathbb{C}}
\newcommand\BR{\mathbb{R}}
\newcommand\BN{\mathbb{N}}
\newcommand\BZ{\mathbb{Z}}

\newcommand\CD{\mathcal{D}}

\newcommand\pl{\partial}

\newcommand\w{\wedge}
\newcommand\ul{\underline}
\newcommand\td{\tilde}

\newcommand\wtd{\widetilde}

\newcommand\fe{\mathfrak{e}}

\newcommand\ptl{\mathfrak{E}}

\newcommand\ptq{\mathfrak{Q}}
\newcommand\stm{\underline{N}}
\newcommand\out{X^{\mathrm{o}}}

\newcommand\srp{\dagger}
\newcommand\ir{\mathfrak{r}}
\newcommand\is{\mathfrak{s}}
\newcommand\xsin{x_{\star}}

\newcommand\coneC{\mathbf{C}}

\DeclareMathOperator\re{Re}
\DeclareMathOperator\im{Im}
\DeclareMathOperator\coker{coker}
\DeclareMathOperator\ind{Ind}

\title[Infinite-Time Singularities of LMCF]{Infinite-Time Singularities of\\the Lagrangian Mean Curvature Flow}
\subjclass{Primary 53E10, Secondary 53D12}

\author{Wei-Bo Su}
\address{Department of Mathematics, National Central University, Taoyuan City 320317, Taiwan}
\email{weibosu@math.ncu.edu.tw}

\author{Chung-Jun Tsai}
\address{Department of Mathematics, National Taiwan University, and National Center for Theoretical Sciences, Mathematics Division, Taipei 106319, Taiwan}
\email{cjtsai@ntu.edu.tw}

\author{Albert Wood}
\address{Department of Mathematics, The Chinese University of Hong Kong, Ma Liu Shui, Hong Kong}
\email{albertwood@math.cuhk.edu.hk}

\date{\today}

\begin{document}
\graphicspath{images/}
%%%%%%%%%%%%%%%%%%%%%%%%%%%%%%%%%%%%%%%%%%%%%%%%%%%%%%%%%%%%%%%%
\thanks{W.-B.~Su is supported in part by the Taiwan MOST grant 111-2917-I-564-002.  C.-J.~Tsai is supported in part by the Taiwan NSTC grants 111-2636-M-002-022, 112-2636-M-002-003 and 112-2628-M-002-004-MY4.  A.~Wood is supported in part by a Simons Collaboration Grant on “Special holonomy in Geometry, Analysis and Physics”.  This work is supported in part by the National Center for Theoretical Sciences, Mathematics Division.}

\begin{abstract}
In this paper, we construct solutions of Lagrangian mean curvature flow which exist and are embedded for all time, but form an infinite-time singularity and converge to an immersed special Lagrangian as $t\to\infty$. In particular, the flow decomposes the initial data into a union of special Lagrangians intersecting at one point. This result shows that infinite-time singularities can form in the Thomas--Yau \cite{TY02} `semi-stable' situation. A precise polynomial blow-up rate of the second fundamental form is also shown.

The infinite-time singularity formation is obtained by a perturbation of an approximate family $N^{\vep(t)}$ constructed by gluing in special Lagrangian `Lawlor necks' of size $\vep(t)$, where the dynamics of the neck size $\vep(t)$ are driven by the obstruction for the existence of nearby special Lagrangians to $N^{\vep(t)}$. This is inspired by the work of Brendle and Kapouleas \cite{BrendleK2017} regarding ancient solutions of the Ricci flow.

% Our work is a parabolic analogue of the results of Joyce \cite{Joyce2003SLCS5} and Lee \cite{Lee2003} regarding desingularisation of special Lagrangians with conical singularities. The gluing construction that we employ is inspired by the work of Brendle and Kapouleas \cite{BrendleK2017} regarding ancient solutions of the Ricci flow.
\end{abstract}
%%%%%%%%%%%%%%%%%%%%%%%%%%%%%%%%%%%%%%%%%%%%%%%%%%%%%%%%%%%%%%%%
\maketitle
\setcounter{tocdepth}{1}
\tableofcontents

%%%%%%%%%%%%%%%%%%%%%%%%%%%%%%%%%%%%%%%%%%%%%%%%%%%%%%%%%%%%%%%%
\section{Introduction}

%%%%%%%%%%%%%%%%%%%%%%%%%%%%%%%%
\subsection{Singularities of Lagrangian Mean Curvature Flow}

The celebrated theorem of Yau \cite{Yau78} states that if the canonical bundle of a K\"ahler manifold is holomorphically trivial, then it admits a Ricci flat K\"ahler metric, referred to as the Calabi--Yau metric.  Over the past four decades, understanding the Lagrangian submanifolds minimal with respect to such a metric (known as special Lagrangians \cite{HL1982}) has been a major direction in differential geometry.  Calabi--Yau manifolds and their special Lagrangians also appear in various proposals of theoretical physics.  In particular, Strominger, Yau and Zaslow \cite{SYZ96} have proposed to use special Lagrangian fibrations to understand the mirror symmetry of Calabi--Yau manifolds.

The basic question about the existence of a special Lagrangian representative for a given homology class in a Calabi--Yau manifold is still open.  In contrast, the Lagrangian condition is symplectic-topological, and so it is easier to find Lagrangian submanifolds. Moreover, the Lagrangian condition is preserved along the mean curvature flow if the ambient metric is Calabi--Yau \cites{Oh1994,Smoczyk96}, a process known as Lagrangian mean curvature flow. One can therefore naturally deform a Lagrangian submanifold by its mean curvature vector to decrease its volume, and hope that the flow will exist forever and converge to a special Lagrangian submanifold. Motivated by mirror symmetry, Thomas and Yau \cites{Thomas00, TY02} proposed a conjectural picture for Lagrangian mean curvature flow, relating the behavior of the flow to a ``stability" property of the Lagrangian cycle; these conjectures have since been refined and reformulated by Joyce \cite{Joyce2015}.

In practice, the Lagrangian mean curvature flow often forms singularities in finite-time. In fact, Neves \cites{Neves2007, Neves2013} constructed examples of Lagrangian mean curvature flow forming a finite-time singularity within \emph{any} Hamiltonian isotopy class of Lagrangians, in the case of 2-dimensional Lagrangians. A resolution of the Thomas--Yau conjecture will therefore require a detailed understanding of singularity formation. Recently, Lotay--Schulze--Sz\'ekelyhidi \cite{LSchSz} studied the singularity formation in $2$-dimensions, assuming that a tangent flow at a finite-time singularity is given by a special Lagrangian pair of planes. In particular, by proving the uniqueness of the singular tangent flow together with the uniqueness results in \cites{IJO, LLSch}, they showed that the singularity formation is modeled on a family of shrinking `Lawlor necks', the unique special Lagrangian resolving a special Lagrangian pair of planes in $\mathbb{C}^{2}$.  

In this paper, we study the complementary phenomenon of infinite-time singularities, in order to improve our understanding of the Thomas--Yau picture.  Explicitly, we show that there exist Lagrangian mean curvature flows that exist for time $t \in [\Lambda, \infty)$ for which the flow converges to a \emph{singular} special Lagrangian as $t \to \infty$, and the singularity formation is modeled on a family of shrinking Lawlor necks with an explicit rate. Since the limit is not smooth, the second fundamental form cannot remain uniformly bounded; this is therefore an example of an \emph{infinite-time singularity} of Lagrangian mean curvature flow. To the authors' knowledge, this is the first example of an infinite-time singularity of mean curvature flow in the compact setting. 

We remark that Chen and He \cite{CH2010} proved that the mean curvature flow cannot have an infinite-time singularity when the ambient manifold is non-compact satisfying some mild conditions.  In contrast, using a rotationally symmetric Ansatz, Chen and Sun \cite{CS23} recently constructed a non-compact example in $\BR^3$ with an infinite-time singularity.

%%%%%%%%%%%%%%%%%%%%%%%%%%%%%%%%
\subsection{Desingularising Special Lagrangians with Isolated Conical Singularities}

We now give details of our construction. Consider an immersed special Lagrangian $X^m$ in a Calabi-Yau manifold $M^{2m}$, whose singular points are modelled on the transverse intersection of two half-dimensional planes. If the two half-dimensional planes at a singular point satisfy the \emph{angle criterion} (also known as a type 1 intersection, see Definition \ref{int_type}), then there exists an asymptotically conical special Lagrangian $L \subset \mathbb{C}^m$ (known as a \emph{Lawlor Neck} \cite{Lawlor1989}) with these same planes as asymptotes. One may therefore `glue in' this Lawlor neck at scale $\vep$ at the singular point to produce an almost-minimal Lagrangian desingularisation $N^\vep$ (see Figure \ref{fig-intro}).

A natural question is whether the desingularisation $N^\vep$ can be perturbed to a smooth special Lagrangian submanifold.  This question was studied thoroughly by Joyce in a series of papers \cites{Joyce2002SLCS1, Joyce2004SLCS2, Joyce2004SLCS3, Joyce2004SLCS4, Joyce2003SLCS5}, and also by Lee \cite{Lee2003}.  The upshot is that as long as the immersed special Lagrangian satisfies a ``balancing" condition,  the desingularisation can be perturbed into a special Lagrangian \cite{Joyce2003SLCS5}*{Theorem 9.7}.  These theorems can be applied to construct interesting examples of special Lagrangians; see for instance \cite{Hattori19}.

An overview of Joyce's construction is as follows. Firstly, nearby Lagrangians are represented as graphs of closed one forms in the Lagrangian neighborhood of $N^\vep$, and the special Lagrangian equation is expressed as a scalar equation on $N^\vep$ in the potential functions. The potential function of the mean curvature vector corresponds to the Lagrangian angle, and the linearised operator of the special Lagrangian equation is the Laplace operator on $N^\vep$.  In general, the linearised operator may have eigenfunctions with small eigenvalues (relative to the size of the neck), which means the inverse is not bounded independently of $\vep$. However, the balancing condition guarantees that the orthogonal projection of the mean curvature potential to the small eigenspace is sufficiently small, and can be ignored. One can therefore construct an iteration map using the inverse of the linearised operator, and by applying this map iteratively converge to a solution.

In this work, we consider an immersed special Lagrangian $X$ with only one singular point $\xsin$, such that the tangent cone at $\xsin$ satisfies the angle criterion. For this configuration, Joyce's balancing condition means that the complement $X\setminus\{\xsin\}$ is \emph{connected}, and in this case one can apply Joyce's result to desingularise $X$ and perturb to a special Lagrangian. In contrast, if $X\setminus\{\xsin\}$ is \textit{not} connected (e.g. as in Figure \ref{fig-intro}), the linearised operator of the special Lagrangian equation on the desingularisation $N^\vep$ has a one-dimensional space of non-trivial eigenfunctions with small eigenvalues, which acts as an `obstruction' to finding a special Lagrangian nearby to $N^\vep$. The mean curvature flow of $N^\vep$ will therefore not flow to a nearby special Lagrangian, and indeed a heuristic calculation (described in more detail below) suggests that under mean curvature flow, the neck size of $N^\vep$ will decrease, and form a singularity in infinite time. 

Our main theorem verifies that this infinite-time singular behaviour occurs for a particular example of the above configuration: two intersecting special Lagrangian tori in the complex torus. We consider a one-parameter family of desingularisations $N^{\vep(t)}$ for a suitable decreasing function $\vep: [\Lambda, \infty) \to \mathbb{R}^+$, and show that one may perturb the entire family to a Lagrangian mean curvature flow. The main result may be summarised as follows; a more precise statement is given as Theorem \ref{thm-main_precise}, Proposition~\ref{prop: tip convergence}, and Corollary~\ref{cor: curvature blow up}.

\begin{thmm*}\label{thm-main}
    Let $m \geq 3$, and endow $T^{2m} = \BC^m/\Gm$ with the Calabi--Yau structure induced from the standard one on $\BC^m$.  Suppose that $X_1$ and $X_2$ are two special Lagrangian sub-tori in $T^{2m}$ intersecting transversely at a point $\xsin$, and suppose the tangent planes at $\xsin$ satisfy the angle criterion.
    
    Then for $\vep_{0}$ small enough, there exists a desingularisation $N^{\vep_0}$ of $X := X_1 \cup X_2$ obtained by gluing in a Lawlor neck $L$ at scale $\vep_0$, and a Lagrangian mean curvature flow $N_{t}$ starting from $N^{\vep_{0}}$ existing for all time. Moreover, the flow satisfies the following asymptotic behaviour as $t\to\infty$:
    \begin{itemize}
        \item (infinite-time singularity) $N_{t}\to X_1 \cup X_2$ smoothly away from $\xsin$.
        \item (blow-up rate of curvature) The second fundamental form $A_{N_{t}}$ satisfies $|A_{N_{t}}| = O(t^{\frac{1}{m-2}})$.
        \item (singularity model) There exists a smooth $\vep:(0, \infty)\to\mathbb{R}_{+}$ satisfying\\ $C^{-1}t^{-\frac{1}{m-2}}\leq\vep(t)\leq Ct^{-\frac{1}{m-2}}$ for some $C>0$ and $t$ sufficiently large, such that
        \begin{align*}
            \vep(t)^{-1}(B_R(\xsin)\cap N_{t})\to L
        \end{align*}
        locally smoothly, for a suitable neighborhood of $\xsin$.
    \end{itemize}
    % As $t\to\infty$, the flow converges to $L$ and forms an infinite-time singularity modelled on a shrinking Lawlor neck of size $O(t^{-\frac{1}{m-2}})$, and the second fundamental form blows up like $O(t^{\frac{1}{m-2}})$.
\end{thmm*}

Note the dimensional constraint $m \geq 3$: such a dimensional constraint also appears in the work of Joyce \cites{Joyce2003SLCS5} and Lee \cite{Lee2003}. One reason for this constraint is that the Green's function in dimension $2$ is different from that in higher dimensions, which causes various analytic issues.

We remark that since $X_{1}$ and $X_{2}$ are special Lagrangian of the same phase, the initial data $N^{\vep_0}$ can be viewed as a {\it Thomas--Yau semi-stable} Lagrangian. Hence, our result indicates that in the semi-stable case, even if one has long-time existence of the flow, the convergence to the limiting special Lagrangian may not be smooth.

In \cite{DLee2004}, Lee proved that by allowing a perturbation of the ambient Calabi--Yau structure, $N^{\vep}$ can still be perturbed into a special Lagrangian.  Our main result can be viewed as a parabolic analogue of Lee's result in the sense that, by allowing the neck size to change, $N^{\vep(t)}$ can be perturbed into a Lagrangian mean curvature flow. Our main theorem can also be viewed as a {\it dynamic stability} result for the `singular' special Lagrangian $X = X_{1}\cup X_{2}$, as a critical point of the volume functional.

\subsection{Remarks on the Proof of the Main Theorem}

Our construction is based on a parabolic gluing technique. Over the past five years, there have been several works based on this method; in particular the work of Brendle and Kapouleas \cite{BrendleK2017} on the Ricci flow provides a strong intellectual input to this work (see also \cites{ADGW20,CDK21,DPS18,WZZ22} for other geometric flows). The idea is to start with a one-parameter family of desingularisations $N^{\vep(t)}$ obtained by gluing in a Lawlor neck $L$ at scale $\vep(t)$ to $X$ at $\xsin$, and perturb to a genuine Lagrangian mean curvature flow. The desingularisations $N^{\vep(t)}$ are formulated as a time-dependent embedding $\iota^{\vep(t)}$ from a static manifold $\stm$ to $T^{2m} = \BC^m/\Gm$, so that time derivatives on $\stm \times [\Lambda, \infty)$ are meaningful.  Nearby Hamiltonian isotopic Lagrangians to $N^\vep$ are graphs of exact one-forms in any Lagrangian neighborhood, so are parametrised by functions on $\stm$. Our aim is therefore to find $u:\stm \times [\Lambda, \infty) \to \mathbb{R}$ representing a Lagrangian mean curvature flow nearby to $N^{\vep(t)}$.

Firstly, we express the Lagrangian mean curvature flow equation as a non-linear parabolic scalar equation in $u$. A complication arises in our case, since the desingularisations $N^{\vep(t)}$, and therefore the Lagrangian neighborhoods themselves, are varying with time - this produces an extra term in the equation which a priori may not be integrable to the level of potentials. These issues are not present in the prior work of Joyce and Lee, and are unique to the parabolic case. To resolve them, it is required to carefully construct suitable `exact' Lagrangian neighborhoods for conical and asymptotically conical Lagrangians, which are employed in the construction of Lagrangian neighborhoods for $N^\vep$.

Besides solving the parabolic equation for the perturbation, another  important aspect of our work is the choice of the neck parameter $\vep(t)$. Since the Lawlor neck is a minimal submanifold, to first order the neck is not shrinking along the flow. It is therefore necessary to look for a suitable `external force' for the definition of $\vep(t)$. This force arises from the `balancing condition' of Joyce. In Joyce's elliptic setting, the orthogonal projection of the mean curvature potential to the approximate kernel is not sufficiently small to allow one to perturb to a special Lagrangian, as previously mentioned. However, in our parabolic problem the orthogonal projection includes time derivatives of $\vep(t)$, and so the equivalent balancing condition is an ODE in $\vep(t)$,
\begin{align}\label{eq: approx dynamics}
    \frac{\dd}{\dd t}(\vep(t))^2 &= -c\,(\vep(t))^m + o(\vep(t)^{m}),
\end{align}
which, up to higher order terms, is solved by $\vep_{0}(t) = (\vep_0^{2-m} + \frac{c}{2}(m-2)\,t)^{-\frac{1}{m-2}}$. Therefore the family $\{N^{\vep_{0}(t)}: \vep_{0}(t) = (\vep_0^{2-m} + \frac{c}{2}(m-2)\,t)^{-\frac{1}{m-2}}  \text{ for } 0\leq t<\infty\}$ is `closest' to a genuine Lagrangian mean curvature flow.

Our goal now is to perturb from this approximate flow to a genuine flow. This requires solving a coupled system of the nonlinear parabolic equation for the potential and the balancing ODE (\ref{eq: approx dynamics}). We employ an iteration scheme inspired by \cite{BrendleK2017}, in which Brendle and Kapouleas construct an ancient solution to the Ricci flow using the obstruction of existence of Einstein metrics. Note that the neck parameter $\vep(t)$ is decreasing in time, producing an infinite-time singularity; this contrasts with the ancient Ricci flow in \cite{BrendleK2017}.

Finally, we remark that the restriction to the case of Calabi--Yau tori in our main theorem is in order to make the error terms small enough that the approximate dynamics (\ref{eq: approx dynamics}) dominate. We aim to address the general case in our upcoming work. 

%%%%%%%%%%%%%%%%%%%%%%%%%%%%%%%%

\begin{figure}
\centering
\includegraphics[scale=0.5]{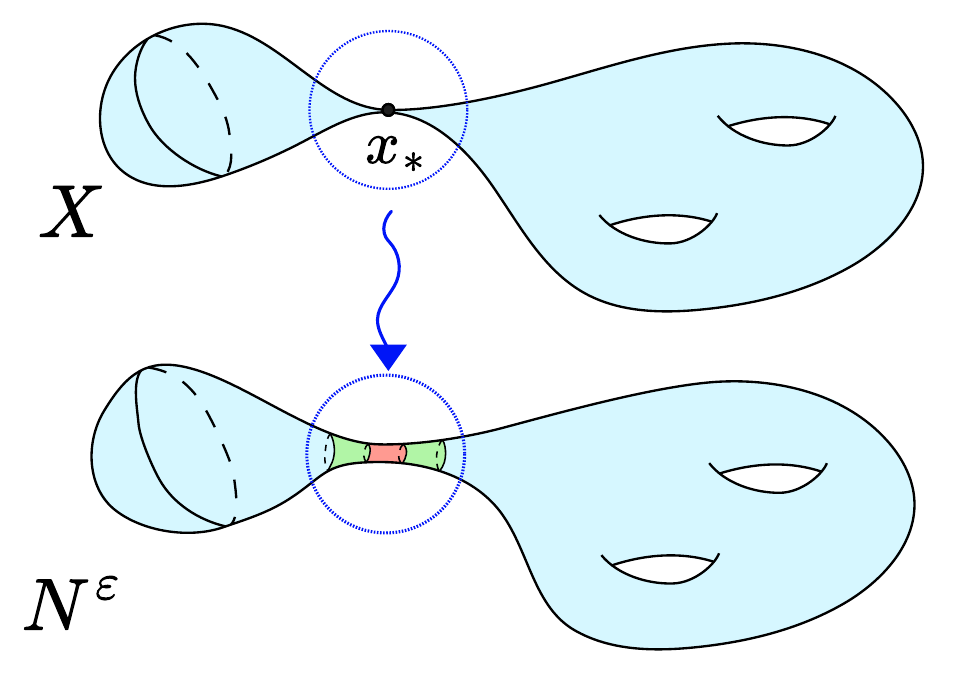}
\caption{An diagram of a special Lagrangian with a single immersed point $X$ such that $X\setminus \{\xsin \}$ is disconnected, along with its desingularisation $N^\vep$. The desingularisation is obtained by `gluing in' an asymptotically conical special Lagrangian (a Lawlor neck) at scale $\vep$. The Laplacian on $N^\vep$ has a one-dimensional space of non-trivial eigenfunctions with small eigenvalues.}
\label{fig-intro}
\end{figure}

\subsection{Structure of the Paper} Section \ref{sec_Lag_nbd} is devoted to the construction of `exact' Lagrangian neighborhoods of Lagrangian cones and asymptotically conical Lagrangians.  In section \ref{sec-desingularisation}, the desingularisation $N^{\vep}$ is introduced, as the image of an $\vep$-dependent map from a static manifold $\iota^\vep:\stm \to T^{2m} = \BC^m/\Gm$. The non-parametric form of the Lagrangian mean curvature flow equation is derived in section \ref{sec-lmcfequation}.
%The condition under which the Lagrangian mean curvature flow equation can be integrated to the level of potentials is also discussed in section \ref{sec-lmcfequation}.
In section \ref{sec-linear-approxkernel}, we compute the linearised operator, and introduce its approximate kernel.  The spatial properties of the approximate kernel are established by Joyce; the materials in sections \ref{sec_Lag_nbd} and \ref{sec-desingularisation} allow us to study their parabolic properties.
In section \ref{sec-apriori}, we prove three Liouville theorems, and use them to establish the weighted Schauder estimate for the solution to the inhomogeneous heat equation.
%We note that the discussions in sections \ref{sec_Lag_nbd} to \ref{sec-apriori} are valid not only for the specific case of our main theorem, but for general immersed Lagrangians in a general ambient Calabi--Yau manifold, where the intersections satisfy the angle condition.  

%From section \ref{sec-exist-torus}, we focus on the case of two intersecting Lagrangian tori in a complex torus.
The main purpose of section \ref{sec-exist-torus} is to establish an existence theorem for solutions to the heat equation on the $L^2$-orthogonal complement of the approximate kernel.  In section \ref{sec-estimate-error}, we derive the projection formula to the approximate kernel, and the estimates of the zeroth order and quadratic terms of the Lagrangian mean curvature flow equation.  Finally, these materials are put together in section \ref{sec-iteration}, and the main theorem is proven by a Schauder fixed point argument.

\subsection{Conventions}
Here are some conventions that will be used throughout this paper.
\begin{enumerate}
\item Unless otherwise specified, $\BC^{m}\cong\BR^{2m}$ is equipped with the standard Calabi--Yau structure $(g_0,J_0,\om_0,\Om_0)$, where $\om_0 = \sum_{j=1}^{m}\dd x_j\w\dd y_j$ and $\Om_0 = \dd z_1\w\ldots\w\dd z_m$.

\item The complex dimension of $T^{2m} = \BC^m/\Gm$ is assumed to be $3$ or greater, $m\geq 3$.  It is always endowed with the Calabi--Yau structure induced from the standard on on $\BC^m$, which will be denoted by $(g,J,\om,\Om)$.

\item A half-dimensional submanifold $L^m$ in $T^{2m}$ is called a \emph{special Lagrangian} submanifold if it is calibrated by $\im\Om$.  Namely, $\im\Om|_L$ coincides with the volume form of $L$.  According to \cite{HL1982}*{p.89}, this is equivalent to the vanishing of $\om|_L$ (the Lagrangian condition) and the vanishing of $\re\Om|_L$ (the special condition).

\item Given a diffeomorphism $\vph: L\to \td{L}$, it induces a diffeomorphism $(\vph^*)^{-1}:T^*L\to T^*\td{L}$.  Such a map will be denoted by $\vph_\srp$.  Given a smooth function $u:\td{L}\to \BR$, $\dd u$ embeds $\td{L}$ into $T^*\td{L}$.  It is straightforward to verify the following relation
\begin{align}
	(\dd u)\circ\vph &= \vph_\srp\circ\dd(u\circ\vph) ~. \label{sharp0}
\end{align}

\item The constant $C$ in the estimates may change from line to line.
\end{enumerate}

\begin{ackn*}
    The authors would like to thank Dominic Joyce and Yng-Ing Lee for their helpful discussions and interest in this work.  The authors are grateful to Simon Brendle for answering our questions regarding \cite{BrendleK2017}.
\end{ackn*}

%%%%%%%%%%%%%%%%%%%%%%%%%%%%%%%%%%%%%%%%%%%%%%%%%%%%%%%%%%%%%%%%
\section{neighborhood Theorems and Local Models} \label{sec_Lag_nbd}

%%%%%%%%%%%%%%%%%%%%%%%%%%%%%%%%
\subsection{Equivariant neighborhoods of Lagrangian Cones}

In this section, we consider a Lagrangian cone $\coneC \subset \mathbb{C}^m \cong \mathbb{R}^{2m}$, where $\BC^{m}\cong\BR^{2m}$ is equipped with the standard Liouville form $\ld_0 = \frac{1}{2}\sum_{j=1}^{m}(y_j\dd x_j - x_j\dd y_j)$, so that $\dd\ld_0 = -\om_0$. Note that the link $\Sm = \coneC\cap S^{2m-1}$ is a Legendrian submanifold in the contact manifold $(S^{2m-1},\left.\ld_0\right|_{S^{2m-1}})$.  On the other hand, one can equip $T^*\Sm\times\BR$ with the contact form $\ld_\Sm - \dd s$, where $\ld_\Sm$ is the tautological $1$-form on $T^*\Sm$ and $s$ is the coordinate on $\BR$.  It is clear that $\Sm$, as the zero section of $T^*\Sm$, is a Legendrian submanifold in $(T^*\Sm\times\BR, \ld_\Sm-\dd s)$.  By Moser's trick, one can show that the latter is the standard local model of Legendrian neighborhoods.

%%%%%%%%
\begin{lemma}[\cite{Lyc1977}] \label{LN of links}
Denote by $\ul{0}$ the zero section in $T^*\Sm\subset T^*\Sm\times\BR$.  There exist an open neighborhood $W_\Sm$ of $\ul{0}$ in $T^*\Sm\times\BR$ and an embedding $\Psi_\Sm: W_\Sm\to S^{2m-1}$ such that $\left.\Psi_\Sm\right|_{\ul{0}} = \iota_\Sm$ and $\Psi^{*}_\Sm(\left.\ld_{0}\right|_{S^{2m-1}}) = \ld_\Sm - \dd s$, where $\iota_\Sm: \Sm\to S^{2m-1}$ is the inclusion map.
\end{lemma}

%%%%%%%%
Using Lemma \ref{LN of links}, we construct an equivariant Lagrangian neighborhood for $C$. Recall the notion of a Lagrangian neighborhood.
\begin{definition}
Let $\iota:L^m\to(M^{2m},\om)$ be a Lagrangian embedding.  A \emph{Lagrangian neighborhood} consists of an open neighborhood $U\subset T^*L$ of the zero section $\ul{0}$, and an embedding $\Psi_L:U\to M$ such that $\left.\Psi_L\right|_{\ul{0}} = \iota$ and $\Psi_L^*(\om) = \om_L$, where $\om_L$ is the canonical symplectic form on $T^*L$.
\end{definition}

We consider the natural $\BR_+$-action on $\BC^{m}\setminus\{0\}$ given by dilations, and the $\BR_+$-action on $T^*\coneC = T^*(\Sm\times(0,\infty))$ defined as follows.  Formally writing a point in $T^*(\Sm\times(0,\infty))$ as $(\sm,r,\vsm,s)$ where $\sm\in\Sm$, $r\in(0,\infty)$, $\vsm\in T^*_\sm\Sm$ and $s\in\BR\cong T^*_r(0,\infty)$, and letting $\ep\in\BR_+$, we define
\begin{align}\label{cone scaling}
	\ep\cdot(\sm, r, \vsm, s) = (\sm, \ep r, \ep^{2}\vsm, \ep s) ~.
\end{align}
The following proposition gives not only the neighborhood, but also the expression of the Liouville form on the neighborhood.  It is an extension of \cite{Joyce2002SLCS1}*{Theorem 4.3}. 
 
%%%%%%%%
\begin{proposition} \label{LN of cones}
Let $\Sm$ be a Legendrian link in $(S^{2m-1},\left.\ld_0\right|_{S^{2m-1}})$, and let $\coneC = \Sm\times(0,\infty)$ be the corresponding Lagrangian cone in $(\BC^m\setminus\{0\},\om_0)$.  There exists a Lagrangian neighborhood $\Phi_{\coneC}: U_{\coneC}\subset T^*{\coneC} = T^*(\Sm\times (0, \infty)) \to \BC^m\setminus\{0\}$ such that
\begin{align*}
	\Phi_{\coneC}^{*}\ld_{0} = \ld_{\coneC} - \dd\left(\frac{rs}{2}\right) ~,
\end{align*}
where $\ld_{\coneC}$ is the tautological $1$-form on $T^*\coneC$, $r\in(0,\infty)$, and $s\in\BR\cong T^*_r(0, \infty)$.
Moreover, $U_{\coneC}$ is invariant under the $\mathbb{R}_{+}$-action defined in \eqref{cone scaling}, and $\Phi_{\coneC}$ is equivariant with respect to it.
\end{proposition}
%%%%
\begin{proof}
The \emph{symplectisation} of $(S^{2m-1},\left.\ld_0\right|_{S^{2m-1}})$ is $(\BR^{2m}\setminus\{0\},\om_0 = -\dd\ld_0)$.  More precisely, identify $\bx\in \BR^{2m}\setminus\{0\}$ with $(\frac{\bx}{|\bx|},|\bx|)\in S^{2m-1}\times(0,\infty)$.  Under this identification, $\ld_0$ is the pull-back of $r^2\left.\ld_0\right|_{S^{2m-1}}$, where $r$ is the coordinate on $(0,\infty)$.  With this understood, it is equivalent to construct the embedding $\Phi_{\coneC}$ to $S^{2m-1}\times(0,\infty)$ so that $\Phi_{\coneC}^{*}(r^2\left.\ld_0\right|_{S^{2m-1}}) = \ld_{\coneC} - \dd\left(\frac{rs}{2}\right)$.  Note that on $S^{2m-1}\times(0,\infty)$, the dilation acts only on the $(0,\infty)$-summand.

Consider the diffeomorphism
\begin{align*}
\begin{array}{cccl}
	\psi: &T^*(\Sm\times(0,\infty)) &\to &T^*\Sm\times\BR\times(0,\infty) \\
	&(\sm,r,\vsm,s) &\mapsto & \left((\sm,r^{-2}\vsm),(2r)^{-1}s,r\right) ~.
\end{array}
\end{align*}
For the open set $W_\Sm$ given by Lemma \ref{LN of links}, it is not hard to see that $U_{\coneC} = \psi^{-1}(W_\Sm\times(0,\infty))$ is an open neighborhood of the zero section in $T^*(\Sm\times(0,\infty))$.  Let
\begin{align*}
\Phi_{\coneC} &= (\Psi_\Sm\times\id_{(0,\infty)}) \circ \psi : U_{\coneC}\subset T^*(\Sm\times(0,\infty)) \to S^{2m-1}\times(0,\infty)
\end{align*}
where $\Phi_{\coneC}$ is given by Lemma \ref{LN of links}.  The pull-back of the Liouville form under $\Phi_{\coneC}$ is
\begin{align*}
	\Phi_{\coneC}^*(r^2\left.\ld_0\right|_{S^{2m-1}}) &= \psi^*(r^2 (\ld_\Sm - \dd s)) \\
	&= \ld_\Sm - r^2\dd\left(\frac{s}{2r}\right) = (\ld_\Sm + s\dd r) - \dd\left(\frac{rs}{2}\right) ~.
\end{align*}
Note that $\ld_\Sm+s\dd r$ is exactly the tautological $1$-form on $T^*(\Sm\times(0,\infty))$.

The invariance of $U_{\coneC}$ under \eqref{cone scaling} follows from the construction.  It remains to check the $\BR_+$-equivariance of $\Phi_{\coneC}$.  For any $\ep>0$,
\begin{align*}
	\Phi_{\coneC}(\ep\cdot(\sm, r, \vsm, s)) &= \Phi_{\coneC}(\sm, \ep r, \ep^2\vsm,\ep s) \\
	&=  (\Psi_\Sm\times\id_{(0,\infty)})\left((\sm,r^{-2}\vsm),(2r)^{-1}s, \ep r\right) = \ep\cdot\Phi_{\coneC}(\sm,r,\vsm,s) ~.
\end{align*}
This finishes the proof of the proposition.
\end{proof}

%%%%%%%%%%%%%%%%%%%%%%%%%%%%%%%%
\subsection{Asymptotically Conical Lagrangians}

Proposition \ref{LN of cones} can be used to construct good neighborhoods for asymptotically conical Lagrangians.  We first recall their definition.

%%%%%%%%
\begin{definition} \label{AC lag}
A Lagrangian $L\subset \BC^m$ is called \emph{asymptotically conical} with cone $\coneC$ and rate $\gm$ if the following holds.  Let $\Sm = \coneC\cap S^{2m-1}$ be the link of $\coneC$.  There exist a compact subset $K\subset L$, a constant $R_1 > 0$, and a diffeomorphism $\varphi: \Sm\times(R_1, \infty)\to L\setminus K$ such that for any non-negative integer $k$,
\begin{align}\label{AC condition}
	|\nabla^{k}(\varphi - \iota_{\coneC})|(\sm, r) = O(r^{\gm - 1- k}) \quad\text{as } r\to\infty ~,
\end{align}
where $\nabla$ and $|\cdot|$ are computed using the cone metric $g_{\coneC} = \dd r^{2} + r^{2}g_{\Sm}$.
\end{definition}
%%%%
\begin{remark}\hfill
\begin{itemize}
\item Later on, we will consider the ``potential" of $L$ over $\Sm\times(R_1,\infty)$.  The $-1$ in the power of $r$ in \eqref{AC condition} will imply that the potential is of order $\gm$.
\item In \cite{Joyce2004SLCS3}*{Definition 4.1}, \eqref{AC condition} is only required for $k=0,1$.  Under suitable assumptions, it can be upgraded to all $k\geq0$; see for instance Theorem 3.8 and Theorem 4.6 in \cite{Joyce2004SLCS3}.  Since we will only work with specific asymptotically conical Lagrangians, a more restrictive assumption is chosen here for convenience.
\end{itemize}
\end{remark}

Suppose that the rate satisfies $\gm<0$.  According to Proposition \ref{LN of cones}, $L\setminus K$ can be written as the graph of a smooth \emph{closed} $1$-form on $\Sm\times(R_1,\infty)$, after taking $R_1$ larger if necessary.  That is to say, $L\setminus K$ belongs to $\Phi_{\coneC}(U_{\coneC})$.  Thus, there exists a closed $1$-form, $\fe$, on $\Sm\times(R_1,\infty)$ such that
\begin{align}
	\vph(\sm,r) = \Phi_{\coneC}(\sm,r,\fe_1(\sm,r),\fe_2(\sm,r)) \label{def_kp}
\end{align}
for any $(\sm, r)\in \Sm\times (R_1, \infty)$, where $\fe_2 = \fe(\frac{\pl\,}{\pl r})$ and $\fe_1 = \fe - \fe_2\dd r$.  The condition \eqref{AC condition} implies that
\begin{align}
	|\nabla^{k}\fe |= O(r^{\gamma - 1 - k}) \quad\text{as } r\to\infty \label{kp rate}
\end{align}
where $\nabla$ and $|\cdot|$ are computed using the cone metric $g_{\coneC} = \dd r^{2} + r^{2}g_{\Sm}$.

Recall that a Lagrangian submanifold $L\subset\BC^m$ is said to be \emph{exact} if the restriction of the Liouville form, $\iota_L^*\ld_0$, is exact.  Here is the neighborhood theorem we will need.

%%%%%%%%
\begin{theorem} \label{AC Lag nbhd}
Let $L\subset(\mathbb{C}^{m},\om_0)$ be an exact, connected, asymptotically conical Lagrangian submanifold with cone $\coneC = \Sm\times(0,\infty)$ and rate $\gamma<0$. Then: 
\begin{itemize}
\item There exist a Lagrangian neighborhood $\Phi_{L}:U_{L}\subset T^*L \to \mathbb{C}^{m}$ and a function $\af_L:U_L\to\BR$ such that 
\begin{align*}
	\Phi_{L}^{*}\ld_{0} = \ld_{L} - \dd\af_{L} ~.
\end{align*}
Moreover, $\Phi_L$ can be chosen so that
\begin{align}
	(\Phi_{L}\circ\vph_{\srp})(\sm, r, \vsm, s) = \Phi_{\coneC}(\sm, r, \vsm+\fe_1(\sm,r), s + \fe_2(\sm,r)) \label{Phi kp}
\end{align}
for any $(\sigma, r, \vsm, s)\in \vph_\srp^{-1}(U_L)\subset T^*(\Sm\times(R_1,\infty))$, where $\Phi_{\coneC}$ is the map given by Proposition \ref{LN of cones}, $\vph$ is the map in Definition \ref{AC lag}, and $\fe = \fe_1 + \fe_2\dd r\in\Om^1(\Sm\times(R_1,\infty))$ is explained in \eqref{def_kp}.

\item The $1$-form $\fe$ on $\Sm\times(R_1,\infty)$ is exact, $\fe = \dd\ptl$, and thus $\vph = \Phi_{\coneC}\circ\dd\ptl$.  Moreover, the potential function $\ptl$ can taken to obey that $|\nabla^\ell\ptl| = O(r^{\gm-\ell})$ as $r\to\infty$, for every $\ell\geq0$.

\item The function $\af_L$ is unique up to adding a constant. Moreover, there are constants $c_a$, associated with the connected components $L_a$ of $L\setminus K$, and on each $\vph_\srp^{-1}(T^*L_a)$:
\begin{align}
 (\af_{L}\circ\vph_{\srp})(\sm, r, \vsm, s) - \left(\frac{rs}{2} + c_a \right) &= \frac{r}{2}(\pl_r\ptl)(\sm,r) - \ptl(\sm,r),  \label{afL_def} & \\
\left| (\af_{L}\circ\vph_{\srp})(\sm, r, \vsm, s) - \left(\frac{rs}{2} + c_{a}\right)\right| &= O(r^{\gm}) &\text{ as } r\to \infty. \notag
\end{align}
The restriction of $\alpha_L:U_L \rightarrow \mathbb{R}$ to the zero section $\beta_L := \alpha_L|_{\ul 0}$ is a primitive of the Liouville form up to a minus sign, $\iota^{*}_{L}\ld_{0} = -\dd\bt_{L}$. %, such that on each $\vph_\srp^{-1}(T^*L_a)$:
%\begin{align}
% (\beta_{L}\circ\vph)(\sm, r) + c_a &= -\frac{r}{2}(\pl_r\ptl)(\sm,r) + \ptl(\sm,r),  \notag & \\
%\left|(\beta_{L}\circ\vph)(\sm, r) + c_a\right| &= O(r^{\gm}) &\text{ as } r\to \infty. \notag
%\end{align}
\end{itemize}
\end{theorem}
%%%%
\begin{proof}
The argument is very similar to \cite{LotayNeves2013}*{Proposition 5.3}.  The proof is separated into 4 steps.  The bundle projection map is denoted by $\pi$.
\smallskip

%%%%
\noindent{\it Step 1: Lagrangian neighborhood near infinity}.
We first construct the neighborhood of $L\setminus K$.  Define an open subset $U_{\coneC} - \fe$ of $T^{*}(\Sm\times (R_1, \infty))$ by
\begin{align*}
	U_{\coneC} - \fe = \{ (\sigma, r, \vsm, s) :~ (\sm, r, \vsm+\fe_1(\sm,r), s + \fe_2(\sm,r))\in U_{\coneC}\} ~.
\end{align*}
Its image under $\vph_{\srp}$ is an open neighborhood of the zero section in $T^*(L\setminus K)$.  Denote its image, $\vph_{\srp}(U_{\coneC}-\fe) \subset T^*(L\setminus K)$, by $U_{L\setminus K}$.  One naturally defines an embedding $\Phi_{L\setminus K}: U_{L\setminus K}\subset T^*(L\setminus K)\to \mathbb{C}^{m}$ by
\begin{align}
	(\Phi_{L\setminus K}\circ\vph_{\srp})(\sm, r, \vsm, s) = \Phi_{\coneC}(\sigma, r, \vsm + \fe_{1}(\sm, r), s + \fe_{2}(\sm, r)) ~.
\label{Phi LK} \end{align}

We prove the exactness of $\fe$.  For the self-diffeomorphism $f_\fe$ on $T^*(\Sm\times(R_1,\infty))$ defined by
\begin{align}
	f_\fe(\sm, r, \vsm, s) = (\sm, r, \vsm + \fe_{1}(\sm, r), s + \fe_{2}(\sm, r)) ~, \label{1form_trans}
\end{align}
one has $f_\fe^*(\ld_{\coneC}) = \ld_{\coneC} + \fe$, where $\ld_{\coneC}$ is the tautological $1$-form on $T^*\coneC\supset T^*(\Sm\times(R_1,\infty))$, and $\fe$ is regarded as a $1$-form on $T^*(\Sm\times(R_1,\infty))$ under the pull-back of the projection.

Since $\dd\fe = \dd (\fe_1+\fe_2\dd r)= 0$,
\begin{align*}
	\dd_\Sm\fe_2 = \frac{\pl\fe_1}{\pl r}  \quad\text{and}\quad  \dd_\Sm\fe_1 = 0 ~.
\end{align*}
Since $|\fe| = O(r^{\gm-1})$ as $r\to\infty$, the function
\begin{align}
	\ptl(\sm, r) = -\int_r^\infty\fe_2(\sigma,y) \dd y
\label{exact_potential} \end{align}
is well-defined, and $|\ptl(\sm, r)| = O(r^{\gm})$ as $r\to\infty$. With $\dd\fe = 0$, one finds that $\dd\ptl = \fe$.  The rate on $|\nabla\ptl|$ follows directly from its construction.

According to Proposition \ref{LN of cones}, \eqref{Phi LK}, $\dd\ptl = \fe$, and the fact that $\vph_{\srp}$ preserves the tautological $1$-form,
\begin{align}
	\Phi_{L\setminus K}^{*}\ld_{0} = \left((\vph_{\srp}^{-1})^{*}\circ f_{\fe}^{*}\circ\Phi_{\coneC}^{*}\right)(\ld_{0}) &= (\vph_{\srp}^{-1})^{*}f_{\fe}^{*}\left(\ld_{\coneC} - \dd\left(\frac{rs}{2}\right)\right) \notag \\
	&= (\vph_{\srp}^{-1})^{*}\left[ \ld_C - \dd\left(f_{\fe}^{*}(\frac{rs}{2})  - \ptl\right)\right] \notag \\
	&= \ld_{L} - \dd\left[ \left( \frac{rs}{2} + \frac{r\fe_2}{2} - \ptl\right)\circ \vph_\srp^{-1} \right] \label{ld_LK}
\end{align}
where $\lambda_{L}$ is the tautological $1$-form on $T^{*}L \supset T^*(L\setminus K)$.
\smallskip

%%%%
\noindent{\it Step 2: constants}.
Since $L$ is exact, there exists $\bt_{L}:L\to \BR$ such that $\iota^{*}_{L}\ld_{0} = -\dd\bt_{L}$.  Due to the connectedness of $L$, $\bt_L$ is unique up to adding a constant.  Fix a choice of $\bt_L$.

On the other hand, the restriction of \eqref{ld_LK} on the zero section, $L\setminus K$, implies that
\begin{align*}
	\left(\left.\iota^*_{L}\ld_0\right)\right|_{L\setminus K} &= -\dd \left.\left[ \left(\frac{rs}{2} + \frac{r\fe_2}{2} - \ptl \right) \circ \vph_\srp^{-1} \right] \right|_{L\setminus K} .
\end{align*}
Therefore, one each connected component, $L_a$, of $L\setminus K$, there must exist a constant $c_a$ such that
\begin{align*}
	\left.\bt_L\right|_{L_a} &= \left.\left[ \left(\frac{rs}{2} + \frac{r\fe_2}{2} - \ptl \right) \circ \vph_\srp^{-1} \right] \right|_{L_a} + c_a.
\end{align*}
With these constants $c_a$'s, define $\af_{L\setminus K}: U_{L\setminus K}\to\BR$ by 
\begin{align}
	\left.\af_{L\setminus K}\right|_{\pi^{-1}(L_{a})} &= \left. \left[ \left(\frac{rs}{2} + \frac{r\fe_2}{2} - \ptl\right) \circ \vph_\srp^{-1} \right]\right|_{\pi^{-1}(L_{a})} + c_{a} ~. \label{af_LK}
\end{align}
\smallskip

%%%%
\noindent{\it Step 3: Moser's trick on $T^*K$}.
Let $\wtd{U}_{L}\subset T^{*}L$ and $\wtd{\Phi}_{L}:\wtd{U}_{L}\to\BC^{m}$ be smooth extensions of the open neighborhood $U_{L\setminus K}$ and the embedding $\Phi_{L\setminus K}$ over the compact subset $K$.  Namely, $\wtd{U}_L$ is an open neighborhood of the zero section in $T^*L$ with $\wtd{U}_{L}\cap\pi^{-1}(L\setminus K) = U_{L\setminus K}$, and $\wtd{\Phi}_{L}\big|_{U_{L\setminus K}} = \Phi_{L\setminus K}$.  Moreover, the embedding can be chosen so that $\wtd{\Phi}_{L}\big|_{\ul{0}} = \iota_{L}$.  The neighborhood of $L$ asserted in this theorem will be constructed by perturbing $\wtd{U}_{L}$ and $\wtd{\Phi}_{L}$.

Let $h\in C^{\infty}(L)$ be a cut-off function such that $h\equiv 1$ on $L\setminus K$, $h\equiv 0$ on $K'\subset\subset K$.  Define an extension $\wtd{\af}_{L}$ of $\af_{L\setminus K}$ \eqref{af_LK} to $\wtd{U}_{L}$ by 
\begin{align*}
	\wtd{\af}_{L} := (h\circ\pi)\,\af_{L\setminus K} + (1-h\circ\pi)(\bt_{L}\circ\pi) ~.
\end{align*}
From step 2, the restriction of $\wtd{\af}_{L}$ on the zero section is $\bt_{L}$, $\left.\wtd{\af}_{L}\right|_{\ul{0}} = \bt_{L}$.
The goal is to construct a one-parameter family of self-diffeomorphisms, $\{\up_{t}\}_{t\in[0,1]}$, of $\wtd{U}_{L}$ with the following properties.
\begin{itemize}
\item $\up_{0} = \id_{\wtd{U}_{L}}$ and $\left.\up_{t}\right|_{\underline{0}}\in\mbox{Diff}_{c}(L)$ for all $t\in[0,1]$.
\item Let $\ld^{t} = (1-t)(\ld_{L} - \dd\widetilde{\af}_{L}) + t\wtd{\Phi}_{L}^{*}\ld_{0}$.  There exists a family of functions, $\{\af_{t}\}_{t\in[0,1]}$, on $\wtd{U}_{L}$ with $\af_{0} = 0$ and
\begin{align}\label{moser family}
	\up_{t}^{*}\ld^{t} = \ld_{L} - \dd(\wtd{\alpha}_{L} + \af_{t})
\end{align}
for every $t\in[0,1]$.
\end{itemize}

Suppose $\frac{\dd}{\dd t}\up_{t} = Y_{t}\circ\up_{t}$. Differentiating (\ref{moser family}) gives
\begin{align*}
	- \dd\left(\frac{\dd\af_{t}}{\dd t}\right) &= \up_{t}^{*}\left\{\frac{\dd\ld^{t}}{\dd t} + \iota_{Y_{t}}(\dd\ld^{t}) + \dd(\iota_{Y_{t}}\ld^{t})\right\}\\
	&= \up_{t}^{*}\left\{-(\ld_{L} - \dd\wtd{\alpha}_{L}) + \wtd{\Phi}_{L}^{*}\ld_{0} + \iota_{Y_{t}}\left[(1-t)(-\om_{L}) - t\wtd{\Phi}_{L}^{*}\om_{0}\right] + \dd(\iota_{Y_{t}}\ld^{t})\right\}\\
	&= \up_{t}^{*}\left\{\iota_{Y_{t}}\left[(1-t)(-\om_{L}) - t\wtd{\Phi}_{L}^{*}\om_{0}\right] - \left[\ld_{L} - \dd\wtd{\af}_{L} - \wtd{\Phi}_{L}^{*}\ld_{0}\right] \right\} + \dd\left(\up_{t}^{*}\iota_{Y_{t}}\lambda^{t}\right) ~,
\end{align*}
where $\om_{L} = -\dd\ld_{L}$ is the canonical symplectic form on $T^{*}L$. By shrinking $\widetilde{U}_{L}$ in the fiber direction if necessary,  the $2$-form  $(1-t)(-\om_{L}) - t\wtd{\Phi}_{L}^{*}\om_{0}$ is non-degenerate for every $t\in [0, 1]$.  Define the one-parameter family of vector field $\{Y_t\}_{t\in[0,1]}$ by
\begin{align*}
	\iota_{Y_{t}}\left[(1-t)(-\om_{L}) - t\wtd{\Phi}_{L}^{*}\om_{0}\right] = \ld_{L} - \dd\wtd{\af}_{L} - \wtd{\Phi}_{L}^{*}\ld_{0} ~.
\end{align*}
Due to \eqref{ld_LK}, the right hand side vanishes on $\pi^{-1}(L\setminus K)$.  Thus, $Y_t$ only supports on $\pi^{-1}(K)$.

Note that the zero section is Lagrangian with respect to $(1-t)(-\om_{L}) - t\wtd{\Phi}_{L}^{*}\om_{0}$, and for every $V$ tangent to the zero section,
\begin{align*}
	\left[(1-t)(-\om_{L}) - t\wtd{\Phi}_{L}^{*}\om_{0}\right](\left.Y_{t}\right|_{\ul{0}}, V) = - \dd\left[\left.\wtd{\af}\right|_{\ul{0}} - \bt_{L}\right](V) = 0 .
\end{align*}
It follows that $Y_{t}$ is tangent to the zero section. Therefore, for the diffeomorphism $\up_{t}$ generated by $Y_{t}$, one has $\left.\up_{t}\right|_{\ul{0}}\in\mbox{Diff}_{c}(L)$ and 
\begin{align*}
	\frac{\dd}{\dd t}\up_{t}^{*}\ld^{t} = - \dd\left(\frac{\dd\af_{t}}{\dd t}\right) = \dd\left(\up_{t}^{*}\iota_{Y_{t}}\ld^{t}\right).
\end{align*}
Integrating it against with $t$ gives
\begin{align*}
	\up_{t}^{*}\ld^{t} = \ld_{L} - \dd\wtd{\af}_{L} + \dd\left(\int_{0}^{t}\up_{\tau}^{*}\iota_{Y_{\tau}}\ld^{\tau}\,\dd\tau\right) .
\end{align*}
Hence, $\up_{t}$ is the desired diffeomorphism, and $\af_t = -\int_{0}^{t}\up_{\tau}^{*}\iota_{Y_{\tau}}\ld^{\tau}\,\dd\tau$. 

Finally, set $\Phi_{L}$ to be $\wtd{\Phi}_{L}\circ\up_{1}$. %and choose a suitable open subset $U_{L}\subseteq\widetilde{U}_{L}$ so that the above argument works, we have
It follows that
\begin{align*}
	\Phi_{L}^{*}\ld_{0} = \up_{1}^{*}\wtd{\Phi}_{L}^{*}\ld_{0} = \up_{1}^{*}\ld^{t=1} = \ld_{L} - \dd\af_{L} 
\end{align*}
where $\af_{L} = \wtd{\af}_{L} + \af_{t=1}$.
\smallskip

%%%%
\noindent{\it Step 4: asymptotic behavior}.
It remains to verify the decay rate of $\af_L$.  Note that $\af_{t}$ only supports on $\pi^{-1}(K)$.  By construction, we have
\begin{align}
	\left.\af_{L}\right|_{\pi^{-1}(L_{a})} = \left(\frac{rs}{2} + \frac{r\fe_2}{2} - \ptl\right)\circ\vph_\srp^{-1} + c_a ~. \label{eq-alphaasymptotes}
\end{align}
Thus,
\begin{align*}
	\left|(\af_{L}\circ\vph_{\srp})(\sm, r, \vsm, s) - \left(\frac{rs}{2} + c_{a}\right)\right| = O\left(r|\fe_2| + |\ptl|\right) = O(r^{\gamma}) \quad\text{as } r\to\infty
\end{align*}
on each $\vph_\srp^{-1}(T^*L_a)\subset T^*(\Sm\times(R_1,\infty))$.  This completes the proof of this theorem.
\end{proof}

\begin{remark}\hfill
\begin{itemize}
    \item In Theorem \ref{AC Lag nbhd}, the function $\bt_L$ is harmonic if $L$ is a special Lagrangian (calibrated by $\re\Om_0$).  See \cite{Neves2007}*{Lemma 6.2}.
    \item By \cite{LotayNeves2013}*{Lemma 5.4}, one finds that $\af_L$ can be expressed as follows
    \begin{align}
        \af_L &= \bt_L\circ\pi + \frac{1}{2}\int_0^1 \langle{\bx},{\bar{\nabla}\bar{u}}\rangle_{\Phi_L\circ s\dd u} \,\dd s ~.
    \end{align}
    Here, $\bx$ is the position vector in $\BC^m$, $\bar{\nabla}$ is taken with respect to the standard structure of $\BC^m$, and $\bar{u}$ is a function on $\Phi_L(U_L)$ defined to be $u\circ\pi \circ \Phi_L^{-1}$.
\end{itemize}
\end{remark}

We will also consider the dilation of an asymptotically conical Lagrangian submanifold $L$ by a scale $\vep>0$, which will be denoted by $\iota_{\vep L} = \vep\cdot\iota_{L}:L\to\BC^{m}$.  It is clear that $\vep L$ is asymptotically conical with the same cone and the same rate.  The following corollary describes the effect of dilation on Theorem \ref{AC Lag nbhd}.

%%%%%%%%
\begin{corollary}\label{Lag nbhd scaling}
Let $\Phi_L: U_{L}\subset T^{*}L \to \BC^m$ be the Lagrangian neighborhood constructed by Theorem \ref{AC Lag nbhd}.  For any $\varepsilon>0$, let $f_{\varepsilon}:T^{*}L\to T^{*}L$ be the diffeomorphism defined by $f_{\vep}(q, p) = (q, \vep^{-2}p)$.  Then, the open neighborhood $U_{\vep L}:= f_{\vep}^{-1}(U_{L})\subset T^{*}L$ of the zero section and the embedding
\begin{align*}
	\Phi_{\vep L} &= \vep\cdot\Phi_{L}\circ f_{\vep}: U_{\vep L} \to \BC^{m}
\end{align*}
constitute a Lagrangian neighborhood of $\iota_{\vep L}$, and
\begin{align*}
	\Phi_{\vep L}^{*}\ld_{0} = \ld_{L} -  \dd\left(\vep^{2}\cdot(\af_{L}\circ f_\vep)\right) ~\text{ on } U_{\vep L},
\end{align*}
where $\af_{L}$ is the function given by Theorem \ref{AC Lag nbhd}.
\end{corollary}
%%%%
\begin{proof}
By $f_{\vep}^{*}\ld_{L} = \vep^{-2}\ld_{L}$, $(\varepsilon\cdot\,)^{*}\ld_{0} = \vep^{2}\ld_{0}$ and Theorem \ref{AC Lag nbhd},
\begin{align*}
	\Phi_{\vep L}^{*}\ld_{0} = f_{\vep}^{*}\Phi_{L}^{*}(\vep^{2}\ld_{0}) = \vep^{2}f_{\vep}^{*}\left(\ld_{L} - \dd\af_{L}\right) = \ld_{L} -  \dd\left(\vep^{2}f_{\vep}^{*}\af_{L}\right) ~.
\end{align*}
This finishes the proof of the corollary.
\end{proof}

The notations $f_\vep$ and $f_{\fe}$ defined by \eqref{1form_trans} both denote self-diffeomorphisms of the cotangent bundle, whose restriction on the fibers are affine transformations.  When the subscript is a real number, it is a fiberwise dilation; when the subscript is a $1$-form, it is a fiberwise translation.

\begin{remark} \label{dilate_potential}
Denoting $\mathfrak{r}(r) := \vep r$, by \eqref{def_kp} and \eqref{cone scaling} $\vep L\setminus\vep K$ is given by
\begin{align*}
	\vep\cdot\Phi_{\coneC}(\sm, r, \fe_1(\sm,r), \fe_2(\sm, r)) &= \Phi_{\coneC}(\sm, \vep r, \vep^2\fe_1\,(\sm, r), \vep\,\fe_2(\sm,r)) \\
	&= \Phi_{\coneC}(\sm, \mathfrak{r}, \vep^2\fe_1\,(\sm,\vep^{-1}\mathfrak{r}), \vep\,\fe_2(\sm,\vep^{-1}\mathfrak{r}))
\end{align*}
This means that $\vep L\setminus\vep K$ is the graph of a closed $1$-form on $\Sm\times(\vep R_1,\infty)$.  Its potential function (for $\mathfrak{r}>\vep R_1$) is given by
\begin{align*}
	-\int_{r_\vep}^\infty \vep\,\kp_2(\sm,\vep^{-1}y)\,\dd y &= -\vep^{2}\int_{\vep^{-1}\mathfrak{r}}\kp_2(\sm,y')\dd y' = \vep^2\,\ptl(\sm,\vep^{-1}\mathfrak{r}) ~,
\end{align*}
so that
\begin{equation}
\vep \cdot \varphi \, = \, \Phi_{\coneC} \circ d(\vep^2 \ptl(\sigma, \vep^{-1} r)) \circ \mathfrak{r} \label{epsilon-immersion-identity}
\end{equation}
The above identity may be extended to a similar identity on the cotangent bundle:
\begin{align}
\varepsilon\cdot \Phi_L \circ \varphi_\srp \, = \, \Phi_{\coneC} \circ f_{d(\vep^2 \ptl(\sigma, \vep^{-1} r))} \circ f_{\vep}^{-1} \circ \mathfrak{r}_\srp, \label{epsilon-cotangent-identity}
\end{align}
where $f_{\varepsilon},\, f_{dA}:T^{*}L\to T^{*}L$ are the diffeomorphisms defined by $f_{\vep}(q, p) = (q, \vep^{-2}p)$ and $f_{dA}(q,p) := (q, p + dA)$ respectively. 
\end{remark}

When an asymptotically conical Lagrangian $L$ is also a special Lagrangian, Joyce in \cite{Joyce2004SLCS3}*{section 4.1} defines two cohomological invariants.  We will require one of them.

%%%%%%%%
\begin{definition} \label{Z_inv}
Let $L$ be an asymptotically conical, special Lagrangian submanifold with cone $\coneC = \Sm\times(0,\infty)$.  It follows from $\im\Om_0|_L = 0$ that $(\im\Om_0,0)$ defines an element in the relative de Rham cohomology $\RH^m(\BC^m,L;\,\BR)$.  Since $\Sm$ is in effect the boundary of $L$, there is a natural map $\RH^{m-1}(L;\,\BR) \to \RH^{m-1}(\Sm;\,\BR)$.  Together with the long exact sequence
\begin{align*}
	0 = \RH^{m-1}(\BC^m;\,\BR) \to \RH^{m-1}(L;\,\BR) \stackrel{\cong}{\longrightarrow} \RH^m(\BC^m,L;\,\BR) \to \RH^{m}(\BC^m;\,\BR) = 0 ~,
\end{align*}
the invariant $Z(L) \in \RH^{m-1}(\Sm;\,\BR)$ is defined to be the image of $[(\im\Om_0,0)] \in \RH^m(\BC^m,L;\,\BR)\cong \RH^{m-1}(L;\,\BR)$ under the map $\RH^{m-1}(L;\,\BR) \to \RH^{m-1}(\Sm;\,\BR)$.
\end{definition}

%%%%%%%%%%%%%%%%%%%%%%%%%%%%%%%%
\subsection{The Lawlor Neck}\label{sec-lawlorneck}

It is known that $\RSU(m)$ acts transitively on the space of special Lagrangian $m$-planes in $\BC^m$.  In fact, the Grassmannian of oriented special Lagrangians is $\RSU(m)/\RSO(m)$.  It follows that up to an $\RSU(m)$ transformation, one may assume that a special Lagrangian $m$-plane is $\BR^m\subset\BC^m$.  It turns out that given a pair of special Lagrangian $m$-planes, one may still put them into a \emph{standard} form by $\RSU(m)$.

%%%%%%%%
\begin{lemma}[\cite{Joyce2003SLCS5}*{Proposition 9.1}]\label{two_planes}
Let $(\Pi^-,\Pi^+)$ be a pair of transverse special Lagrangian $m$-planes in $\BC^m$, namely, $\Pi^-\cap\Pi^+ = \{0\}$.  There exist $\bfU\in\RSU(m)$ and $0< \phi_1\leq \ldots\leq \phi_m< \pi$ such that $\bfU(\Pi^-) = \Pi^0$ and $\bfU(\Pi^+) = \Pi^{\bphi}$, where
\begin{align}
	\Pi^0 = \{(x_1,\ldots,x_m):~ x_j\in\BR^m\} \quad\text{and}\quad \Pi^{\bphi} = \{(e^{i\phi_1}x_1,\ldots,e^{i\phi_m}x_m):~ x_j\in\BR^m\} ~. \label{standard_planes} \end{align}
Moreover, $\bphi = (\phi_1,\ldots,\phi_m)$ is unique, and $\sum_{j=1}^{m}\phi_j = k\pi$ for some $k\in\{1,\ldots,m-1\}$.
\end{lemma}
%%%%%%%%
\begin{definition} \label{int_type}
For a pair of transverse special Lagrangian $m$-planes in $\BC^m$, $(\Pi^-,\Pi^+)$, the integer $k$ given by Lemma \ref{two_planes} is called the type of $(\Pi^-,\Pi^+)$.  Note that $(\Pi^+,\Pi^-)$ is of type $m-k$.
\end{definition}

Clearly, $(\Pi^-\cup\Pi^+)\setminus\{0\}$ is a Lagrangian cone, whose link is the disjoint union of two $S^{m-1}$'s.  When $(\Pi^-,\Pi^+)$ is of type $1$, there are \emph{special} Lagrangians asymptotic to $\Pi^-\cup\Pi^+$.  They are constructed by Lawlor in \cite{Lawlor1989}, and are usually referred as \emph{Lawlor necks}.  The explanation below is based on \cite{Joyce2003SLCS5}*{Example 6.11} and \cite{Lee2003}*{section 1}.

For positive numbers $a_1,\ldots,a_m$, introduce the functions
\begin{align*}
	{\mathrm P}_{\ba}(y) &= \frac{-1 + \prod_{j=1}^m(1+a_jy^2)}{y^2} \quad\text{and} \\
	z_j(s) &= \exp\left(i\,a_j\int_{-\infty}^s \frac{\dd y}{(1+a_jy^2)\sqrt{{\mathrm P}_{\ba}(y)}} \right) \sqrt{a_j^{-1}+s^2}
\end{align*}
for $j\in\{1,\ldots,m\}$.  Define the real numbers $\phi_1,\ldots,\phi_m$ and $A$ by
\begin{align}
	\phi_j = a_j\int_{-\infty}^{\infty} \frac{\dd y}{(1+a_jy^2)\sqrt{{\mathrm P}_{\ba}(y)}}
	\quad\text{and}\quad
	A = \om_m\left(\prod_{j=1}^m a_j\right)^{-\frac{1}{2}} ~,
\label{Lawlor_A} \end{align}
where $\om_m$ is the volume of the unit $S^{m-1}\subset\BR^m$.  With these functions and constants, the construction and the properties of Lawlor necks are summarised in the following proposition.  The proof can be found in the aforementioned references.
%%%%%%%%
\begin{proposition}\label{Lawlor}
For positive numbers $a_1,\ldots,a_m$, the followings hold true.
\begin{enumerate}
\item The numbers defined by \eqref{Lawlor_A} satisfy
\begin{align}
	\phi_j\in(0,\pi) \quad\text{for all } j ~,\quad  \sum_{j=1}^m\phi_j = \pi ~,\quad\text{and }~ A>0 ~.
\label{Lawlor_condition}\end{align}
Moreover, \eqref{Lawlor_A} gives a \emph{bijection} between
\begin{align*}
	\{(a_1,\ldots,a_m):~ a_j>0\text{ for all }j\} \quad\text{and}\quad
	\{(\phi_1,\ldots,\phi_m,A) \text{ obeying \eqref{Lawlor_condition}}\} ~.
\end{align*}
\end{enumerate}
With this understood, denote $(\phi_1,\ldots,\phi_m)$ by $\bphi$.  For $(\bphi,A)$ obeying \eqref{Lawlor_condition}, let
\begin{align}
	L^{\bphi,A} &= \left\{ \left(z_1(s)x_1, \ldots, z_m(s)x_m\right)\in\BC^m :~
	s\in\BR ,~ x_j\in\BR,~ \sum_{j=1}^mx_j^2=1 \right\} ~,
\label{Lawlor_eqn}\end{align}
They are called Lawlor necks.
\begin{enumerate}
\setcounter{enumi}{1}
\item\label{Lawlor2} The Lawlor neck $L^{\bphi,A}$ defined by \eqref{Lawlor_eqn} is an embedded, special Lagrangian submanifold in $\BC^m$.  It is diffeomorphic to $S^{m-1}\times\BR$, and is thus an exact Lagrangian.  It is asymptotically conical to $\Pi^0\cup\Pi^{\bphi}$ with rate $\gamma = {2-m}$, where
\begin{align*}
	\Pi^0 = \{(x_1,\ldots,x_m):~ x_j\in\BR^m\} \quad\text{and}\quad \Pi^{\bphi} = \{(e^{i\phi_1}x_1,\ldots,e^{i\phi_m}x_m):~ x_j\in\BR^m\} ~.
\end{align*}
\item\label{Lawlor3} The number $A$ is essentially the volume of the topological $B^m$ bound by the $S^{m-1}$ defined by $s=0$.  The dilation of a Lawlor neck is still a Lawlor neck.  Specifically, $\vep\cdot L^{\bphi,A} = L^{\bphi,\vep^m A}$ for any $\vep>0$.
\end{enumerate}
\end{proposition}

In item (\ref{Lawlor3}), one may also describe the dilation effect on the data $(a_1,\ldots,a_m)$; $\vep\cdot L^{\bphi,A}$ corresponds to $\vep\cdot(a_1,\ldots,a_m) = (\vep^{-2}a_1,\ldots,\vep^{-2}a_m)$.

Because of item (\ref{Lawlor2}), Theorem \ref{AC Lag nbhd} applies to the Lawlor necks.  We would like to determine the constants $c_a$'s described in that theorem.  The Lawlor neck $L^{\bphi,A}$ has two ends.  One is asymptotic to $\Pi^0$, whose constant is denoted by $c_-(L^{\bphi,A})$.  The other is asymptotic to $\Pi^{\bphi}$, whose constant is denoted by $c_+(L^{\bphi,A})$.  According to step 2 of the proof of Theorem \ref{AC Lag nbhd}, these constants are the limit of a primitive of $-\ld_0$.  A direct computation shows that
$\left.-\ld_0\right|_{L^{\bphi,A}} = \frac{1}{2\sqrt{{\mathrm P}_{\ba}(y)}}\dd y $,
and hence
\begin{align}
	c_+(L^{\bphi,A}) - c_-(L^{\bphi,A}) &= \int_{-\infty}^{\infty} \frac{1}{2\sqrt{{\mathrm P}_{\ba}(y)}}\dd y ~.
\label{Lawlor_const} \end{align}
By a change of variable, $c_+(\vep\cdot L^{\bphi,A}) - c_-(\vep\cdot L^{\bphi,A}) = \vep^2\left[c_+(L^{\bphi,A}) - c_-(L^{\bphi,A})\right]$.  This coincides with Corollary \ref{Lag nbhd scaling}.

Since $\alpha_{L^{\phi,A}}$ is unique up to the addition of a constant, we may choose the asymptotic constant $c_-(L^{\phi,A}) = 0$, from which it follows that $c_+(L^{\phi,A}) = \int_{-\infty}^{\infty} \frac{1}{2\sqrt{{\mathrm P}_{\ba}(y)}}\dd y $. We maintain this choice for the remainder of our work.

The $Z$-invariant (see Definition \ref{Z_inv}) of the Lawlor necks is computed by Joyce in \cite{Joyce2003SLCS5}*{section 9.1}:
%%%%%%%%
\begin{lemma} \label{Lawlor_Z}
For a Lawlor neck $L^{\bphi,A}$, let $\Sm^-$ be the link of $\Pi^0$, and $\Sm^+$ be the link of $\Pi^{\bphi}$.  The $Z$-invariant of the Lawlor neck satisfies
\begin{align*}
Z(L^{\bphi,A})\cdot[\Sm^-] = A \quad\text{and}\quad Z(L^{\bphi,A})\cdot[\Sm^+] = -A ~,
\end{align*}
where the notation means the evaluation on the fundamental cycles.
\end{lemma}

It is not hard to see the effect of the dilations: $Z(\vep\cdot L^{\bphi,A})\cdot[\Sm^\mp] = \pm\vep^m A$.

%%%%%%%%
\begin{remark}
The notation $\pm$ here plays the role of the index $a$ in Theorem \ref{AC Lag nbhd}.  The choice here matches with the $s$-parameter in \eqref{Lawlor_eqn}, but is opposite to that in \cite{Joyce2003SLCS5}*{section 9.1}.  Note that this does not mean that we reverse the orientation of $L^{\bphi,A}$; the orientation of a special Lagrangian is always given by $\re\Om$.
\end{remark}

\section{Desingularisations of Special Lagrangians with Transverse Self-Intersections} \label{sec-desingularisation}

Given a special Lagrangian with isolated conical singularities in a Calabi-Yau manifold $\iota:X \to M$, such that there exist suitable local models for the desingularisations, Joyce in \cite{Joyce2004SLCS3}*{Definition 6.2} shows how to construct desingularisations of the special Lagrangian.  The main purpose of this section is to give an exposition of Joyce's construction, in the particular case of Lagrangians with transverse self-intersections of type $1$. We make this restriction so that there exist suitable local models for the desingularisation process (the Lawlor necks of section \ref{sec-lawlorneck}). 

%For the remainder of this work, we will focus on a particular case of the preceding theory - flat special Lagrangians in complex tori. 
From now on, we will focus on the setting of the Main Theorem, depicted in Figure \ref{fig-torus}.

\begin{assumption} \label{assumption torus}
    Let $X = X_1\cup X_2$ be the the disjoint union of two $m$-tori, and $\iota: X\to T^{2m}$ be a special Lagrangian immersion such that $\iota: X_b\to T^{2m}$ for $b = 1,2$, and $\iota: X\to T^{2m}$ has only one transverse self-intersection point of type $1$.

    To be more precise, there exist $\xsin^-\in X_1$ and $\xsin^+\in X_2$ such that $\iota(\xsin^-) = \iota(\xsin^+) = \xsin$, and $( \iota_*(T_{\xsin^-}X_1),\iota_*(T_{\xsin^+}X_2) )$ is of type $1$ as defined in Definition \ref{int_type}.
    %The Calabi--Yau manifold $M$ and immersed special Lagrangian $\iota: X \to M$ take the following form:
    %\begin{itemize}
        %\item $M$ is a complex torus, i.e. $M := \BC^m/\Gm$ for a lattice $\Gm$ (where $m \geq 3$), and the Calabi--Yau structure is induced from the standard one on $\BC^m$.
        %\item The underlying manifold $X$ of the special Lagrangian immersion is the disjoint union of two $m$-tori, $X = X_1\cup X_2$, and $\iota:X_b\to M$ is a special Lagrangian embedding for $b=1,2$. The map $\iota:X\to M$ has \emph{only one} transverse self-intersection point of type 1 (as defined in Definition \ref{self-int}), which is denoted by $\xsin\in M$, and it may assumed that $\iota^{-1}(\xsin) = \{\xsin^-,\xsin^+\}$, where $\xsin^-\in X_1$ and $\xsin^+\in X_2$.
    %\end{itemize}   
\end{assumption}

\begin{figure}
    \centering
    \includegraphics[scale=0.3]{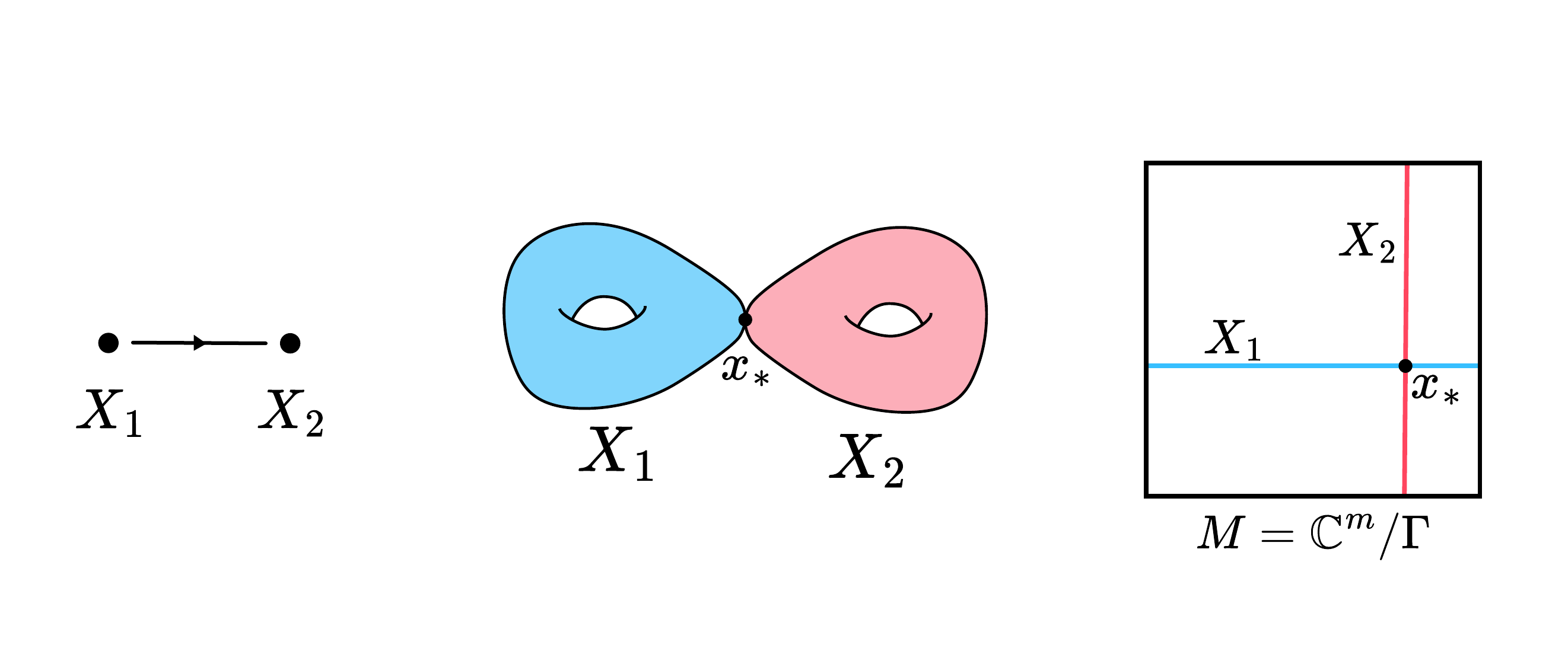}
\caption{Three figures of two Lagrangian tori $X_1, X_2$ inside a complex torus $M$ with a type 1 transverse intersection. In order, the figures depict: a schematic (directed) graph of their intersection type (c.f. \cite{Joyce2003SLCS5}*{Sec. 9.2}), a topological representation, a diagrammatic representation of the containment $X_1 \cup X_2 \subset M$.}
    \label{fig-torus}
\end{figure}

\begin{remark}
    It is fairly easy to construct examples.  Consider the planes $\Pi^0$ and $\Pi^{\bphi}$ as in \eqref{standard_planes} with $k=1$.  Choose bases for $\Pi^0$ and $\Pi^{\bphi}$, and set the lattice to be generated by them.
\end{remark}

We will modify the construction in \cite{Joyce2004SLCS3} to allow for diffeomorphisms between the desingularisations of different sizes of necks; this will allow us to set up and solve the Lagrangian mean curvature flow equation. In particular, we construct a family of embeddings $\iota^\vep$ from a fixed manifold $\stm$ to $T^{2m}$, satisfying $\iota^\vep \to \iota$ as $\vep \to 0$.

%%%%%%%%%%%%%%%%%%%%%%%%%%%%%%%%
\subsection{The Static Manifold}

We first construct the underlying topological manifold for the desingularisations.

%%%%%%%%
\begin{definition} \label{static_mfd}
Under Assumption \ref{assumption torus}, the domain manifold $\stm$ for the Lagrangian mean curvature flow is constructed as follows.
%Let $X$ be a compact manifold.  Let $\iota:X\to M$ be a special Lagrangian immersion with only transverse self-intersection points of type $1$, $\{x_1,\ldots,x_n\}$.  Denote $\iota^{-1}(x_j)$ by $x_j^-$ and $x_j^+$ for $j\in\{1,\ldots,n\}$.
There are three positive numbers in the construction, $\hbar$, $R_1$ and $R_2$, with $(1+2\hbar)R_1 \leq (1-\hbar)R_2$.  The number $\hbar$ is no greater than $1/100$, and plays no significant role.  The radii $R_1$ and $R_2$ may have to be taken smaller in each step if necessary.

%%%%
{\it Step 1}.
By Lemma \ref{two_planes}, there exists a standard neighborhood $\Up: B_{R_2}\subset\BC^{m}\to M = T^{2m}$ of $\xsin$ such that $( \iota_*(T_{\xsin^-}X_1),\iota_*(T_{\xsin^+}X_2) )$ corresponds to $(\Pi^0, \Pi^{\bphi})$ under $\Up_*$, where $\bphi = (\phi_1,\ldots\,\phi_m)$ is produced by that lemma.
%For each self-intersection point $x_j$, apply Lemma \ref{Darboux_chart} to find a Darboux chart $\Up_j: B_{R_2}\to M$.  Denote by $\bphi_j$ the output of (\ref{Dc2}) of that lemma.  We may assume $\iota^{-1}(\bigcup_{j=1}^n\Up_j(B_{R_2}))$ is the disjoint union of $2n$ topological $m$-dimensional balls.

%%%%
{\it Step 2}. Let $\coneC$ be the special Lagrangian cone $\Pi^0\cup\Pi^{\bphi}$, and let $\Sm = \coneC\cap S^{2m-1}$ be its link.  This link is the union of two $S^{m-1}$'s, which we label $\Sigma^-$ and $\Sigma^+$ respectively.  Apply Proposition \ref{LN of cones} to find the Lagrangian neighborhood $\Phi_{\coneC}:U_{\coneC}\subset T^*(\Sm_j\times(0,\infty))\to\BC^m$.

%%%%
%{\it Step 3}. Due to Lemma \ref{CS_rate}, there exists a function $\ptx_j: \Sm_j\times(0,R_2)\to\BR$ such that
%\begin{align*}
%	\Up_j^{-1}(\iota(X)\setminus\{x_j\}) = \{(\Phi_{C_j}\circ\dd\ptx_j)(\sm,r) :~\sm\in\Sm_j,~ 0<r<R_2\} ~,
%\end{align*}
%and $|\nabla^\ell\ptx_j| = O(r^{3-\ell})$ as $r\to0$ for $\ell\in\{0,1,2\}$.  Note that $\iota^{-1}\circ(\Up_j\circ\Phi_{C_j}\circ\dd\ptx_j)$ induces a diffeomorphism from $\Sm_j\times(0,R_2)$ to $\iota^{-1}(\Up_j(B_{R_2})\setminus\{x_j\})$.

%%%%
{\it Step 3}. By Proposition \ref{Lawlor}, Theorem \ref{AC Lag nbhd} and Remark \ref{dilate_potential}, there exists an $A>0$ such that the Lawlor neck $L = L^{\bphi,A}$ satisfies
\begin{align*}
	L\cap(\BC^m\setminus\overline{B_{R_1}}) &= \{(\Phi_{\coneC}\circ\dd\ptl)(\sm,r):~ \sm\in\Sm,~ r> R_1\}
\end{align*}
for some function $\ptl: \Sm\times(R_1,\infty)\to\BR$ with $|\nabla^\ell\ptl| = O(r^{(2-m)-\ell})$ as $r\to\infty$, for any $\ell\geq0$.

%%%%
{\it Step 4}. The \emph{static manifold} $\stm$ for the desingularisation of $X$ will be constructed from the following three types of pieces:
\begin{itemize}
\item $\out = X\setminus\iota^{-1}(\Up(\overline{B_{(1-\hbar)R_2}})) = X\setminus\overline{B_{(1-\hbar)R_2}(\xsin)}$, the \emph{outer region},
\item $Q = \Sm\times(R_1,R_2)$, the \emph{intermediate region}, consisting of the connected components $Q^- = \Sigma^- \times (R_1,R_2)$ and $Q^+ = \Sigma^+ \times (R_1,R_1)$,
\item $P = L\cap B_{(1+\hbar)R_1}$, the \emph{tip region}.
\end{itemize}
For $(\sm,r)\in\Sm\times((1-\hbar)R_2,R_2)\subset Q$, it is naturally identified with its image under $\iota^{-1}\circ\Up\circ\Phi_{\coneC}\circ\ul{0}$ in $\out$.  For $(\sm,r)\in\Sm\times(R_1,(1+\hbar)R_1)$, identify it with its image under $\Phi_{\coneC}\circ\dd\ptl$ in $P$.  The resulting manifold is the static manifold $\stm$, which is clearly a compact, smooth manifold.  Topologically, it is the connected sum of two $T^{m}$'s.
\end{definition}

%%%%%%%%%%%%%%%%%%%%%%%%%%%%%%%%
\subsection{Desingularisations}

In \cite{Joyce2004SLCS3}*{section 6.1}, Joyce constructs the desingularisations as a submanifold in $M$. Here, we instead construct an embedding $\iota^{\vep}:\stm \to M$.

%%%%%%%%
\begin{definition} \label{desing}
Under Assumption \ref{assumption torus}, let $\stm$ be the static manifold constructed by Definition \ref{static_mfd}.  Fix a $\tau\in(0,\frac{1}{2})$, whose precise value will be determined later.  Given $\vep$ with
\begin{align}
	0< \vep &< \min\left\{ 1, \left((1+\hbar)R_1\right)^{\frac{-1}{1-\tau}}, \left(\frac{(1-\hbar)R_2}{2}\right)^{\frac{1}{\tau}} \right\}  ~,
\label{ep_range} \end{align}
define a Lagrangian embedding $\iota^{\vep}: \stm\to M$ on the pieces of $\stm$ as follows.  Its image, $\iota^{\vep}(\stm)$, will be denoted by $N^{\vep}$.

%%%%
{\it Step 0}. Choose a smooth, increasing function $\chi:(0,\infty)\to \BR$ such that $\chi(y)\in[0,1]$ for all $y$, and
\begin{align*}
\chi(y) = \begin{cases} 0 & \text{when }0< y\leq 1 ~, \\ 1 & \text{when }2\leq y <\infty ~. \end{cases}
\end{align*}
 
 %%%%
{\it Step 1}. For any $q$ in the outer region $\out$, $\iota^{\bvep}$ is set to be the original immersion, $\iota^{\bvep}(q) = \iota(q)$.

 %%%%
{\it Step 2}. For any $q$ in the tip region $P$, set $\iota^{\vep}(q)$ to be $\Up(\vep q)$.  Namely, it is simply the dilation of the Lawlor neck by $\vep$ at $\xsin$.

%\begin{figure}[t]
%\centering
%\includegraphics[scale=0.35]{images/intermediate_graph-2.pdf}
%\caption{A diagram of the gluing procedure to construct the desingularisation $\iota^\vep:\stm \to M$ in the intermediate region, with the left side depicting the ball $B_{R_2} \subset \mathbb{C}^m$ and the right side depicting the cotangent bundle of $\Sigma_j \times [\varepsilon_j R_1, R_2]$ (for notational simplicity, $j$ is suppressed). The interpolation region for the function $\mathfrak{Q}_{\varepsilon_j}$ is shaded.}
%\label{fig-intermediate}
%\end{figure}

 %%%%
{\it Step 3}. The map on the intermediate region interpolates between the above two maps.%(refer to Figure \ref{fig-intermediate} for a diagram). The procedure is the same for all $j\in\{1,\ldots,n\}$.  For notational simplicity, suppress the subscript $j$ in $\vep_j$, $C_j$ etc.
\begin{align}
	\kp_{\vep}(r) &= \left[ 1 - \chi\left(\frac{r-R_1}{\hbar R_1}\right) \right]\vep\,r + \chi\left(\frac{r-R_1}{\hbar R_1}\right) r
\label{static_dilate} \end{align}
for $r\in(R_1,R_2)$.  It is a diffeomorphism from $(R_1,R_2)$ to $(\vep R_1, R_2)$, which will be verified momentarily.  Denote the diffeomorphism $\id_{\Sm}\times\kp_{\vep}: \Sm\times(R_1,R_2) \to \Sm\times(\vep R_1, R_2)$ by $\bar{\kp}_{\vep}$.
Next, for $(\sm, \ir)\in\Sm\times(\vep R_1,R_2)$, let
\begin{align}
	\ptq_{\vep} (\sm, \ir) &= (1-\chi(\vep^{-\tau}\ir))\,\vep^{2}\,\ptl(\sm,\vep^{-1}\ir) ~.
\label{int_potential} \end{align}
As noted in Remark \ref{dilate_potential}, $\vep^{2}\,\ptl(\sm,\vep^{-1}\ir)$ is the potential function of $\vep\,L^{\bphi,A}$.  It naturally gives a Lagrangian embedding:
\begin{align*}
	\Up\circ\Phi_{\coneC}\circ\dd\ptq_{\vep}: \Sm\times(\vep R_1,R_2) \to M ~.
\end{align*}
Finally, $\iota^{\vep}$ on $Q = \Sm\times(R_1,R_2)$ is set to be $\iota^{\vep} := \Up\circ\Phi_{\coneC}\circ\dd\ptq_{\vep}\circ\bar{\kp}_{\vep}$.
\end{definition}

%%%%%%%%
\begin{remark}\hfill
\begin{itemize}
\item We leave it for the readers to check that $N^{\bvep}$ is the same as the desingularisation constructed in \cite{Joyce2004SLCS3}*{section 6.1}. This allows us to invoke the estimates established in that paper.
\item For the intermediate region, $\Sm\times(\vep R_1,R_2)$ is more geometric.  To be precise, the coordinate ``$\ir$" is the (Euclidean) distance to the origin in the Darboux chart.  However, in order to take the ``time" derivative of a potential function, we must work on the time-independent region $\Sm\times(R_1,R_2)$.
\item Since $\iota(X)$ is an immersed special Lagrangian in $T^{2m}$, and the Lawlor necks are special Lagrangians in $\BC^m$, one can verify that $N^{\vep}$ is of zero-Maslov\footnote{The argument of the complex valued function $\frac{\Om|_{N^{\vep}}}{\dd V_{N^{\vep}}}$ is a well-defined function on $N^{\vep}$.} class.
\end{itemize}
\end{remark}

We now verify that $\kp_{\vep}$ gives a diffeomorphism from $(R_1,R_2)$ to $(\vep R_1,R_2)$, and $\iota^{\vep}$ is well-defined.  %We continue to suppress the subscript $j$.
%%%%%%%%
\begin{lemma} \label{static_dilate_property}
Suppose that $(1+2\hbar)R_1 \leq (1-\hbar)R_2$, then the function $\kp_{\vep}(r)$ defined by \eqref{static_dilate} is increasing for $r\in(R_1,R_2)$.  Indeed, $\frac{\dd}{\dd r}\kp_\vep(r) \geq \vep$.  There exists $c_\ell>0$ for all $\ell\in\BN$, depending on $\hbar$, $R_1$, $R_2$ and $\chi$, such that $\left|\frac{\dd^\ell}{\dd r^\ell}\kp_\vep(r)\right| \leq c_\ell$ for $r\in(R_1,R_2)$.  Moreover, $\kp_\vep(r) = \vep r$ when $R_1<r<(1+\hbar)R_1$, and $\kp_\vep(r) = r$ when $(1-\hbar)R_2<r<R_2$.
\end{lemma}

%%%%
\begin{proof}
The derivative of $\kp_\vep$ is
\begin{align*}
	\frac{\dd}{\dd r}\kp_\vep(r) &=  \vep + (1-\vep)\,\chi\left(\frac{r-R_1}{\hbar R_1}\right)
	+ \frac{(1-\vep)r}{\hbar R_1}\chi'\left(\frac{r-R_1}{\hbar R_1}\right) ~,
\end{align*}
which is clearly no less than $\vep$, and is bounded from above.  It is not hard to see that the higher order derivatives of $\kp_\vep(r)$ are uniformly bounded on $(R_1,R_2)$.

When $r<(1+\hbar)R_1$, $\frac{r-R}{\hbar R_1} < 1$, and hence $\kp_\vep(r) = \vep r$.  It follows from $(1+2\hbar)R_1 \leq (1-\hbar)R_2$ that $\frac{r-R_1}{\hbar R_1} > 2$ when $r > (1-\hbar)R_2$.  Hence, $\kp_\vep(r) = r$ when $r > (1-\hbar)R_2$.
\end{proof}

%%%%%%%%
\begin{lemma} \label{wd_desing}
The map $\iota^{\bvep}$ introduced in Definition \ref{desing} is well-defined.
\end{lemma}

%%%%
\begin{proof}
It follows from \eqref{ep_range} that $(1+\hbar)R_1\vep < \vep^\tau < 2\vep^\tau < (1-\hbar)R_2$.

%%%%
{\it Intermediate-Tip region}. When $R_1<r<(1+\hbar)R_1$, it follows from Lemma \ref{static_dilate_property} that $\kp_\vep(r) = \vep r$, and $\ir = \kp_\vep(r) < (1+\hbar)R_1\vep$.  Thus, $\ptq_{\vep}(\sm,\ir) = \vep^{2}\,\ptl(\sm,\vep^{-1}\ir)$, and $(\ptq_{\vep}\circ\bar{\kp}_\vep) (\sm,r) = \vep^2\,\ptl(\sm, r)$.  Denote $\dd_\Sm\ptl$ by $\fe_1$, and $\frac{\pl}{\pl r}\ptl$ by $\fe_2$.  By \eqref{sharp0} and using the coordinate system introduced around \eqref{cone scaling},
\begin{align*}
	\left((\dd\ptq_{\vep})\circ\bar{\kp}_{\vep}\right)(\sm,r) &= \left((\bar{\kp}_\vep)_\srp\circ\dd(\ptq_{\vep}\circ\bar{\kp}_\vep)\right)(\sm,r) \\
	&= (\bar{\kp}_\vep)_\srp\left(\sm, r, \vep^2\fe_1(\sm,r), \vep^2\fe_2(\sm,r)\right) = (\sm, \vep r, \vep^2\fe_1(\sm,r), \vep\fe_2(\sm,r)) ~.
\end{align*}
This coincides with the right hand side of the first equation in Remark \eqref{dilate_potential}.  It follows that $\iota^{\vep}$ is well-defined in this region.

%%%%
{\it Intermediate-Outer region}.  This part is left as an exercise for the readers.%When $(1-\hbar)R_2<r<R_2$, it follows from Lemma \ref{static_dilate_property} that $\kp_\vep(r) = r$, and $\ir = \kp_\vep(r) > (1-\hbar)R_2$.  In other words, $\bar{\kp}_\vep$ is the identity map on this region.  One can also find that $\ptq_{\vep} (\sm, \ir) = \ptx(\sm,\ir)$.  Hence, $\iota^{\bvep}$ coincides with the original map $\iota$.
\end{proof}

\subsection{Weight Function} Later on, the equations on $N^{\vep}$ will be analysed on some weighted H\"older spaces, requiring the following weight function designed to capture the geometry of the self-intersection points and the Lawlor necks.

%%%%%%%%
\begin{definition} \label{weight}
Under Assumption \ref{assumption torus}, let $\stm$ be the static manifold constructed by Definition \ref{static_mfd}.  Given $\vep$ satisfying \eqref{ep_range}, let $\iota^{\vep}$ be the Lagrangian embedding constructed by Definition \ref{desing}.  Define a smooth function $\rho_{\vep}: \stm\to\BR_+$ as follows.
%For a special Lagrangian immersion, $\iota:X\to M$, with only transverse self-intersection points of type $1$, $\{x_j\}_{j=1}^n$, let $\stm$ be the static manifold constructed by Definition \ref{static_mfd}.  Given $\bvep = (\vep_1,\ldots,\vep_n)$ satisfying \eqref{ep_range}, let $\iota^{\bvep}$ be the Lagrangian embedding constructed by Definition \ref{desing}.  Define a smooth function $\rho_{\bvep}: \stm\to\BR_+$ as follows.
\begin{itemize}
\item {\it Tip region}. For the Lawlor neck $L$, choose a smooth function $\hat{\rho}: L\to[1,\infty)$ such that $\hat{\rho}(\bx)$ depends only on $|\bx|$, and $\hat{\rho}(\bx) = |\bx|$ when $|\bx| > R_1$.  For $\bx\in P = L\cap B_{(1+\hbar)R_1}$, $\rho_{\vep}(\bx)$ is defined to be $\vep\cdot\hat{\rho}(\bx)$.

\item {\it Intermediate region}. On $Q = \Sm\times(R_1,R_2)$,
\begin{align*}
\rho_{\vep}(\sm,r) &= \kp_{\vep}(r) + (R_2-r)\left[1-\chi\left(\frac{R_2-r}{\hbar R_2}\right)\right]
\end{align*}
where $\kp_{\vep}(r)$ is given by \eqref{static_dilate}.  Note that $\rho_{\vep}(\sm,r) = \kp_{\vep}(r)$ when $R_1<r\leq(1-2\hbar)R_2$.  It is not hard to see that $\rho_{\vep}(\sm,r) = R_2$ when $r\geq(1-\hbar) R_2$.  Also, note that when $(1-2\hbar)R_2\leq r<R_2$,  $\rho_{\vep}(\sm,r)$ is independent of $\vep$, and is increasing in $r$.

\item {\it Outer region}. On the outer region $\out$, extend the function from the intermediate region by the constant $R_2$.
\end{itemize}
\end{definition}

%%%%%%%%%%%%%%%%%%%%%%%%%%%%%%%%
\subsection{Lagrangian neighborhoods}

In \cite{Joyce2004SLCS3}*{section 6.3}, Joyce constructs the Lagrangian neighborhood $\Psi_{N^{\bvep}}: U_{N^{\bvep}}\subset T^*\stm \to M$ of $\iota^{\bvep}(\stm)$ as follows.

%%%%%%%%
\begin{definition} \label{desing_nbd}
Under Assumption \ref{assumption torus}, let $\stm$ be the static manifold constructed by Definition \ref{static_mfd}.  Given $\vep$ satisfying \eqref{ep_range}, let $\iota^{\bvep}$ be the Lagrangian embedding constructed by Definition \ref{desing}.  Define a Lagrangian neighborhood $\Psi_{N^{\vep}}: U_{N^{\vep}}\subset T^*\stm \to M$ for $\iota^{\vep}$ as follows.
%For a special Lagrangian immersion, $\iota:X\to M$, with only transverse self-intersection points of type $1$, $\{x_j\}_{j=1}^n$, let $\stm$ be the static manifold constructed by Definition \ref{static_mfd}.  Given $\bvep = (\vep_1,\ldots,\vep_n)$ satisfying \eqref{ep_range}, let $\iota^{\bvep}$ be the Lagrangian embedding constructed by Definition \ref{desing}.  Define a Lagrangian neighborhood $\Psi_{N^{\vep}}: U_{N^{\vep}}\subset T^*\stm \to M$ for $\iota^{\vep}$ as follows.

%Define an open neighborhood $U_{N^{\bvep}}$ of the zero section $\ul{0}$ in $T^*\stm$, and an embedding $\Psi_{N^{\bvep}}: U_{N^{\bvep}} \to M$ as follows.  They satisfy
%\begin{align*}
%\left.\Psi_{N^{\bvep}}\right|_{\ul{0}} = \iota^{\bvep} \quad\text{and}\quad (\Psi_{N^{\bvep}})^*(\om) = \om_{\stm} ~,
%\end{align*}
%where $\om_{\stm}$ is the canonical symplectic form on $T^*\stm$.

 %%%%
{\it Step 1: Tip region}.  Remember that the tip region $P$ is a subset of the Lawlor neck $L$.  Denote by $\pi$ the bundle projection of the cotangent bundle.  Apply Corollary \ref{Lag nbhd scaling} to the Lawlor neck $L$ and $\vep$ to find an open set $U_{\vep L}\subset T^*L$ and an embedding $\Phi_{\vep L}: U_{\vep L} \to \BC^m$.  Define
\begin{align*}
	U_{N^{\vep}}\cap \pi^{-1}(P) &= U_{\vep L}\cap\pi^{-1}(P) ~, \\
	\left.\Psi_{N^{\vep}}\right|_{U_{N^{\vep}}\cap \pi^{-1}(P)} &= \Up\circ\Phi_{\vep L} ~.
\end{align*}

 %%%%
{\it Step 2: Intermediate region}.  As in step 2 of Definition \ref{static_mfd}, apply Proposition \ref{LN of cones} to the cone $\coneC = \Sm\times(0,\infty)$ to find an $\BR_+$-invariant open set $U_{\coneC}\subset T^*\coneC$ and equivariant embedding $\Phi_{\coneC}:U_{\coneC}\to\BC^m$.  The map $\bar{\kp}_{\vep}$ given by step 3 of Definition \ref{desing} induces a diffeomorphism
\begin{align*}
	(\bar{\kp}_{\vep})_\srp: T^*Q = T^*(\Sm\times(R_1,R_2)) \to T^*(\Sm\times(\vep R_1,R_2)) ~.
\end{align*}
Similar to \eqref{1form_trans}, let $f_{\dd\ptq_{\vep}}(q,p) = \left(q, p + (\dd\ptq_{\vep})(q)\right)$ be the self-diffeomorphism\footnote{Or equivalently, $f_{\dd\ptq_\vep}$ sends $(\sm,\ir,\vsm,\is)$ to $\left(\sm, \ir, \vsm+(\dd_\Sm\ptq_\vep)(\sm,\ir), \is + \frac{\pl\ptq_{\vep}}{\pl \ir}(\sm, \ir)\right)$.}  of $T^*(\Sm\times(\vep R_1,R_2))$.
Define $U_{N^{\vep}}\cap \pi^{-1}(Q)$ to be
\begin{align*}
	\left( (\bar{\kp}_{\vep})_\srp \right)^{-1} \left\{ (\sm,\ir,\vsm,\is)\in T^*(\Sm\times(\vep_j R_1,R_2)):~ f_{\dd\ptq_{\vep}}((\sm,\ir),\vsm+\is\dd\ir)\in U_{\coneC} \right\} ~,
\end{align*}
and define the map to be
\begin{align*}
	\left.\Psi_{N^{\vep}}\right|_{U_{N^{\vep}}\cap \pi^{-1}(Q)} &= \Up\circ\Phi_{\coneC}\circ f_{\dd\ptq_{\vep}} \circ (\bar{\kp}_{\vep})_\srp ~.
\end{align*}

{\it Step 3: Outer region}.  With the help of Lemma \ref{static_dilate_property} and Lemma \ref{wd_desing}, $U_{N^{\vep}}$ and $\Psi_{N^{\vep}}$ are independent of $\vep$ on the overlap between the intermediate and outer region.  One use the same Moser's trick argument as that in step 3 in the proof of Theorem \ref{AC Lag nbhd} to extend $U_{N^{\vep}}$ and $\Psi_{N^{\vep}}$ over $\out$.  The extensions are also independent of $\vep$.
\end{definition}

We leave it for the readers to check the well-definedness of the open set and the embedding, or one may consult \cite{Joyce2004SLCS3}*{Definition 6.7}.

%%%%%%%%%%%%%%%%%%%%%%%%%%%%%%%%%%%%%%%%%%%%%%%%%%%%%%%%%%%%%%%%
\section{The Lagrangian Mean Curvature Flow Equation}\label{sec-lmcfequation}

Under Assumption \ref{assumption torus}, our goal for the remainder of this work is to construct $u:\stm\times[\Lambda,\infty)\to\BR$, and $\vep(t)$ such that $\dd u\in U_{N^{\vep(t)}}$ for all $t$, and $\Psi_{N^{\vep(t)}}\circ\dd u$ is the solution to the mean curvature flow (where the notation $\dd u$ denotes the spatial exterior derivative at time $t$). The equation reads
\begin{align}
	\left( \frac{\pl}{\pl t}\Psi_{N^{\vep(t)}}\circ\dd u \right)^\perp &= H(t) ~,
\label{MCF0} \end{align}
where $H(t)$ is the mean curvature vector of $\left(\Psi_{N^{\vep(t)}}\circ\dd u\right)(\stm)$, and $\perp$ denotes the orthogonal projection onto its normal bundle. In this section, we take our first step towards this goal, by rewriting (\ref{MCF0}) as a differential equation involving the potential function $u$. % In particular, with a suitable assumption on the topology of $\iota:X \to M$ (Proposition \ref{prop-tree-case}), we are able to rewrite (\ref{MCF0}) as (\ref{MCF3}).

We will require the following basic facts about the geometry of Lagrangian submanifolds.  Suppose that $F: L\to M = T^{2m}$ is a Lagrangian immersion, and denote by $T^\perp L$ the normal bundle of $F(L)$.  Then there is a bundle isomorphism
\begin{align*}
	\begin{array}{ccl}
	T^\perp L &\to &T^*L\\
	\eta &\mapsto &F^*\left(\om(\eta,\,\cdot\,) \right) ~.
	\end{array}
\end{align*}
In particular, suppose that $Y$ is a section of $F^*(TM)$, then
\begin{align}
	F^*\left(\om(Y^\perp, \,\cdot\,)\right) = F^*\left(\om(Y-Y^\top\,\cdot\,)\right) = F^*\left(\om(Y, \,\cdot\,)\right) ~,
\label{normal_cot} \end{align}
using the fact that $F^*\om$ vanishes. Furthermore, one has the (multi-valued) function $\ta_L = \arg(\frac{F^*(\Om)}{\dd V_L})$, the \emph{Lagrangian angle}, whose exterior derivative is a well-defined $1$-form on $L$. According to \cite{HL1982}*{section III.2.D}, $\dd\ta_L$ is the image of the mean curvature vector of $L$ under the above isomorphism,
\begin{align}
	F^*\left(\om(H_L,\,\cdot\,)\right) &= -\dd\ta_L ~.
\label{grad_ta} \end{align}

%%%%%%%%%%%%%%%%%%%%%%%%%%%%%%%%
\subsection{The Equation} \label{sec-theequation}

%The main purpose of this subsection is to rewrite \eqref{MCF0} as a differential equation in the exterior derivative of the potential $u:\stm\times[\Lambda,\infty)\to\BR$.
Denote by $F_t$ the embedding $\Psi_{N^{\vep(t)}}\circ\dd u$, and by $\ta(\dd u)$ the Lagrangian angle of $F_t: \stm\to M$.  Since $N^{\vep(t)}$ is of zero-Maslov class, we may choose $\ta(\dd u)$ to be a single-valued function.  By \eqref{grad_ta} and \eqref{normal_cot}, 
\eqref{MCF0} reads
\begin{align}
	\dd \left[ \ta(\dd u) \right] &=  -F_t^*\left(\om\Big(\frac{\pl F_t}{\pl t}, \,\cdot\,\Big)\right) ~.
\label{MCF1} \end{align}
The right hand side will be computed on different pieces.

%%%%%%%%%%%%%%%%
\subsubsection{Outer Region}

On the outer region $\out$, $\Psi_{N^{\vep(t)}}$ is independent of $\vep(t)$.  In this case, the right hand side of  \eqref{MCF1} was computed by Behrndt in his thesis \cite{Behrndt2011}*{Lemma 4.11}.  The proof is included for completeness.

%%%%%%%%
\begin{lemma} \label{Behrndt lemma}
Let $\iota_L: L\to (M,g,J,\om)$ be a Lagrangian embedding in a K\"ahler manifold, and $\Psi_L: U_L\subset T^*L\to M$ be a Lagrangian neighborhood.  Then, given a one-parameter family of closed $1$-forms $\eta_t$ on $L$ whose image belongs to $U_L$,
\begin{align*}
	(\Psi_L\circ\eta_t)^* \left( \om\Big( \Big( \frac{\pl (\Psi_L\circ\eta_t)}{\pl t} \Big)^\perp, \,\cdot\,\Big) \right) &= - \frac{\pl\eta_t}{\pl t} ~.
\end{align*}
\end{lemma}

%%%%
\begin{proof}
We work on $U_L$ equipped with the induced K\"ahler triple $\left(\Psi_L^*(g), \Psi_L^*(J), \om_L\right)$.  Denote by $\td{F}_t: L\to U_L$ the embedding given by $\eta_t$, i.e. $\td{F}_t := \Psi_L \circ \eta_t$.  Since $\eta_t$ is closed, $\td{F}_t$ is a Lagrangian embedding.  By \eqref{normal_cot}, computing the left hand side is equivalent to computing $\td{F}_t^*(\om_L(\frac{\pl\td{F}_t}{\pl t}, \,\cdot\,) )$.

Choose a local coordinate system $\{q_i\}$ on $L$.  Let $\{p^i\}$ be the coordinate induced by $\{\dd q_i\}$ for the fibers of $T^*L$.  The canonical symplectic form is $\sum_i\dd q_i\w\dd p^i$.  Write $\eta_t$ as $\sum_i \eta_t^{\,i}(q)\dd q_i$.  In terms of the $(q,p)$ coordinate, $\td{F}_t(q_1,\ldots,q_m) = (q_1,\ldots,q_m,\eta_t^{\,1}(q),\ldots,\eta_t^{\,m}(q))$, and $\frac{\pl\td{F}_t}{\pl t} = \sum_i \frac{\pl\eta_t^{\,i}}{\pl t}\frac{\pl}{\pl p^i}$.  It follows that
\begin{align*}
	\td{F}_t^*\Big(\om_L\Big(\frac{\pl\td{F}_t}{\pl t}, \,\cdot\,\Big) \Big) =  -\td{F}_t^*\Big(\, \sum_i \frac{\pl\eta_t^{\,i}}{\pl t}\dd q_i\Big) = -\frac{\pl\eta_t}{\pl t} ~.
\end{align*}
It finishes the proof of this lemma.
\end{proof}

It follows that on the outer region, \eqref{MCF1} becomes
\begin{align}
	\dd\left[\frac{\pl u}{\pl t}\right] &= \dd\left[\ta(\dd u)\right] ~.
\label{MCF1_ext} \end{align}

%%%%%%%%%%%%%%%%
\subsubsection{Tip Region}

On the tip region $P = L\cap B_{(1+\hbar)R_1}$, the image of $\iota^{\vep}$ belongs to $\Up(B_{R_2})$.  Computing the right-hand side of (\ref{MCF1}) is equivalent to computing $\td{F}_t^*(\om_0(\frac{\pl\td{F}_t}{\pl t}, \,\cdot\,) )$ for
\begin{align*}
	\td{F}_t = \Phi_{\vep(t) L}\circ\dd u : P\subset L\to B_{R_2}\subset\BC^m ~,
\end{align*}
where $\Phi_{\vep(t) L}$ is defined in Corollary \ref{Lag nbhd scaling}.

By using the chain rule on $\td{F}_t(q) =  \vep(t)\cdot\Phi_{L}(q,\vep(t)^{-2}\dd u(q))$,
\begin{align}
	\frac{\pl\td{F}_t}{\pl t} &= \frac{\vep'(t)}{\vep(t)}\td{F}_t + \vep(t)\frac{\pl(\Phi_{L}\circ\eta_t)}{\pl t}
\label{MCF1_tip0} \end{align}
where $\eta_t = \vep(t)^{-2}\dd u$.

For the first term on the right hand side of \eqref{MCF1_tip0}, note that $\td{F}_t$ is the position vector, and thus $\td{F}_t^*(\om_0( \td{F}_t, \,\cdot\,)) = \td{F}_t^*(-2\ld_0)$.  According to Corollary \ref{Lag nbhd scaling},
\begin{align*}
	\td{F}_t^*\Big(\om_0\Big( \frac{\vep'(t)}{\vep(t)}\td{F}_t, \,\cdot\,\Big) \Big) &= -2\frac{\vep'(t)}{\vep(t)}\,(\dd u)^*\Phi_{\vep(t) L}^*(\ld_0) \\
	&= -2\frac{\vep'(t)}{\vep(t)}\cdot(\dd u)^*\left[ \ld_{L} -  \dd\left(\vep(t)^{2}\cdot(\af_{L}\circ f_{\vep(t)})\right) \right] \\
	&= -2\frac{\vep'(t)}{\vep(t)} \dd u + 2\vep(t)\vep'(t)\,\dd\left[ \af_{L}\circ f_{\vep(t)}\circ\dd u \right] ~.
\end{align*}
The map $f_{\vep(t)}\circ\dd u: L\to T^*L$ is $\vep(t)^{-2}\dd u$.

For the second term on the right hand side of \eqref{MCF1_tip0}, apply Lemma \ref{Behrndt lemma} to the map $\Phi_{L}$ and the family of $1$-forms $\eta_t$.  One finds that
\begin{align*}
	\td{F}_t^*\Big(\om_0\Big( \vep(t)\frac{\pl(\Phi_L\circ\eta_t)}{\pl t}, \,\cdot\,\Big) \Big) &= -\vep(t)^2\frac{\pl(\vep(t)^{-2}\dd u)}{\pl t} \\
	&= 2\frac{\vep'(t)}{\vep(t)} \dd u - \dd\left[\frac{\pl u}{\pl t}\right].
\end{align*}

To sum up, on the tip region $P_j$, \eqref{MCF1} becomes
\begin{align}
	\dd\left[\frac{\pl u}{\pl t}\right] &= \dd\left[\ta(\dd u) + (\vep(t)^2)'\cdot \af_{L_j}\circ(\vep(t)^{-2}\dd u) \right] ~.
\label{MCF1_tip} \end{align}

%%%%%%%%%%%%%%%%
\subsubsection{Intermediate Region}

On the intermediate region $Q = \Sm\times (R_1,R_2)$, the image belongs to $(\Up\circ\Phi_{\coneC})(\Sm\times(0,R_2))$, where $\coneC = \Sm\times(0,\infty) = \Pi^0 \cup \Pi^{\bphi}$, and $\Phi_{\coneC}$ is as in Definition \ref{static_mfd}. Computing the right hand side of (\ref{MCF1}) is therefore equivalent to computing $\td{F}_t^*(\om_{\coneC}(\frac{\pl\td{F}_t}{\pl t}, \,\cdot\,) )$ for
$$ \td{F}_t = f_{\dd\ptq_{\vep(t)}} \circ (\bar{\kp}_{\vep(t)})_\srp \circ \dd u : Q \to U_{\coneC}\subset T^*\coneC ~. $$
The canonical symplectic form $\om_{\coneC}$ on $T^*{\coneC} = T^*(\Sm\times(0,\infty))$ is $\om_{\Sm} + \dd\ir \w\dd\is$.

In the following calculation, we use the equivariant coordinates introduced around \eqref{cone scaling}.  By a direct computation,
\begin{align*}
	(\dd u)(\sm,r) &= \left( \sm, r, (\dd_\Sm u)(\sm,r), (\pl_r u)(\sm,r) \right) ~, \\
	(\bar{\kp}_{\vep(t)})_\srp(\sm, r, \vsm, s) &= \left( \sm, \kp_{\vep(t)}(r), \vsm, \big(({\pl_r\kp_{\vep(t)})(r)}\big)^{-1}s \right) ~, \\
	f_{\dd\ptq_{\vep(t)}}(\sm, \ir, \vsm, \is) &= \left( \sm, \ir, \vsm + (\dd_\Sm\ptq_{\vep(t)})(\sm,\ir), \is + (\pl_{\ir}\ptq_{\vep(t)})(\sm,\ir) \right) ~.
\end{align*}
In this setting, we will also take the partial derivative of $\ptq_{\vep}(\sm,\ir)$ \eqref{int_potential} and $\kp_\vep(r)$ \eqref{static_dilate} in $\vep$.  Let
\begin{align*}
	\hat{F}_t(\sm, r) = \left( \sm, (\dd_\Sm u)(\sm,r) + (\dd_\Sm\ptq_{\vep(t)})(\sm,{\kp}_{\vep(t)}(r)) \right)
	&: \Sm\times(R_1,R_2)\to T^*\Sm ~, \\
	\check{F}_t(\sm,r) = \left( \kp_{\vep(t)}(r), \frac{(\pl_r u)(\sm,r)}{(\pl_r\kp_{\vep(t)})(r)} + (\pl_{\ir}\ptq_{\vep(t)})(\sm,{\kp}_{\vep(t)}(r)) \right)
	&: \Sm\times(R_1,R_2)\to T^*(0,\infty) ~.
\end{align*}
The above computation means that $\td{F}_t = (\hat{F}_t,\check{F}_t)$, and
\begin{align}
	\td{F}_t^*\Big(\om_{\coneC}\Big(\frac{\pl\td{F}_t}{\pl t}, \,\cdot\,\Big) \Big) &= \hat{F}_t^*\Big(\om_{\Sm}\Big(\frac{\pl\hat{F}_t}{\pl t}, \,\cdot\,\Big) \Big) + \check{F}_t^*\Big((\dd\ir\w\dd\is)\Big(\frac{\pl\check{F}_t}{\pl t}, \,\cdot\,\Big) \Big) ~.
\label{int_split} \end{align}

For the first term on the right hand side of \eqref{int_split}, the same argument as that in the proof of Lemma \ref{Behrndt lemma} shows that
\begin{align*}
	\hat{F}_t^*\Big(\om_{\Sm}\Big(\frac{\pl\hat{F}_t}{\pl t}, \,\cdot\,\Big) \Big) &= -\frac{\pl}{\pl t} \left[ (\dd_\Sm u)(\sm,r) + (\dd_\Sm\ptq_{\vep(t)})(\sm,{\kp}_{\vep(t)}(r)) \right] \\
	&= -\left(\dd_\Sm\left[\frac{\pl u}{\pl t}\right]\right)(\sm,r) - \vep'(t)\cdot \left(\dd_\Sm({\pl_\vep\ptq_{\vep(t)}})\right)(\sm,{\kp}_{\vep(t)}(r)) \\
	&\qquad - \vep'(t)\cdot (\pl_\vep\kp_{\vep(t)})(r)\cdot \left(\dd_\Sm({\pl_\ir\ptq_{\vep(t)}})\right)(\sm,{\kp}_{\vep(t)}(r)) ~.
\end{align*}
By a direct computation,
\begin{align*}
	\check{F}_t^*\Big((\dd\ir\w\dd\is)\Big(\frac{\pl\check{F}_t}{\pl t}, \,\cdot\,\Big) \Big) &= \vep'(t)\cdot (\pl_\vep\kp_{\vep(t)})(r)\cdot \dd\left[ \frac{(\pl_r u)(\sm,r)}{(\pl_r\kp_{\vep(t)})(r)} + (\pl_{\ir}\ptq_{\vep(t)})(\sm,{\kp}_{\vep(t)}(r)) \right] \\
	&\quad - \frac{\pl}{\pl t}\left[\frac{(\pl_r u)(\sm,r)}{(\pl_r\kp_{\vep(t)})(r)} + (\pl_{\ir}\ptq_{\vep(t)})(\sm,{\kp}_{\vep(t)}(r))\right]\cdot (\pl_r\kp_{\vep(t)})(r)\,\dd r \\
	&= -\left(\pl_r\left[\frac{\pl u}{\pl t}\right]\right)(\sm,r)\,\dd r + \dd\left[ \vep'(t)\cdot (\pl_\vep\kp_{\vep(t)})(r)\cdot \frac{(\pl_r u)(\sm,r)}{(\pl_r\kp_{\vep(t)})(r)} \right] \\
	&\quad + \vep'(t)\cdot (\pl_\vep\kp_{\vep(t)})(r)\cdot \left(\dd_\Sm({\pl_\ir\ptq_{\vep(t)}})\right)(\sm,{\kp}_{\vep(t)}(r)) \\
	&\quad - \vep'(t)\cdot \left(\pl_{\ir}({\pl_\vep\ptq_{\vep(t)}})\right)(\sm,{\kp}_{\vep(t)}(r))\cdot (\pl_r\kp_{\vep(t)})(r)\,\dd r ~.
\end{align*}
Putting these into \eqref{int_split} gives that
\begin{align*}
	\td{F}_t^*(\om_{\coneC}(\frac{\pl\td{F}_t}{\pl t}, \,\cdot\,) ) &= \dd\left[ - \frac{\pl u}{\pl t} + \vep'(t)\cdot \frac{\pl_\vep\kp_{\vep(t)}}{\pl_r\kp_{\vep(t)}}\cdot \pl_ru - \vep'(t)\cdot (\pl_\vep\ptq_{\vep(t)})\circ\bar{\kp}_{\vep(t)} \right] ~.
\end{align*}

It follows that, on the intermediate region $Q$, \eqref{MCF1} becomes
\begin{align}
	\dd\left[\frac{\pl u}{\pl t}\right] &= \dd\left[\ta(\dd u) + \vep'(t)\cdot \frac{\pl_\vep\kp_{\vep(t)}}{\pl_r\kp_{\vep(t)}}\cdot \pl_ru - \vep'(t)\cdot (\pl_\vep\ptq_{\vep(t)})\circ\bar{\kp}_{\vep(t)} \right] ~.
\label{MCF1_int} \end{align}

%%%%%%%%%%%%%%%%
\subsubsection{Conclusion}

Equations \eqref{MCF1_ext}, \eqref{MCF1_tip} and \eqref{MCF1_int} are summarised in the following proposition:
%%%%%%%%
\begin{proposition} \label{dMCF}
Under Assumption \ref{assumption torus}, suppose that there exist a $1$-parameter family of functions, $u$, on $\stm$, a real number $\Ld>0$, and $\vep(t)$ such that for $t \in [\Lambda, \infty)$
\begin{itemize}
    \item $\vep(t)$ satisfying \eqref{ep_range};
    \item $\dd u$ belongs to the open set $U_{N^{\vep(t)}}$, where $U_{N^{\vep(t)}}$ is constructed by Definition \ref{desing_nbd}.
\end{itemize}
Then, the one-parameter family of immersions $\Psi_{N^{\vep(t)}}\circ\dd u: \stm \to M$ is a solution to the mean curvature flow if and only if
\begin{align}
	\dd\left[\frac{\pl u}{\pl t}\right] &= \begin{cases}
	\dd\left[\ta(\dd u) + (\vep(t)^2)'\cdot \af_{L}\circ\dd(\vep(t)^{-2} u) \right] & \text{on }P ~, \\
	\dd\left[\ta(\dd u) + \vep'(t)\cdot \frac{\pl_\vep\kp_{\vep(t)}}{\pl_r\kp_{\vep(t)}}\cdot \pl_ru - \vep'(t)\cdot (\pl_\vep\ptq_{\vep(t)})\circ\bar{\kp}_{\vep(t)} \right] & \text{on }Q ~, \\
	\dd\left[\ta(\dd u)\right] & \text{on }\out ~.
\end{cases}
\label{MCF2} \end{align}
Here, $\kp_{\vep(t)}$ and $\ptq_{\vep(t)}$ are defined in Definition \ref{desing}; $\af_{L}$ is the output of Theorem \ref{AC Lag nbhd} on the Lawlor neck $L$ (found by step 3 of Definition \ref{static_mfd}).
\end{proposition}

%%%%%%%%%%%%%%%%%%%%%%%%%%%%%%%%
\subsection{On the Potential} \label{sec_on_potential}

The right hand side of \eqref{MCF2} is locally exact.  It is natural to ask when the \eqref{MCF2} can be integrated to the level of potentials.  The Lagrangian angle $\ta(\dd u_t)$ is globally defined. We fix the branch by requiring that $\theta_{N^{\vep}}|_{\out} \, = \, \ta[\dd 0]|_{\out} \,=\, 0$.

The intermediate region $Q$ has two connected component, $Q^-$ and $Q^+$, corresponding to $\Pi^0$ and $\Pi^{\bphi}$ respectively.  The outer region $\out$ has two connected components, $\out_{\,1} = \out\cap X_1$ and $\out_{\,2} = \out\cap X_2$.

Note that on $\Sm\times((1-\hbar)R_2,R_2)\subset Q$, it is not hard to check that the middle line of \eqref{MCF2} is exactly $\dd\left[\ta(\dd u)\right]$.
Thus, if there exist time-dependent constants, $C_{P}(t)$, $C_{Q^-}(t)$, and $C_{Q^+(t)}$ such that
\begin{align}
	\frac{\pl u}{\pl t} &= \begin{cases}
	\ta(\dd u) + (\vep(t)^2)'\cdot \af_{L}\circ\dd(\vep(t)^{-2} u) + C_{P}(t) & \text{on }P ~, \\
	\ta(\dd u) + \vep'(t)\cdot \frac{\pl_\vep\kp_{\vep(t)}}{\pl_r\kp_{\vep(t)}}\cdot \pl_ru - \vep'(t)\cdot (\pl_\vep\ptq_{\vep(t)})\circ\bar{\kp}_{\vep(t)} + C_{Q^\pm}(t) & \text{on }Q^\pm ~, \\
	\ta(\dd u) + C_{Q^-}(t) & \text{on }\out_{\,1} ~, \\
        \ta(\dd u) + C_{Q^+}(t) & \text{on }\out_{\,2} ~,
\end{cases}
\label{MCF3}
\end{align}
then \eqref{MCF2} holds.  It remains to study the relation between $C_{P}(t)$ and $C_{Q^\pm}(t)$.

%Let $\{\out_{\,b}\}_{b=1}^{n'}$ be the connected components of the outer region $\out$. It follows that if there exist time-dependent constants, $C_{P_j}(t)$, $C_{Q_j^-}(t)$, $C_{Q_j^+(t)}$ and $C_{\out_{\,b}}(t)$ such that
%\begin{align}
%	\frac{\pl u}{\pl t} &= \begin{cases}
%	\ta(\dd u) + (\vep_j(t)^2)'\cdot \af_{L_j}\circ\dd(\vep_j(t)^{-2} u) + C_{P_j}(t) & \text{on }P_j ~, \\
%	\ta(\dd u) + \vep'_j(t)\cdot \frac{\pl_\vep\kp_{\vep_j(t)}}{\pl_r\kp_{\vep_j(t)}}\cdot \pl_ru - \vep_j'(t)\cdot (\pl_\vep\ptq_{\vep_j(t)})\circ\bar{\kp}_{\vep_j(t)} + C_{Q_j^\pm}(t) & \text{on }Q_j^\pm %~, \\
%	\ta(\dd u) + C_{\out_{\,b}}(t) & \text{on }\out_{\,b} ~,
%\end{cases}
%\label{MCF3}
%\end{align}
%then \ref{MCF2} holds. We now investigate the necessary conditions for the existence of such constants.

The overlap between the intermediate and tip region corresponds to $\Sm\times(R_1,(1+\hbar)R_1)\subset Q$.  The expression on the right hand side of \eqref{MCF3} is based on the coordinate of each piece.  To compare the equation, we have to use the same parametrisation.  Parametrise the overlap part of $P$ by the transition map:
\begin{align*}
	\Phi_{\coneC}\circ\dd\ptl: \Sm\times(R_1,(1+\hbar)R_1)\subset Q \to P ~.
\end{align*}
Denote $\Phi_{\coneC}\circ\dd\ptl$ by $\vph$.  For $u: Q\to\BR$, one has to plug $u\circ\vph^{-1}$ for $u$ into \eqref{MCF3} on $P$, and compose with $\vph$.  That is to say, \eqref{MCF3} on $P$ transforms into the following expression on $\Sm\times(R_1,(1+\hbar)R_1)\subset Q$: 
\begin{align*}
	&\quad \ta(\dd u) + (\vep(t)^2)'\cdot \af_{L}\circ\dd(\vep(t)^{-2} u\circ\vph^{-1})\circ\vph + C_{P}(t) & & \\
	&= \ta(\dd u) + (\vep(t)^2)'\cdot \af_{L}\circ(\vph)_\srp\circ\dd(\vep(t)^{-2} u) + C_{P}(t) &\text{by \eqref{sharp0}} & \\
	&= \ta(\dd u) + (\vep(t)^2)'\left[ \frac{r}{2}\frac{\pl_r u}{\vep(t)^{2}} + \frac{r}{2}(\pl_r\ptl)(\sm,r) - \ptl(\sm,r) + c_{\pm}(L) \right]  + C_{P}(t) &\text{by \eqref{afL_def}} & ~.
\end{align*}
Since $L$ has two ends, $c_{\pm}(L)$ are the corresponding constants produced by Theorem \ref{AC Lag nbhd} for the Lawlor neck $L = L^{\bphi,A}$. Recall that we choose $\alpha_{L}$ such that $c_-(L) = 0$, and so $c_+(L)$ is given by the right-hand side of \eqref{Lawlor_const}. For convenience, we denote $c_+(L) = c_+(L) - c_-(L)$ by $c_L$.

For \eqref{MCF3} on $Q\supset\Sm_j\times(R_1,(1+\hbar)R_1)$, it follows from Lemma \ref{static_dilate_property} that $\kp_{\vep(t)}(r) = \vep(t)\,r$, and $\ptq_{\vep(t)}(\sm,\ir) = \vep(t)^2\,\ptl_j(\sm, \vep(t)^{-1}\ir)$.  A direct computation shows that \eqref{MCF3} on $Q^\pm$ becomes
\begin{align*}
	\ta(\dd u) + \vep'(t) \frac{r}{\vep(t)}\pl_ru - \vep'(t) \left[ 2\vep(t)\ptl(\sm,r) - \vep(t)\,r\,(\pl_r\ptl)(\sm,r) \right] + C_{Q^{\pm}}(t) ~.
\end{align*}

To sum up, the matching condition of \eqref{MCF3}  on the intermediate-tip region is $ (\vep(t)^2)'\, c_L + C_{P}(t) = C_{Q^+}(t)$ and $C_{P}(t) = C_{Q^-}(t)$. %The compatibility condition therefore reduces to
%\begin{align}
%	C_{Q^+}(t) - C_{Q^-}(t) &=  c_L\cdot(\vep(t)^2)' ~.
%\label{match_int_tip} \end{align}
We now fix a particular choice of these constants.  Denote by $V_b$ be the volume of $\iota(X_b)$ for $b \in \{1, 2\}$.  Let
\begin{align}\label{choice of constants in l.o.t} \begin{split}
    &C_{P}(t) = C_{Q^{-}}(t) = -\frac{c_LV_{2}}{V_{1}+V_{2}}\frac{\dd\vep^{2}(t)}{\dd t} ~, \\
    &C_{Q^{+}}(t) =  \frac{c_LV_{1}}{V_{1}+V_{2}}\frac{\dd\vep^{2}(t)}{\dd t} ~.
\end{split} \end{align}

\section{The Linear Operator and its Approximate Kernel}\label{sec-linear-approxkernel}

In this section, we derive the linearisation of the LMCF equation (\eqref{MCF3} and \eqref{choice of constants in l.o.t}), and construct an approximate kernel for this operator (\ref{def: approximate kernel}).

%From now on, we will assume the conditions in Proposition \ref{MCF_match} are satisfied, so that the Lagrangian mean curvature flow can be integrated to the level of potentials.
Given any smooth function $u:\stm\times[\Ld, \infty)\to\BR$ such that $\dd u(x, t)\in U_{N^{\vep}(t)}$ for all $(x, t)\in\stm\times[\Ld, \infty)$, define a function $\xi(\dd u):\stm\to\BR$ by
\begin{align}\label{l.o.t}
    \xi(\dd u) &= \begin{cases}
	(\vep(t)^2)'\cdot \af_{L}\circ\dd(\vep(t)^{-2} u) + C_{P}(t) & \text{on }P ~, \\
	\vep'(t)\cdot \frac{\pl_\vep\kp_{\vep(t)}}{\pl_r\kp_{\vep(t)}}\cdot \pl_ru - \vep'(t)\cdot (\pl_\vep\ptq_{\vep(t)})\circ\bar{\kp}_{\vep(t)} + C_{Q^\pm}(t) & \text{on }Q^\pm ~, \\
        C_{Q^-}(t) & \text{on }\out_{\,1} ~, \\
	C_{Q^+}(t) & \text{on }\out_{\,2} ~. \\
\end{cases}
\end{align}
It follows from the discussion in section \ref{sec_on_potential} that $\xi(\dd u)$ is well-defined, and 
%(it follows from the assumption that there exist time-dependent constants $C_{P_{j}}(t)$, $C_{Q^{\pm}_{j}}(t)$, and $C_{\out_{\,b}}(t)$ such that $\xi(\dd u)$ is well-defined).
\eqref{MCF3} becomes
\begin{align}\label{single_eq}
    \pl_{t}u = \ta(\dd u) + \xi(\dd u).
\end{align}
The function $\xi(0)$ has its geometric significance. It is not hard to see that $\xi(0)$ is the potential of the velocity of $N^{\vep}$, namely,
\begin{align}
    (\iota^{\vep})^{*}(\omega(\frac{\dd\iota^{\vep}}{\dd t}, \cdot)) = \dd[\xi(0)].
\end{align}

For the remainder of the paper, we will linearise the right hand side of (\ref{single_eq}) at the zero section $\ul{0}$ and split it into zeroth order, linear and higher order parts, denoted as follows:

\begin{equation}
    \partial_t u \, = \, \theta_{N^{\bvep}} + \xi(0) + \mathcal{L}^\vep_{\ul 0}[u] + Q^{\vep}[u].
\end{equation}

%%%%%%%%%%%%%%%%%%%%%%%%%%%%%%%%
\subsection{Linearised LMCF}\label{sec-linearisation}

Denote the embedding of the zero section by $\iota^{\vep} := \Phi_{N^{\vep}}\circ \ul{0}: \stm\times[\Ld, \infty)\to M$. Let $u:\stm\times[\Ld, \infty)\to\BR$ be a smooth function such that $s\,\dd u(x, t)\in U_{N^{\vep(t)}}$ for all $(x, t)\in\stm\times[\Ld, \infty)$ and small $s\in\BR$.

We first employ the following result by Behrndt \cite{Behrndt2011}.
\begin{lemma}\label{lem: linearised angle}
The deformation vector field of $\iota^{\vep}$ in the direction of $\dd u$ is given by
\begin{align}\label{deformation vector field}
    \left.\frac{\dd}{\dd s}\right|_{s=0}\Phi_{N^{\vep}}\circ\dd(su) = J(\iota^{\vep})_{*}\nabla u + (\iota^{\vep})_{*}\widehat{V}(\dd u)
\end{align}
for some $\widehat{V}(\dd u)\in\Gamma(T\underline{N})$, where the gradient $\nabla u$ is computed using the induced metric $g^{\vep} := (\iota^{\vep})^{*}g$ on $\stm$. Moreover, the linearisation of the Lagrangian angle is given by
\begin{align}\label{linearised angle}
    \left.\frac{\dd}{\dd s}\right|_{s=0}\ta(\dd(su)) = \Delta_{g^{\vep}}u \,-\, \langle\nabla\ta, \widehat{V}(\dd u)\rangle_{g^{\vep}}~. %- \langle\nabla\ta(\dd v), \widehat{V}_{\dd v}(\dd u)\rangle_{g_{\dd v}}.
\end{align}
\end{lemma}

Next, we linearise $\xi(\dd u)$. In the tip region, we have
\begin{align}\label{l.o.t. tip}
    \left.\frac{\dd}{\dd s}\right|_{s=0}\xi(s\dd u) &= (\vep(t)^{2})'\cdot\left.\frac{\dd}{\dd s}\right|_{s=0}\af_{L}\circ\dd(\vep(t)^{-2}su) \notag \\
    &=(2\log\vep(t))'(\dd\bt_{L})(\dd u^{V}),
\end{align}
where, in standard local coordinates $\{x^{i}, p_{i}\}_{i=1}^{m}$ of $T^{*}\stm$, $\dd u^{V}$ is the vertical vector field $\dd u^{V} := \frac{\partial u}{\partial x^{j}}\frac{\partial}{\partial p_{j}}$ on $U_{N^{\vep}}$. In the intermediate region, it is clear that
\begin{align}\label{l.o.t. intermediate}
    \left.\frac{\dd}{\dd s}\right|_{s=0}\xi(s\dd u)  = \vep'(t)\cdot \frac{\pl_\vep\kp_{\vep(t)}}{\pl_r\kp_{\vep(t)}}\cdot \pl_ru.
\end{align}
Finally, in the outer region, we have 
\begin{align}\label{l.o.t. outer}
    \left.\frac{\dd}{\dd s}\right|_{s=0}\xi(s\dd u) = 0.
\end{align}
The above computations are summarised by the following proposition.
\begin{proposition}\label{prop: linearised operator}
The linearisation of (\ref{single_eq}) at the zero section is given by
\begin{align}
    \partial_{t}u - \mathcal{L}^{\vep}_{\ul 0}[u] &:= \partial_{t}u - \frac{\dd}{\dd s}_{|_{s=0}}(\theta + \xi)(s\dd u) \notag \\
    & = \partial_{t}u - \Delta_{g^{\vep}}u + \langle\nabla\ta, \widehat{V}(\dd u)\rangle_{g^{\vep}} - S^{\vep}[u]. \label{linearisation}
\end{align}
where $S^{\vep}[u]$ is a first order linear differential operator defined by
\begin{align}
    S^{\vep}[u] &= \begin{cases}
	(2\log\vep(t))'(\dd\bt_{L})(\dd u^{V}) & \text{on }P ~, \\
	\vep'(t)\cdot \frac{\pl_\vep\kp_{\vep(t)}}{\pl_r\kp_{\vep(t)}}\cdot \pl_ru & \text{on }Q^\pm ~, \\
	0 & \text{on }\out ~.
\end{cases}
\end{align}
\end{proposition}
\subsection{Approximate Kernels} \label{sec-approxkernel}

Define the function $\underline{\alpha} :P \cup Q \to \mathbb{R}$ by
\begin{align*}
    \underline{\alpha}(p) &:= \alpha_{L}\big|_{\underline{0}}(p) &\text{for } p \in P ~,\\
    \underline{\alpha}(\sigma,r) &:= \alpha_{L}\big|_{\underline{0}} \left(\vph( \sigma,\vep^{-1}\kp_{\vep}(r))\right) &\text{for } (\sigma,r) \in Q ~,
\end{align*}
where $\varphi$ is as in Definition \ref{AC lag} for the Lawlor neck $L$.  By considering its asymptotic behavior, one can kind of extend it to $\stm$ by
\begin{align}\label{def: approximate kernel}
    w_{(0,1)}^{\vep} := 
    \begin{cases}
        \frac{1}{c_L}\underline{\alpha} & \text{on } j\\
        \frac{1}{c_L}\underline{\alpha} \cdot \left(1 - \chi\big(2\vep^{-\tau}\kp_{\vep_j}(r)\big) \right) & \text{on } Q^- \\
        1 + \frac{1}{c_L}\underline{\alpha} \cdot \left(1 - \chi\big(2\vep^{-\tau}\kp_{\vep_j}(r)\big) \right) & \text{on } Q^+ \\
        0 & \text{on } \out_{\,1}\\
        1 & \text{on } \out_{\,2}\\
    \end{cases}
\end{align}
We normalise it so that it has zero total integral:
 \begin{align}
    w^\vep := w^\vep_{(0,1)} - \frac{1}{\text{Vol}(N^\vep)} \int_{\ul N}w^{\vep}_{(0,1)} \dd V_{g^\vep} ~. \label{eq-normalisedapproxkernel}
\end{align}

%By interpolating this function with constants on the exterior region, we construct the `approximate kernel' of our linearised operator.

%Explicitly, given $\bd := (d_1,\ldots, d_{n'}) \in \BR^{n'}$, we define the function $w_{\bd}^{\bvep}$ to be:
%\begin{align}\label{def: approximate kernel}
%    w_{\bd}^{\bvep} := 
%    \begin{cases}
%        d_b & \text{on } \out_b\\
%        d_{\scriptsize\ula{j}} + \frac{1}{c_j} (d_{\scriptsize\ura{j}} - d_{\scriptsize\ula{j}})\,\underline{\alpha}_j & \text{on } P_j\\
%        d_{\scriptsize\ula{j}}\chi\big(2\vep_j^{-\tau}\kp_{\vep_j}(r)\big) \, + \, \left(d_{\scriptsize\ula{j}} + \frac{1}{c_j} (d_{\scriptsize\ura{j}} - d_{\scriptsize\ula{j}})\underline{\alpha}_j\right)\left(1 - \chi\big(2\vep_j^{-\tau}\kp_{\vep_j}(r)\big)\right)& \text{on } Q_j^-\\
%        d_{\scriptsize\ura{j}}\chi\big(2\vep_j^{-\tau}\kp_{\vep_j}(r)\big) \, + \, \left(d_{\scriptsize\ula{j}} + \frac{1}{c_j} (d_{\scriptsize\ura{j}} - d_{\scriptsize\ula{j}})\underline{\alpha}_j\right)\left(1 - \chi\big(2\vep_j^{-\tau}\kp_{\vep_j}(r)\big)\right)& \text{on } Q_j^+
%    \end{cases}
%\end{align}

It turns out that the functions $\{1, w^{\vep}\}$ form the `approximate kernel' of $\mathcal{L}^{\vep}_{\ul{0}}$.  According to \cite{Joyce2004SLCS3}*{p.49}  (see also \cite{DLee2004}*{Lemma~11}), they approximate the {\it small eigenfunctions} of the Laplacian $\Delta_{g^{\vep}}$, with eigenvalues of the order $O(|\vep|^{m-2})$.  Since our linearised operator $\mathcal{L}^{\vep}_{\ul{0}}$ will turn out to be a small perturbation of the Laplacian (see Lemma~\ref{lem-lower order term estimate}), we have $\mathcal{L}^{\vep}_{\ul{0}}w^{\vep}\approx 0$. %These functions are the obstructions to the uniform invertibility of the linearised operator.

\section{A Priori Estimates for the Linear Operator}\label{sec-apriori}

In this section, we prove a uniform injectivity estimate for solutions to the inhomogeneous heat equation which are orthogonal to the approximate kernel (Theorem \ref{weighted sup norm estimate}). 

We will from now on assume that $\vep(t)$ satisfies the following estimates. Ultimately, $\vep$ will solve an ODE that appears as the dominant term in an integral error, and such a solution will automatically satisfy these estimates (see the last paragraph of section \ref{sec-estimate-error}).
\begin{assumption} \label{assumption on epsilon}
    There exist $\Lambda>0$ and $C,C'>1$ such that:
    \begin{align}
        C^{-1}t^{-\frac{1}{m-2}} \,\leq \,\vep(t) &\leq\, Ct^{-\frac{1}{m-2}},\quad |\vep'(t)|\leq Ct^{\frac{1-m}{m-2}} \leq C' \vep(t)^{m-1}, \notag \\
        \text{and}\quad \frac{|\vep'(t_1) - \vep'(t_2)|}{|t_1 - t_2|^{\alpha}}\,&\leq\, C t^{\frac{1-m}{m-2} + \frac{2\alpha}{m-2}} \,\leq\, C' \bvep(t)^{m-1-2\alpha}.
    \end{align}
    for all $t \in[\Lambda, \infty)$, $t_1,t_2 \in [t,2t]$, $0<|t_1 - t_2|<t^{-\frac{2}{m-2}}$.
\end{assumption}
\noindent Note that Assumption \ref{assumption on epsilon} also provides the following bound on the weight function given in Definition \ref{weight}, for some constant $C' > 1$:
\begin{equation}\label{assumption on rho}
    (C')^{-1} \rho_{t^{-\frac{1}{m-2}}} \leq \rho_{\vep(t)} \leq C' \rho_{t^{-\frac{1}{m-2}}}. 
\end{equation}

\subsection{Liouville Theorems}

The proof of Theorem \ref{weighted sup norm estimate} is based on a blow-up argument which ultimately reduces the question to Liouville-type theorems on various model spaces. We start by establishing these theorems.

\subsubsection{Lawlor Necks}

The corresponding Liouville theorem on the Lawlor neck is obtained by adapting the scheme of Lockhart and McOwen \cites{Lock1987, LockMc1985}.  The main machinery in the current setting is established by Joyce in \cite{Joyce2002SLCS1}*{section 7.3}, which is summarised here for the reader's convenience.  Note that $\Dt$ in \cite{Joyce2002SLCS1} is the Hodge Laplacian, which differs from the Laplacian in this paper by a minus sign.

Let $L\subset\BC^m$ be a Lawlor neck described by Proposition \ref{Lawlor}.  Let ${\hat\rho}(\bx):L \to[1, \infty)$ be the smooth function defined in Definition \ref{weight}.  Given $k\in\BN\cup\{0\}$ and $\nu\in\mathbb{R}$, define the spaces $C^{k}_{\nu}(L)$ to be the set of locally $C^{k}$ functions whose weighted norm
\begin{align*}
    \|u\|_{C^{k}_{\nu}} = \sum_{j=0}^{k}\sup_{L}\left|\hat\rho^{j+\nu}\,\nabla^{j}u\right|_{g}
\end{align*}
is finite.  The covariant derivative and the norm are computed using the induced metric $g = \iota_{L}^{*}g_{0}$, where $\iota_{L}:L\to\mathbb{C}^{m}$ is the inclusion map.  The $C^\infty_\nu$ space is defined to be the intersection of all $C^{k}_{\nu}$ spaces: $C^{\infty}_{\nu}(L) = \bigcap_{k=0}^{\infty} C^{k}_{\nu}(L)$. 

Similarly, the weighted Sobolev spaces $W^{k, p}_{\dt}$ is defined by the norm
\begin{align*}
    \|u\|_{W^{k, p}_{\nu}} = \left(\sum_{j=0}^{k}\int_{L}\left|\hat\rho^{j+\nu}\,\nabla^{j}u\right|^{p}\,\hat\rho^{-m}\,\dd V_{g}\right)^{1/p} ~.
\end{align*}
As usual, denote $W^{0, p}_{\nu}(L)$ by $L^{p}_{\nu}(L)$.  For any $\nu\in\BR$, $p>1$, and $k\geq 2$, the Laplace operator $\Dt_{g}: C^{\infty}_{\text{cpt}}(L)\to C^{\infty}_{\text{cpt}}(L)$ extends to a continuous operator
\begin{align*}
    \Dt^{k,p}_{\nu}: W^{k, p}_{\nu}(L) \longrightarrow W^{k-2,p}_{\nu+2}(L) ~.
\end{align*}

The operator $\Dt^{k, p}_{\nu}$ is Fredholm for generic $\nu$.  Here is the complete characterisation.  The Lawlor neck $L$ is asymptotic to $\Pi^0\cup\Pi^{\bphi}$ near infinity.  The link $\Sm$ of $\Pi^0\cup\Pi^{\bphi}$ is the disjoint union of two round spheres.  Let
\begin{align*}
    \CD_{\Sm} = \left\{ \nu\in\BR \,|\, \nu(-\nu+m-2) \text{ is an eigenvalue of } \Dt_\Sm \right\} ~.
\end{align*}
It is not hard to verify that $\CD_\Sm$ is a discrete subset of $\BR$ satisfying $m-2\in\CD_\Sm$, $0\in\CD_\Sm$, and $\CD_\Sm\cap(2-m, 0) = \varnothing$.  It turns out that $\Dt^{k, p}_{\nu}$ is Fredholm if and only if $\nu\in\BR\setminus\CD_\Sm$.

Its Fredholm index,
$\ind\left(\Dt^{k,p}_{\nu}\right) = \dim\ker\left(\Dt^{k,p}_{\nu}\right) - \dim\coker\left(\Dt^{k,p}_{\nu}\right)$,
depends only on the connected components of $\BR\setminus\CD_\Sm\ni\nu$, and is given by
\begin{align}\label{index formula}
    \ind\left(\Dt^{k,p}_{\nu}\right) = \bfN_{\Sm}(\nu) ~,
\end{align}
where $\bfN_{\Sm}:\BR\to\BZ$ is defined by
\begin{align*}
    \bfN_{\Sm}(\nu) =\begin{cases}
        -\sum_{\gm\in\CD_\Sm\cap(0,\nu)}\bm_{\Sm}(\gm) &\text{when }\nu>0 ~, \\
        \sum_{\gm\in\CD_\Sm\cap[\nu,0]}\bm_{\Sm}(\gm) &\text{when }\nu\leq0 ~,
    \end{cases}
\end{align*}
and $\bm_\Sm(\gm)$ is the multiplicity of the eigenvalue $\gm(-\gm+m-2)$ of $\Dt_{\Sm}$.

From now on, focus on the Fredholm case, $\nu\notin\CD_\Sm$.  According to the weighted elliptic estimate and the weighted Sobolev embedding, any $u \in \ker(\Dt^{k,p}_{\nu}) \subset W^{k,p}_{\nu}(L)$ must be smooth, $u\in C^{\infty}_{\nu}(L)$.  By using the maximum principle, if $u \in \ker(\Dt^{k,p}_{\nu}) \subset W^{k,p}_{\nu}(L)$ with $\nu>0$, then $u\equiv 0$.  In other words, $\Delta^{k,p}_{\nu}$ is injective when $\nu>0$.

%%%%%%%%
\begin{lemma}\label{kernel of Laplacian on Lawlor neck}
    Let $u\in W^{k,p}_{\nu}(L)$ with $\nu>0$. Suppose $\Delta_{g}u = 0$ in the distributional sense, then $u\equiv 0$.
\end{lemma}

\begin{remark}
    By the duality property, the cokernel of $\Dt^{k,p}_{\nu}$ is isomorphic to the dual space of the kernel of $\Dt^{k, q}_{-\nu+m-2}$, where $1/p+1/q=1$.  Thus, $\Dt^{k,p}_{\nu}$ is a surjective when $\nu<m-2$.  It follows that $\Dt^{k,p}_{\nu}$ is an isomorphism when $\nu\in(0,m-2)$.
\end{remark}

We now prove a Liouville theorem for the heat equation on the Lawlor neck.
%%%%%%%%
\begin{proposition}\label{Liouville thm tip}
    Let $u:L\times(-\infty,0)\to\mathbb{R}$ be a solution to the heat equation $\partial_{t}u = \Dt_{g}u$.  Suppose there exist $c>0$ and $\nu\in(0, m-2)$ such that $|u(\,\cdot\,, t)|\leq C\hat\rho^{-\nu}$ for all $t\in(-\infty,0)$.  Then, $u\equiv 0$.
\end{proposition}

\begin{proof}
    By the the weighted Schauder estimate (\cite{Behrndt2011} and \cite{Su2020}*{section 3.2}) and the bootstrapping argument, $u(\,\cdot\,, t)\in C^{k}_{\nu}(L)$ for any $k\geq 2$. %It follows that $u$ is a smooth solution to the heat equation.
    It follows that $|(\Dt_g)^\ell u| \leq C_\ell\hat\rho^{-\nu-2\ell}$ for any $\ell\in\BN$.

    Fix $\ell$ with $4\ell>m-2\nu$, and let $w = (\Dt_g)^\ell u$.  Clearly, $\partial_{t}w = \Dt_{g}w$.  Consider $E(t)=\frac{1}{2}\int_{L}w^{2}(\,\cdot\,, t)\,\dd V_{g}$.  The choice of $\ell$ guarantees that $E(t)<\infty$.  Its derivative is
    \begin{align*}
        \frac{\dd}{\dd t}E(t) = \int_{L}w\Dt_{g}w\,\dd V_{g} = -\int_{L}|\nabla w|^{2}_{g}\,\dd V_{g}\leq 0 ~,
    \end{align*}
    and hence $E(t)$ is non-increasing in $t$.  It follows that $\lim_{t\to-\infty}E(t)$ exists, and denote the limit by $E$.

    We claim that $E = 0$, which implies that $w\equiv0$.  Pick a sequence $t_{j}\to -\infty$. Define $w^{(j)}(x, t)$ to be $w(x, t_{j}+t)$.  After passing to a subsequence, $w^{(j)}$ converges smoothly on every compact subset of $L\times\BR$ to an eternal solution $\widehat{w}:L\times\mathbb{R}\to\mathbb{R}$ of the heat equation.  Since $|\widehat{w}|\leq C\hat\rho^{-\nu-2\ell}$, it follows from the dominated convergence theorem that
    \begin{align*}
        \int_{L}\widehat{w}(\,\cdot\,, t)^{2}\,\dd V_{g} = \lim_{t_{j}\to-\infty}\int_{L}w(\,\cdot\,, t_{j}+t)^{2}\,\dd V_{g} = 2E \quad\text{for all } t\in\BR ~.
    \end{align*}
    Taking derivative in $t$ give
    \begin{align*}
        0 = \frac{\dd}{\dd t}\frac{1}{2}\int_{L}\widehat{w}(\,\cdot\,, t)^{2}\,\dd V_{g} = -\int_{L}|\nabla\widehat{w}(\,\cdot\,, t)|^{2}_{g}\,\dd V_{g} ~.
    \end{align*}
    It follows that $\widehat{w}(\,\cdot\,, t)$ is a time-dependent constant function. Since $\widehat{w}(\,\cdot\,, t)$ tends to zero at the end of $L$, the constant must be $0$.  Hence, $E = 0$ as claimed.

    In other words, $(\Dt_g)^\ell u\equiv0$.  According to Lemma \ref{kernel of Laplacian on Lawlor neck}, $u\equiv0$.
\end{proof}

\subsubsection{Punctured $m$-Planes}

The model space for the intermediate region is the cone over the link.  In our setting, it is the union of two punctured $\BR^m$'s, endowed with the standard metric on $\BR^m$.

%%%%%%%%
\begin{proposition} \label{liouville intermediate}
Let $u:(\mathbb{R}^m\backslash\{0\})\times(-\infty, \Ld)\to\mathbb{R}$ be a solution to the heat equation $\partial_{t}u = \Delta_{g}u$, for some $\Ld\in\BR$.  Suppose that there exist $C>0$ and $0<\nu<m-2$ such that $|\nabla^{\ell} u(x, t)|\leq C|x|^{-\nu-\ell}$ for all $t\in(-\infty, \Ld)$ and $\ell\in\{0,1,2\}$.  Then, $u\equiv 0$.
\end{proposition}

\begin{proof}
This proof is a modification of the proof of \cite{BrendleK2017}*{Proposition 5.3}.  It follows from the rate condition that $u$ satisfies the heat equation on $\mathbb{R}^m$ in the sense of distribution.

Fix any $t_0\in(-\infty,\Ld)$.  For any $t>0$ and any $x_0\in\mathbb{R}^m\backslash\{0\}$, 
\begin{align*}
u(x_0,t_0) &= \int_{\mathbb{R}^m\backslash\{0\}} \frac{1}{(4\pi t)^{\frac{m}{2}}} e^{-\frac{|x-x_0|^2}{4t}} u(x,t_0-t) \:\dd x .
\end{align*}
Now, fix $x_0$, and suppose that $t>9|x_0|^2$.  When $|x-x_0|^2\leq t$, it follows from the triangle inequality that $|x| < \frac{4}{3} \sqrt{t}$, and thus
\begin{align*}
\left|\int_{|x-x_0|^2\leq t} \frac{1}{(4\pi t)^{\frac{m}{2}}} e^{-\frac{|x-x_0|^2}{4t}} u(x,t_0-t) \:\dd x\right| &\leq C_1\,t^{-\frac{m}{2}} \int_{|x|^2<\frac{16}{9}t} |x|^{-\nu}\dd x \leq C_2\,t^{-\frac{\nu}{2}} .
\end{align*}
When $|x-x_0|^2\geq t$, it follows from the triangle inequality that $|x| \geq \frac{2}{3} \sqrt{t}$ and $|x-x_0| \leq \frac{1}{2}|x|$.  Since $e^{-s} \leq C_3\,s^{-\frac{m}{2}}$ for any $s>0$,
\begin{align*}
\left|\frac{1}{(4\pi t)^{\frac{m}{2}}} e^{-\frac{|x-x_0|^2}{4t}}\right| &\leq C_4\,|x-x_0|^{-m} \leq C_5\,|x|^{-m}.
\end{align*}
Therefore,
\begin{align*}
\left|\int_{|x-x_0|^2\geq t} \frac{1}{(4\pi t)^{\frac{m}{2}}} e^{-\frac{|x-x_0|^2}{4t}} u(x,t_0-t) \:\dd x\right| &\leq C_6\,\int_{|x|^2\geq\frac{4}{9}t} |x|^{-\nu-m}\dd x \leq C_7\,t^{-\frac{\nu}{2}} .
\end{align*}
Putting the estimates together gives
\begin{align}
|u(x_0,t_0)| &\leq (C_2 + C_7)\,t^{-\frac{\nu}{2}}
\end{align}
whenever $t>9|x_0|^2$.  By taking $t\to\infty$, it implies that $u\equiv0$.
\end{proof}

\subsubsection{Immersed Special Lagrangians} \label{sec-Xliouville}

We now return to the setting of Assumption \ref{assumption torus}.
%our original special Lagrangian immersion $\iota:X\to M$, and recall that $X$ splits into smooth compact connected components $X = \bigcup_{b=1}^{n'} X_{b}$.
Using the notation of Definition \ref{static_mfd} and in analogy with Definition \ref{weight}, there exists a continuous function $\rho: X\to [0,\infty)$ such that 
\begin{itemize}
    \item $\rho\circ \iota^{-1} \circ \Up$ on $B_{R_1} \cap \Up^{-1}(\iota(X))$ is the distance to the origin, with respect to $g_0$;
    \item $\rho \equiv R_2$ on $X \setminus \iota^{-1}(\Up(B_{R_2}))$;
    \item the zero set of $\rho$ is exactly $\xsin$;
    \item $\rho$ is smooth on $X\setminus\{\xsin\}$.
\end{itemize}
The manifold $X$ is endowed with the smooth metric $\iota^*g$.

%%%%%%%%
\begin{proposition} \label{Liouville outer}
    Let $u: X_1\setminus\{\xsin^-\} \,\times\,(-\infty,0)\to\mathbb{R}$ be a solution to the heat equation $\partial_{t}u = \Delta u$.
    Suppose that
    \begin{align*}
        \int_{X_1} u\,\dd V_{\iota^*g} = 0  \quad\text{and}\quad  |\nabla^{\ell} u(\cdot, t)|\leq C\rho^{-\nu-\ell}
    \end{align*}
    for some $C>0$, $\nu\in(0,{m-2})$, and all $t\in(-\infty,0)$ and $\ell\in\{0,1,2\}$.  Then, $u$ vanishes on $X_1$.  The same statement holds true for $X_2\setminus\{\xsin^+\}$
\end{proposition}

\begin{proof}
    Note that $(X_1,\iota^*g)$ is a flat torus.  It follows from the growth rate condition that $u$ obeys the linear heat equation on $X_1$ in the sense of distribution.  Hence, $u$ must be a constant on $X_1$.  It follows from the zero integration condition that $u\equiv0$.
\end{proof}

%%%%%%%%%%%%%%%%%%%%%%%%%%%%%%%%%%%%%%
\subsection{A Priori Estimate for the Heat Operator}

We now apply our Liouville theorems to prove an a priori sup estimate via a blowup argument. In what follows, we recall the weight function $\rho_{\vep}$ of Definition \ref{weight}, we denote the induced metric on $\stm$ by $g^\vep := (\iota^\vep) ^* g$, and define the following weighted norm for tensors on $\ul N$:
\[ \|T\|_{\mu, \nu, \Lambda} := \sup_{(x, t)\in\ul{N}\times[\Lambda, \infty)}t^{\mu}\rho^{\nu}_{t^{\frac{-1}{m-2}}}(x)|T|_{g^\vep} \]

\begin{theorem}\label{weighted sup norm estimate}
    Let $\vep : [\Lambda,\infty) \rightarrow \mathbb{R}_+^n$ be a smooth function satisfying (\ref{ep_range}) and (\ref{assumption on epsilon}), and fix $(\mu, \nu) \in(\frac{\nu + 2}{m-2}, \infty) \times(0,m-2)$. Then there exists a constant $C>0$ with the following significance.
    
    Suppose $u,\psi:\stm\times[\Lambda, \infty)\to\mathbb{R}$ satisfy $\|\psi\|_{\mu,\nu, \Lambda}<\infty$ and solve the Cauchy problem:
    \begin{align}\label{linearised problem}
        \begin{cases}
            \partial_{t}u(x, t) = \Delta_{g^\vep}[u](x, t) + \psi(x, t), & (x, t)\in \stm\times[\Lambda, \infty),\\
            u(x, \Lambda) = 0, & x\in \stm
        \end{cases}
    \end{align}
    and $u$ satisfies the orthogonality conditions
    \begin{align}\label{uniform sup norm orthogonality assumption}
        \int_{\stm}u(x, t)\:\dd V_{g^{\vep}}(x) = \int_{\stm}u(x, t)w^{\vep}(x)\:\dd V_{g^{\vep}}(x) = 0, \quad t\in[\Lambda, \infty) ~.
    \end{align}
    Then,
    \begin{align}\label{uniform sup norm estimate}
        \sup_{\stm\times[\Lambda, \infty)}t^{\mu}\rho_{t^{-\frac{1}{m-2}}}^{\nu}|u|\leq C\sup_{\stm\times[\Lambda, \infty)}t^{\mu}\rho_{t^{-\frac{1}{m-2}}}^{\nu+2}|\psi|.
    \end{align}
\end{theorem}

\begin{proof}
    Assume that the estimate does not hold. Then there exist sequences $u^{(j)}:\stm\times[\Lambda, \infty)\to\mathbb{R}$, $\psi^{(j)}:\stm\times[\Lambda, \infty)\to\mathbb{R}$, and $\vep^{(j)}:[\Lambda, \infty)\to\mathbb{R}_{+}^n$, satisfying the following properties:
    \begin{itemize}
        \item $u^{(j)}$, $\psi^{(j)}$ solve the Cauchy problem (\ref{linearised problem}) for each $j\in\mathbb{N}$.
        \item $u^{(j)}$ satisfies the orthogonality conditions (\ref{uniform sup norm orthogonality assumption}) for each $j\in\mathbb{N}$.
        \item $\sup_{\stm\times[\Lambda, \infty)}t^{\mu}\rho^{\nu}_{t^{-\frac{1}{m-2}}}|u^{(j)}| > j\cdot \sup_{\stm\times[\Lambda, \infty)}t^{\mu}\rho_{t^{-\frac{1}{m-2}}}^{\nu+2}|\psi^{(j)}|$ for all $j\in\mathbb{N}$.
    \end{itemize}
    Thus, for each $j$, we can pick $(x_{j}', t_{j})\in \stm\times[\Lambda, \infty)$ such that
    \begin{equation}
        \sup_{\stm\times[\Lambda,t_j]} t^\mu \rho_{t^{-\frac{1}{m-2}}}^\nu |u^{(j)}| \,\, = \,\, t^{\mu}_{j}\rho^{\nu}_{{t_{j}}^{-\frac{1}{m-2}}}(x_{j}')|u^{(j)}(x_{j}', t_{j})|\,\,\geq \,\,\frac{j}{2}\cdot  \sup_{\stm\times[\Lambda,t_j]} t^\mu \rho_{t^{-\frac{1}{m-2}}}^{\nu+2} |\psi^{(j)}|.
    \end{equation}
    By interior parabolic Schauder estimates, we must have $t_j \rightarrow \infty$. By passing to a subsequence we assume that $\Lambda < \frac{1}{2}t_j$, so that in particular $\Lambda - t_j < -\frac{1}{2} t_j < 0$. Defining 
    \[ {\parallel} u{\parallel}_{\mu,\nu,\Lambda,j} := \sup_{\stm\times[\Lambda,t_j]} t^\mu \rho_{t^{-\frac{1}{m-2}}}^\nu |u|,\]
    we rescale $u^{(j)}$, $\psi^{(j)}$ by ${\parallel} u^{(j)}{\parallel}_{\mu,\nu,\Lambda,j}^{-1}$ so that in addition to the three properties above,
    \begin{equation}\label{nontrivial condition}
        t^{\mu}_{j}\rho^{\nu}_{{t_{j}}^{-\frac{1}{m-2}}}(x_{j}')|u^{(j)}(x_{j}', t_{j})| = {\parallel} u^{(j)}{\parallel}_{\mu,\nu,\Lambda,j} = 1, \quad {\parallel} \psi^{(j)}{\parallel}_{\mu,\nu+2,\Lambda,j} \rightarrow 0.
    \end{equation}
    By passing to a subsequence, we assume that $x_{j}'$ converges on $\stm$, to $x_\infty$. We consider the following three cases. Throughout, we denote the pullback metric by $g_j := (\iota^{\vep^{(j)}})^*g$.
    
    \noindent{\it Case 1.} $\lim_{j\to\infty}t_{j}^{\frac{1}{m-2}}\rho_{t_{j}^{-\frac{1}{m-2}}}(x_{j}')<\infty$
    
    By the definition of $\rho_\vep$, it follows that, on passing to a subsequence, $x_j' \in P \cup Q$ for sufficiently large $j$.
    %We therefore will work in this region, suppressing the index $k$ (e.g. $L_k$ will be written $L$, $\vep^{(j)}_k$ will be written $\vep^{(j)}$).
    We define the scaling factors $\lambda_j$, the region $P^{(j)}_s$, the map $S_t^{(j)}$ and the rescaled functions $\widetilde u^{(j)}$, $\widetilde\psi^{(j)}$ as follows:
    \begin{align*}
        \lambda_j &:= \varepsilon^{(j)}(t_j) \rightarrow 0, \quad t(s):= t_j + \lambda_j^2 s, \\
        P^{(j)}_s &:= P \cup \Sigma \times (R_1,\varepsilon^{(j)}(t(s))^{\tau-1}), \quad P^{(j)} :=  \{(y,s) \,:\, s \in  [-\tfrac{1}{2}\lambda_j^{-2}t_j,0), y \in P^{(j)}_s \}\\
        S^{(j)}_s &: \Sigma \times (R_1,\varepsilon^{(j)}(t(s))^{\tau-1}) \rightarrow \Sigma \times (\varepsilon^{(j)}(t(s))R_1,R_2), \quad S^{(j)}_s (p,r) := (p,\varepsilon^{(j)}(t(s))r)\\
        \widetilde u^{(j)}&: P^{(j)} \rightarrow \mathbb{R}, \quad \widetilde u^{(j)}(y,s):=
        \begin{cases}
            t_{j}^{\mu}\lambda_{j}^{\nu}u^{(j)}(y,t(s)) & \text{ for } y \in P,\\
            t_{j}^{\mu}\lambda_{j}^{\nu}u^{(j)}\left((\overline\kp_{\vep^{(j)}(t(s))})^{-1}\circ S^{(j)}_{s}(y), t(s)\right) & \text{ for } y \in P^{(j)}_s \setminus P,
        \end{cases}\\
        \widetilde \psi^{(j)}&: P^{(j)} \rightarrow \mathbb{R}, \quad \widetilde \psi^{(j)}(y,s):=
        \begin{cases}
            t_{j}^{\mu}\lambda_{j}^{\nu+2}\psi^{(j)}(y,t(s))& \text{ for } y \in P,\\
            t_{j}^{\mu}\lambda_{j}^{\nu+2}\psi^{(j)}\left((\overline\kp_{\vep^{(j)}(t(s))})^{-1}\circ S^{(j)}_{s}(y),t(s)\right) & \text{ for } y \in P^{(j)}_s \setminus P.
        \end{cases}
    \end{align*}
    
    Define $x: P^{(j)} \to \stm$, $x(y,s) := (\overline\kp_{\vep^{(j)}(t(s))})^{-1}\circ S^{(j)}_{s}(y)$, and $y_j := (S^{(j)}_{0})^{-1}\circ \overline\kp_{\vep^{(j)}(t_j)}(x_j')$, and define the time-independent weight function $\hat \rho^{(j)}:P^{(j)} \rightarrow \mathbb{R}_+$ as in Definition \ref{weight} (without the outer region interpolation) using the standard embedding of the Lawlor neck $L$ into $\mathbb{C}^m$. We also endow $P_s^{(j)}$ with the rescaled metric (writing $
    \vep^{(j)}$ for $\vep^{(j)}(t(s))$ for simplicity):
    \begin{align*}
        g_j'(y,s) := \begin{cases}
        \lambda_j^{-2}\iota_{\vep^{(j)}L}^*g_0 &\text{ for } y \in P\\
        \lambda_j^{-2}\varphi^* \iota_{\vep^{(j)}L}^*g_0 &\text{ for } y \in P^{(j)}_s \setminus P.
        \end{cases}
    \end{align*}	
    Up to pullback, this metric is simply $g_j'(\cdot, s) = \lambda_j^{-2} g_j(\cdot, t(s))$. Note that $\vep^{(j)}(t_j)^{-2}(\times \vep^{(j)})^*g_0 \rightarrow g_0$ locally uniformly in $C^\infty$. We therefore see that the metric $g_j'$ converges in $C^\infty_{\text{loc}}$ to $\iota_L^* g_0 = g_L$:
    \begin{align*}
        \lim_{j\rightarrow \infty} \lambda_j^{-2}\iota_{\vep^{(j)}L}^*g_0 &= (\iota_L)^*g_0, \quad\quad \lim_{j\rightarrow \infty} \lambda_j^{-2}\varphi^* \iota_{\vep^{(j)}L}^*g_0 = \varphi^* g_0.
    \end{align*}
Now, the Case 1 assumption along with (\ref{assumption on rho}) imply that, after passing to a subsequence, there exists a constant $C$ such that for all $j \in \mathbb{N}$ with $x'_j \in Q^{\pm}$:
	\begin{align*}
		\varepsilon^{(j)}(t_j)^{-1} \,\rho_{\varepsilon^{(j)}(t_j)}(x_j') &< C \,\implies\, \kp_{\varepsilon^{(j)}(t_j)}(r(x_j'))< C \varepsilon^{(j)}(t_j) < \varepsilon^{(j)}(t_j)^{\tau},
	\end{align*}
	implying that $y_j \in P^{(j)}_0$ lies in a compact region independent of $j$. Defining the time-independent weight function $\hat\rho: P_s^{(j)} \to \mathbb{R}^+$ as in Definition \ref{weight}, we may use (\ref{nontrivial condition}) and Assumption \ref{assumption on epsilon} to derive bounds on $\widetilde u^{(j)}$, $\widetilde \psi^{(j)}$ as follows:
 \begin{itemize}
     \item \quad$t^\mu \rho_{t^{-\frac{1}{m-2}}}(x)^\nu |u^{(j)}(x,t)| \leq 1$ for $x \in \stm \text{, }\,t\in [\Lambda,t_j)$
	\begin{align*}
 \begin{array}{ll}
      \displaystyle\implies |\widetilde u^{(j)}(y,s)| \leq \left(\frac{t_j}{t_j + \lambda_j^2s}\right)^\mu \left( \frac{\varepsilon^{(j)}(t_j)}{\rho_{t(s)^{-\frac{1}{m-2}}}(x(y))} \right)^\nu &\text{ for } s \in [\lambda_j^{-2}(\Lambda -t_j),\infty), \,y\in P^{(j)}_s,\\
      \implies |\widetilde u^{(j)}(y,s)| \leq C \hat\rho(y)^{-\nu}  &\text{ for } s \in [-\tfrac{1}{2}\lambda_j^{-2}t_j,0), \,y\in P^{(j)}_s.
 \end{array}
 \end{align*}
 \item \quad$\parallel \psi^{(j)}\parallel_{\mu,\nu+2,j} \,\rightarrow\, 0$
 \begin{align*}
 \begin{array}{ll}
      \displaystyle\implies |\widetilde \psi^{(j)}(y,s)| \left(\frac{t_j + \lambda_j^2s}{t_j}\right)^\mu \left( \frac{\rho_{t(s)^{-\frac{1}{m-2}}}(x(y,s))}{\varepsilon^{(j)}(t_j)} \right)^\nu \,\rightarrow\, 0 & \text{ uniformly on } P^{(j)}\\
      \implies |\widetilde \psi^{(j)}(y,s)| \,\rightarrow \,0 &\text{ uniformly on } P^{(j)}.
 \end{array}
	\end{align*}
 \end{itemize}
	Now we calculate the PDE that is satisfied by $\widetilde u^{(j)}$ and $\widetilde \psi^{(j)}$. We consider the tip and intermediate regions of $P^{(j)}$ separately.
	
	In the tip region, $\overline\kp_{\vep^{(j)}(t)}^{-1}\circ S^{(j)}_t = \text{Id}$. Then, by (\ref{linearisation}) and (\ref{linearised problem}):
	\begin{align*}
		&\partial_s \widetilde u^{(j)}(y,s) = t_j^\mu \lambda_j^{\nu + 2} \partial_s u^{(j)}(x(y,s),t(s)) = t_j^\mu \lambda_j^{\nu + 2} \left( \Delta_{g_j(t)}u^{(j)} + \psi^{(j)} \right) = \Delta_{g'_j(s))}\widetilde u^{(j)} + \widetilde \psi^{(j)}.
	\end{align*}
	In the intermediate region,
	\begin{align*}
		\partial_s \widetilde u^{(j)}&(y,s) 
        % =  t_j^\mu \lambda_j^{\nu + 2} \left[ \partial_t u^{(j)}(x,t) + \left(\frac{\dd\kp_{\vep^{(j)}}}{\dd r}(x)\right)^{-1}\frac{\dd u^{(j)}}{\dd r}(x,t)\frac{\dd\varepsilon^{(j)}}{\dd t}\left( - \frac{\dd\kp_{\vep^{(j)}}}{\dd\vep^{(j)}}(x)  + r(x)\right) \right]\\
		% &= t_j^\mu \lambda_j^{\nu + 2} \left[ \Delta_{g_j(t)}u^{(j)} + \psi^{(j)} - \left(\frac{\dd\kp_{\vep^{(j)}}}{\dd r}(x)\right)^{-1}\frac{\dd u^{(j)}}{\dd r}(x,t)\frac{\dd\varepsilon^{(j)}}{\dd t}\left( - \frac{\dd\kp_{\vep^{(j)}}}{\dd\vep^{(j)}}(x)  + r(x)\right) \right]\\  
		&= \Delta_{g'_j(s)} \widetilde u^{(j)} + \widetilde \psi^{(j)} - \lambda_j^2\left( \frac{\dd\kp_{\vep^{(j)}}}{\dd r}(x)\right)^{-1}\frac{\dd\varepsilon^{(j)}}{\dd t}\left( - \frac{\dd\kp_{\vep^{(j)}}}{\dd\vep^{(j)}}(x)  + r(x)\right)\frac{\dd \widetilde{u}^{(j)}}{\dd r}(y,s) .
	\end{align*}
    We note that $\left|\lambda_j^2\left( \frac{\dd\kp_{\vep^{(j)}}}{\dd r}(x)\right)^{-1}\frac{\dd\varepsilon^{(j)}}{\dd t}\left( - \frac{\dd \kp_{\vep^{(j)}}}{\dd\vep^{(j)}}(x)  + r(x)\right)\right| \leq C \cdot t(s)^{\frac{1}{m-2}} \cdot t_j^{-\frac{2}{m-2}}$, so that the coefficient of $\frac{\dd \widetilde{u}^{(j)}}{\dd r}(y,s)$ converges to 0 on compact spacetime regions as $j \to \infty$.
	Therefore, passing to a subsequence, we have the convergences $y_j \rightarrow y_\infty$, $(P^{(j)},g'_j) \rightarrow (L\times (-\infty,0),g_{L})$ locally smoothly, $\widetilde u^{(j)} \rightarrow \overline u$ in $C^{1,2}_{\text{loc}}$, $\widetilde\psi^{(j)} \rightarrow 0$ in $C^0_{\text{loc}}$, where $\overline u$ is an ancient solution to the heat equation \[ \partial_s \overline u(y,s) = \Delta \overline u (y,s), \quad (y,s) \in L \times (-\infty, 0),\] satisfying $|\overline u(\cdot, \tau)|\leq c\hat\rho^{-\nu}$. Finally, $\overline u$ is nontrivial since $\overline u(y_\infty, 0) =1$ by (\ref{nontrivial condition}). This contradicts Proposition \ref{Liouville thm tip}. 

 \ 
 
		\noindent{\it Case 2.} $\lim_{j\to\infty}t_{j}^{\frac{1}{m-2}}\rho_{t_{j}^{-\frac{1}{m-2}}}(x_{j}')=\infty$ and $\lim_{j\to\infty}\rho_{t_{j}^{-\frac{1}{m-2}}}(x_{j}') = 0$.
	
	In this case, \[\rho_{t_j^{-\frac{1}{m-2}}}(x) = \begin{cases} t_j^{-\frac{1}{m-2}} r(x) \quad &\text{near } R_1\\ R_2 \quad &\text{near } R_2,\end{cases}\]
	and so the Case 2 assumption implies that $x_j' \in \Sigma\times (R_1,(1-\hbar)R_2) \subset Q$, after passing to a subsequence. Define the following rescaled region $Q_s^{(j)}$ and rescalings of $u^{(j)}$ and $\psi^{(j)}$:
	\begin{align*}
		&\lambda_j := \kp_{\varepsilon^{(j)}(t_j)}(r(x_j')) = \rho_{\varepsilon^{(j)}(t_j)}(x_j') \rightarrow 0.\\
		&Q^{(j)}_s := \Sigma \times (\lambda_j^{-1} \varepsilon^{(j)}(t_j + \lambda_j^2s) R_1,\, \lambda_j^{-1} (1-\hbar)R_2), \quad	Q^{(j)} := \left\{(y,s) \,:\, y \in Q_s^{(j)},\, s \in [-\tfrac{1}{2}\lambda_j^{-2}t_j,0) \right\}\\
		&S^{(j)}: Q^{(j)} \rightarrow \Sigma \times (0,R_2), \quad S^{(j)}(\sigma,r) := (\sigma, \lambda_j r)\\
		&\widetilde u^{(j)}: Q^{(j)} \rightarrow \mathbb{R}, \quad \widetilde u^{(j)}(y,s) := t_j^\mu \lambda_j^\nu u^{(j)} \left( \overline \kp_{\vep^{(j)}}^{-1} \circ S^{(j)}(y), t_j  + \lambda_j^2 s\right)\\
		&\widetilde\psi^{(j)}: Q^{(j)}\rightarrow \mathbb{R}, \quad \widetilde \psi^{(j)}(y, s) := t_j^\mu \lambda_j^{\nu+2} \psi^{(j)}(\overline \kp_{\vep^{(j)}}^{-1} \circ S^{(j)}(y), t_j + \lambda_j^2s). 
	\end{align*}
	Define $t(s)$, $x(y,s)$, $y_j$ and $g'_j$ as in Case 1. Using (\ref{nontrivial condition}) we may derive bounds for $\widetilde u^{(j)}$:
	\begin{align*}
 \begin{array}{ll}
      t^\mu \rho_{t^{-\frac{1}{m-2}}}(x)^\nu |u^{(j)}(x,t)| \leq 1 & \text{ for } t \in [\Lambda,t_j], \\
      \displaystyle\implies |\widetilde u^{(j)}(y,s)| \leq \left( \frac{t_j}{t_j + \lambda_j^2 s}\right)^\mu r(y)^{-\nu} & \text{ for } s \in [\lambda^{-2}_j(\Lambda-t_j), 0],\\
      \implies |\widetilde u^{(j)}(y,s)| \leq C \, r(y)^{-\nu} & \text{ for } s \in [-\tfrac{1}{2} \lambda^{-2}_jt_j, 0].
 \end{array}
	\end{align*}
	The linear PDE satisfied by $\widetilde u^{(j)}$ is given by (\ref{linearisation}) and (\ref{linearised problem}) in the same way as for Case 1 in the intermediate region:
	\begin{align*}\pl_s u^{(j)} \, &=\,\Delta_{g'_j(s)} \widetilde u^{(j)} + \widetilde \psi^{(j)} + \lambda_j^2\left( \frac{\dd\kp_{\vep^{(j)}}}{\dd r}(x)\right)^{-1}\frac{\dd\varepsilon^{(j)}}{\dd t}\frac{\dd\kp_{\vep^{(j)}}}{\dd\vep^{(j)}}(x)\frac{\dd \widetilde{u}^{(j)}}{\dd r}(y,s).
	\end{align*}
	As in Case 1, after passing to a subsequence we have the convergences $\widetilde u^{(j)} \rightarrow \overline u$, $\widetilde \psi^{(j)} \rightarrow 0$, $Q^{(j)} \to \left(\Sigma \times (0,\infty)\right) \times (\infty, 0)$, $g'_j \rightarrow \overline g$, where $\overline g$ is the metric on $\Sigma \times (0, \infty)$ corresponding to the flat metric on two punctured copies of $\mathbb{R}^m$, and $\overline u$ satisfies
	\[ |\overline u(y,s)| \leq c \:r(y)^{-\nu}, \quad \frac{\partial \overline u}{\partial s} = \Delta_{\overline g}\overline u\;\; \text{ for } s \in (-\infty,0).\] 
	Furthermore, by interior parabolic Schauder estimates, $\overline u$ satisfies $|\nabla^{k} \overline u(y, s)|\leq C|y|^{-\nu-k}$, for $k\in\{0,1,2\}$. Finally, to show that $\overline u$ is nontrivial, note that
	\[ r(y_j) = \lambda_j^{-1} \rho_{t_j^{-\frac{1}{m-2}}}(x_j') = \frac{\rho_{t_j^{-\frac{1}{m-2}}}(x_j')}{\rho_{\varepsilon^{(j)}(t_j)}(x_j')} \in I :=  \left(\frac{1}{C_{\vep}},C_{\vep}\right).\]
	So passing to a subsequence if necessary, $y_j \rightarrow \overline y \in \Sigma \times I$, and by (\ref{nontrivial condition}), $\overline u (\overline y, 0) \geq \frac{1}{2}.$ This contradicts Proposition \ref{liouville intermediate}.

 \ 
 
    \noindent{\it Case 3.}  $\lim_{j\to\infty}\rho_{t_{j}^{-\frac{1}{m-2}}}(x_{j}') > 0$.
    
    In this case, taking $C$ as in Assumption \ref{assumption on epsilon}, we make the definitions:
    \begin{align*}
        &Q^{(j)} := \left\{ (y,s) \,:\, s \in [-\tfrac{1}{2}t_j,\infty), \, y\in \Sigma  \times [C(t_j + s)^{-\frac{1}{m-2}}R_1,R_2] \right\}\\
        &X^{(j)} := Q^{(j)} \cup \left(\out \times [-\tfrac{1}{2}t_j,\infty)\right) \\
        &\widetilde u^{(j)}: X^{(j)} \rightarrow \mathbb{R}, \quad\quad \widetilde u^{(j)}(y,s) \,:=\,
        \begin{cases}
            t_j^\mu u^{(j)}(\overline \kp_{\vep^{(j)}}^{-1}(y)\,, t_j + s) & \text{ on } Q^{(j)} \\
            t_j^\mu u^{(j)}(y,\, t_j + s) & \text{ on } \out \times [\Lambda - t_j, \infty),
        \end{cases}\\
        &\tilde \psi^{(j)}:X^{(j)}\rightarrow \mathbb{R}, \quad\quad \widetilde \psi^{(j)}(y,s) \,:=\,
        \begin{cases}
            t_j^\mu \psi^{(j)}(\overline \kp_{\vep^{(j)}}^{-1}(y)\,, t_j + s) & \text{ on } Q^{(j)}\\
            t_j^\mu \psi^{(j)}(y,\, t_j + s) & \text{ on } \out \times [\Lambda - t_j, \infty),
        \end{cases}
    \end{align*}
    We equip $X^{(j)}$ with the metric $g_j'(\cdot,s) := g_j(\cdot, t_j + s)$, so that we have the convergence $(X^{(j)},g'_j) \rightarrow (X\setminus\iota^{-1}(\xsin), \iota^* g)$ in $C^\infty_\text{loc}$. As before we may derive bounds on $\widetilde u^{(j)}$ and $\widetilde \psi^{(j)}$:
    \begin{align*}
    \begin{array}{ll}
         |\widetilde u^{(j)}(y,s)| \leq C\,\rho_{t_j^{-\frac{1}{m-2}}}(\overline \kp_{\vep^{(j)}}^{-1}(y))^{-\nu}& \text{ for } s \in [-\tfrac{1}{2}t_j,0), \\
         |\widetilde \psi^{(j)}(y,s)| \to 0 & \text{ uniformly on compact subsets}.
    \end{array}
    \end{align*}
    After passing to a subsequence, we have convergences $\widetilde u^{(j)} \rightarrow \overline u$, $\widetilde \psi^{(j)} \rightarrow \overline \psi$, where $\overline u$ satisfies $\overline u(y,s) \leq c \rho(y)$ for $\rho:X \to [0, \infty)$ as in section \ref{sec-Xliouville}. The Case 3 assumption implies that $\rho_{t_j^{-\frac{1}{m-2}}}(x_j')\rightarrow P > 0$ so that $x_\infty \in X\setminus\iota^{-1}(\xsin)$, and (\ref{nontrivial condition}) implies that $\overline u(x_\infty,0) \neq 0$. It follows from (\ref{linearised problem}) that $\pl_s \overline u = \Delta_{\iota^* g} \overline u$, and (\ref{uniform sup norm orthogonality assumption}) implies that $\int_{X_1 \setminus x_\star^-} \overline u\,\, dV_{\iota^* g} = 0$ and $\int_{X_2 \setminus x_\star^+} \overline u\,\, dV_{\iota^* g} = 0$. Finally, by interior parabolic Schauder estimates (see for example, \cite{Behrndt2011}*{Proposition~7.3}), $\overline u$ satisfies $|\nabla^{k} \overline u(y, \tau)|\leq C\rho(y)^{-\nu-k}$, for $k\in\{0,1,2\}$. Proposition \ref{Liouville outer} gives a contradiction.
\end{proof}

%%%%%%%%%%%%%%%%%%%%%%%%%%%%%%%%%%%%%%%%%%%%%%%%%%%%%%%%%%%%%%%%
\section{Existence Theory} \label{sec-exist-torus}

We now focus on the primary goal of this paper: find $\Lambda$, $\varepsilon$ and $u$ solving (\ref{single_eq}) under Assumption \ref{assumption torus}.

\subsection{Estimates in the Torus Case} We will require the following estimates for the induced metric $g^\vep$, its volume form $\dd V_{g^{\vep}}$, the nontrivial approximate kernel element $w^\vep$, and the Lagrangian angle. Throughout we use Assumption \ref{assumption on epsilon} for estimating the time derivative and H\"older derivative of $\vep(t)$, and for convenience use the notation $\slashed\partial^\alpha_{t_1,t_2} f := \frac{f(t_1) - f(t_2)}{|t_1-t_2|^\alpha}$ for the H\"older quotient.

By a straightforward computation, we have the following estimate for the induced metric and the volume.
\begin{lemma}\label{lem-volumeestimates}
    Let $g_{0}$ be the Euclidean metric on $\mathbb{R}^{m}$, and $g_{\coneC}$ be the cone metric on $\Sigma\times(0, \infty)$. Under Assumptions \ref{assumption on epsilon} \and \ref{assumption torus}, the induced metric $g^{\vep}$ on $\ul{N}$ satisfies
    \begin{align}
    g^{\vep} &= g_{0} \quad\mbox{on $X_{b}^{o}\cup(\Sigma\times[2\vep^{\tau}, R_{2}))$},\;b=1,2,\\
        |\nabla^{k}(g^{\vep} - g_{\coneC})|_{g_{\coneC}}(\sigma, \ir) &= \begin{cases}
            O(\vep^{2(1-\tau)m - \tau k}),\;\;&(\sigma,\ir)\in\Sigma\times(\vep^{\tau}, 2\vep^{\tau}),\\
            O(\ir^{-2m-k}\vep^{2m}),\;\;&(\sigma,\ir)\in\Sigma\times(\vep R_{1}, \vep^{\tau}),
        \end{cases}, \; k=0, 1, 2,\\
        g^{\vep} &= \vep^{2}g_{L}\quad\mbox{on $P$}.
    \end{align}
    The volume form $\dd V_{g^{\vep}}$ on the tip region $P$ satisfies (for $t_1, t_2 \in [t, 2t]$,  $0 < |t_1-t_2| < t^{-\frac{1}{m-2}}$):
    \begin{align}
        \dd V_{g^{\vep}} = \vep^m \:\dd V_L, \quad \partial_{\vep} \:\dd V_{g^{\vep}} &= O(\vep^{m-1}) \:\dd V_L, \quad \partial_{t}\: \dd V_{g^{\vep}} = O(\vep(t)^{2m-2})\:\dd V_L,\\
        |\slashed \partial^\alpha_{t_1,t_2}\:\dd V_{g^{\vep}}(t)| &= O(\vep(t)^{2m-2\alpha}) \:\dd V_L,\\
        |\slashed \partial^\alpha_{t_1,t_2} \partial_t \:\dd V_{g^{\vep}}(t)|&= O(\vep(t)^{2m-2-2\alpha})\: \dd V_L.
    \end{align}
    The volume form $\dd V_{g^{\vep}}$ on $\Sigma\times(\vep R_{1}, 2\vep^{\tau})$ satisfies  (for $t_1, t_2 \in [t, 2t]$, $|t_1 - t_2| < t^{-\frac{2}{m-2}}$):
    \begin{align}
        \dd V_{g^{\vep}} &= \begin{cases}
            ( 1 + O(\vep^{m(1-\tau)}))\:\dd V_{\coneC},\;\;&(\sigma,\ir)\in\Sigma\times(\vep^{\tau}, 2\vep^{\tau}),\\
            ( 1 + O(\ir^{-m}\vep^{m}))\:\dd V_{\coneC},\;\;&(\sigma,\ir)\in\Sigma\times(\vep R_{1}, \vep^{\tau}),
        \end{cases}\\
        \partial_{\vep}\dd V_{g^{\vep}} &= \begin{cases}
            O(\vep^{m(1-\tau)-1})\:\dd V_{\coneC},\;\;&(\sigma,\ir)\in\Sigma\times(\vep^{\tau}, 2\vep^{\tau}),\\
            O(\ir^{-m}\vep^{m-1})\:\dd V_{\coneC},\;\;&(\sigma,\ir)\in\Sigma\times(\vep R_{1}, \vep^{\tau}),
        \end{cases}\\
        |\slashed \partial^\alpha_{t_1,t_2}\:\dd V_{g^{\vep}}(t)| &= \begin{cases}
            O(\vep^{m(2-\tau)-2\alpha})\:\dd V_{\coneC},\;\;&(\sigma,\ir)\in\Sigma\times(\vep^{\tau}, 2\vep^{\tau}),\\
            O(\ir^{-m}\vep^{2m-2\alpha})\:\dd V_{\coneC},\;\;&(\sigma,\ir)\in\Sigma\times(\vep R_{1}, \vep^{\tau}),
        \end{cases}\\
        |\slashed \partial^\alpha_{t_1,t_2} \partial_t \dd V_{g^{\vep}}(t)| &= \begin{cases}
            O(\vep^{m(2-\tau)-2-2\alpha})\:\dd V_{\coneC},\;\;&(\sigma,\ir)\in\Sigma\times(\vep^{\tau}, 2\vep^{\tau}),\\
            O(\ir^{-m}\vep^{2m-2-2\alpha})\:\dd V_{\coneC},\;\;&(\sigma,\ir)\in\Sigma\times(\vep R_{1}, \vep^{\tau}).
        \end{cases}
    \end{align}
\end{lemma}

\begin{lemma}\label{lem-approxkernelestimates}
Under Assumptions \ref{assumption on epsilon} and \ref{assumption torus}, the function $w^\vep$ satisfies (for $t_1, t_2 \in [t, 2t]$, $|t_1 - t_2| < t^{-\frac{2}{m-2}}$):
\begin{align*}
    |w^\vep| &\leq 1,\\
    |\partial_t w^\vep| &=  
    \begin{cases}
        O(\vep^{m-2}) &\quad\text{on } \overline\kappa_\vep^{-1}(\Sigma \times (\vep R_1, \vep^{\tau})),\\
    0 &\quad\text{otherwise},
    \end{cases}\\
    |\slashed \partial^\alpha_{t_1,t_2} \partial_t w^\vep| &= 
    \begin{cases}
        O(\vep^{m-2-2\alpha}) &\quad\text{on } \overline\kappa_\vep^{-1}(\Sigma \times (\vep R_1, \vep^{\tau}))\\
        0 &\quad\text{otherwise},
    \end{cases}\\
    |\Delta_{g^\vep} w^\vep| &\leq 
    \begin{cases}
        \kappa_{\vep}(r)^{-m}\vep^{m-2} &\quad\text{on } \overline\kappa_\vep^{-1}(\Sigma \times (\vep R_1, \vep^{\tau})),\\
        0 &\quad\text{otherwise},
    \end{cases}\\
    |\slashed\partial^\alpha_{t_1,t_2}\Delta_{g^\vep} w^\vep| &\leq 
    \begin{cases}
        \kappa_{\vep}(r)^{-m}\vep^{2m-2-2\alpha} &\quad\text{on } \overline\kappa_\vep^{-1}(\Sigma \times (\vep R_1, \vep^{\tau})),\\
        0 &\quad\text{otherwise}.
    \end{cases}
\end{align*}
\end{lemma}

\begin{proof}
    The spatial estimates of $w^{\vep}$ follow from \cite{Joyce2004SLCS3}*{Proposition 7.3}. For the time derivative, note that $\partial_{t}w^{\vep} = 0$ on $\out_{b}$ and $P$. On $Q$, we compute
\begin{align*}
    \partial_{t}\kappa_{\vep} &= \left[1-\chi\left(\frac{r-R_{1}}{\hbar R_1}\right)\right]\vep'r = O(\vep'),\\
    \partial_{t}\ul{\alpha} &= (\alpha_{L}\big|_{\ul{0}}\circ\varphi)(\sigma, \vep^{-1}\kappa_{\vep}(r))[\vep^{-2}\vep'\kappa_{\vep}(r) + \vep^{-1}\partial_{t}\kappa_{\vep}]\\
    &= O(\vep^{(1-\tau)(m-3)}\vep^{-1}\vep')\quad\text{for } r\in \kappa_{\vep}^{-1}(\vep R_{1}, \vep^{\tau}).
\end{align*}
The result now follows from a calculation. The H\"older derivative estimates follow similarly, using the fact that $|t_1 - t_2| < t^{-\frac{2}{m-2}} \implies |\slashed \partial^\alpha_{t_1,t_2}f| \leq C\cdot |\partial_t f (c)| \vep^{2-2\alpha}$ for some $c \in [t_1,t_2]$.
\end{proof}

\begin{lemma}\label{lem: angle and mean curvature in torus case}
    Under Assumption \ref{assumption torus}, for any $\tau\in(0, \frac{1}{2})$ and $k \in \{0,1,2\}$, we have
    \begin{align}
        |\nabla^k\theta_{N^{\vep}}(x)| = \begin{cases}
            O(\vep^{m(1-\tau)-k\tau})& x = (\sigma,\mathfrak{r}),\;\mathfrak{r}\in(\vep^{\tau}, 2\vep^{\tau}),\\
        0, & \mbox{otherwise,}
        \end{cases}
    \end{align}
    where $|\cdot|$ is computed using the pullback metric $g^\vep$.
\end{lemma}
\begin{proof}
    The proof follows as in \cite{Joyce2004SLCS3}*{Proposition 6.4}, with the improvements coming from Assumption \ref{assumption torus}.%, $\mathcal{A} = 0$ and $\Upsilon^*g = g_0$.
\end{proof}

\subsection{Weighted Parabolic H\"older Spaces} We define suitable H\"older spaces for our differential operators. Given $\Lambda>0$, $(\mu, \nu)\in\mathbb{R}^{2}$, $\alpha\in(0, \frac{1}{2})$, and a time-dependent tensor $T$ on $\ul{N}$, and letting $\text{inj}(g^\vep)$ denote the injectivity radius of the induced metric $g^\vep = (\iota^{\vep})^{*}g$, we define:
\begin{align}
    &\|T\|_{\mu, \nu, \Lambda} := \sup_{(x, t)\in\ul{N}\times[\Lambda, \infty)}t^{\mu}\rho^{\nu}_{t^{\frac{-1}{m-2}}}(x)|T|_{g^\vep},\\
    &[T]_{\mu, \nu, \alpha, \Lambda} := \sup_{t\in[\Lambda, \infty)}\sup_{\substack{x_{1}, x_{2}\in\ul{N}\\d_{g^\vep}(x_1, x_2)<\min\{\text{inj}(g^\vep), 1\} }} t^{\mu}\min\{\rho^{\nu+2\alpha}_{t^{\frac{-1}{m-2}}}(x_1), \rho^{\nu+2\alpha}_{t^{\frac{-1}{m-2}}}(x_2)\}\frac{|T(x_{1}, t) - T(x_{2}, t)|_{g^\vep}}{d_{g^\vep}(x_{1}, x_{2})^{2\alpha}},\\
    &\langle T \rangle_{\mu, \nu, \alpha, \Lambda} := \sup_{x\in\ul{N}}\;\sup_{t>\Lambda}\sup_{\substack{t_{1}, t_{2}\in[t, 2t],\\0<|t_{1} - t_{2}|<t^{\frac{-2}{m-2}}}}t^{\mu}\rho^{\nu+2\alpha}_{t^{\frac{-1}{m-2}}}(x)\frac{|T(x, t_{1}) - T(x, t_{2})|_{g^\vep}}{|t_1 - t_2|^{\alpha}}.
\end{align}
Here, the norms are computed by the induced metric on the corresponding tensor bundles, and the difference $T(x_{1}, t) - T(x_{2}, t)$ is understood using the parallel transport along the unique shortest geodesic between $x_{1}$ and $x_{2}$ to compare the values.
\begin{definition}\label{def-parabolicholdernorms}
Define a weighted parabolic H\"older norm for tensors $T$ on $\ul{N}$ by
\begin{align}
    \|T\|_{P^{0, 0, \alpha}_{\mu, \nu, \Lambda}} := \|T\|_{\mu, \nu, \Lambda} + [T]_{\mu, \nu, \alpha, \Lambda} + \langle T \rangle_{\mu, \nu, \alpha, \Lambda}.
\end{align}
The weighted parabolic H\"older spaces $P^{l, k, \alpha}_{\mu, \nu, \Lambda}$ are then defined to be the space of functions $u:\ul{N}\times[\Lambda, \infty)\to\mathbb{R}$ such that the norm
\begin{align}
    \|u\|_{P^{l, k, \alpha}_{\mu, \nu, \Lambda}} := \sum_{i=0}^{l}\|\partial_{t}^{i}u\|_{P^{0, 0, \alpha}_{\mu, \nu+2i, \Lambda}} + \sum_{j=0}^{k}\|\nabla^{j}u\|_{P^{0, 0, \alpha}_{\mu, \nu+j, \Lambda}}
\end{align}
is finite. Analogously, we define the weighted parabolic H\"older norm $\|\cdot\|_{C^{0,\alpha}_{\zeta, \Lambda}}$ (and corresponding Banach space $C^{0,\alpha}_{\zeta, \Lambda}$) for functions $h:[\Lambda, \infty) \to \mathbb{R}$:
\begin{equation}
    \|h\|_{C^{0,\alpha}_{\zeta, \Lambda}} \, := \, \sup_{t \in [\Lambda, \infty)} t^\zeta \, |h(t)| \, + \, \sup_{t \in[\Lambda, \infty)} \sup_{\substack{t_{1}, t_{2}\in[t, 2t],\\0<|t_{1} - t_{2}|<t^{\frac{-2}{m-2}}}} t^\zeta \, \frac{|h(t_1) - h(t_2)|}{|t_1-t_2|^\alpha}.
\end{equation}
\end{definition}

In order to apply the Schauder fixed point theorem to solve our nonlinear PDE for functions belonging to these spaces, we will require the following standard compact embedding theorem.

\begin{lemma}\label{lem-compactness}
    For $\mu'<\mu$, $\alpha'<\alpha$, $\zeta' < \zeta$, and $\Lambda > 1$, the inclusions
    \begin{align*}
        C^{0, \alpha}_{\zeta, \Lambda} \hookrightarrow C^{0, \alpha'}_{\zeta', \Lambda}, \quad P^{l,k,\alpha}_{\mu, \nu, \Lambda} \hookrightarrow P^{l, k, \alpha'}_{\mu', \nu, \Lambda}
    \end{align*}
    are compact.
\end{lemma}

\subsection{A Priori Estimates and Existence Theory for the Linearised Operator}

We now proceed with the linear theory for our linearised operator, which will be viewed as a bounded operator on the weighted H\"older spaces of Definition \ref{def-parabolicholdernorms}. The main result is Theorem \ref{linear theory}. Since the linearised operator has a non-trivial kernel, we prove our estimates and existence theory on the orthogonal complement of the approximate kernel, which will be denoted by:
\begin{align}
    \langle 1, w^{\vep}\rangle^{\perp} := \left\{u\in C^{0}([\Lambda, \infty), L^{2}(\ul{N}, g^{\vep}))\::\:\int_{\ul{N}}u\cdot w^{\vep}\:\dd V_{g^{\vep}} = \int_{\ul{N}}u\cdot 1\:\dd V_{g^{\vep}} = 0,\;\forall\:t\in[\Lambda, \infty)\right\}.
\end{align} 
We will first consider the simpler case of the heat operator. It is clear from the definition that we have the following.

\begin{lemma}
Let $\mu>0$, $\nu\in(0, m-2)$, $\alpha\in(0, 1/2)$. The linear operator $\partial_{t} - \mathcal{L}^{\vep}_{\ul{0}}:C_{c}^{\infty}(\ul{N}\times(\Lambda, \infty))\to C_{c}^{\infty}(\ul{N}\times(\Lambda, \infty))$ extends to a bounded operator
\begin{align*}
    \partial_{t} - \Delta_{g^{\vep}}: P^{1, 2, \alpha}_{\mu, \nu, \Lambda}\longrightarrow P^{0, 0, \alpha}_{\mu, \nu+2, \Lambda}.
\end{align*}
\end{lemma}

We first note that our a priori estimate for the heat operator implies the following weighted Schauder estimate.

\begin{corollary}\label{weighted schauder estimate}
    Let $\mu>0$, $\nu\in(0, m-2)$, $\alpha\in(0, \frac{1}{2})$ and $\Lambda>0$. There exists a constant $C>0$ such that if $u\in P^{1, 2, \alpha}_{\mu, \nu, \Lambda} \cap \langle 1, w^{\vep}\rangle^{\perp}$ and $\psi\in P^{0, 0, \alpha}_{\mu, \nu+2, \Lambda}$ solves the Cauchy problem
    \begin{align}
        \begin{cases}
		\partial_{t}u - \Delta_{g^\vep}u = \psi \quad t\in[\Lambda, \infty),\\
		u(x, \Lambda) = 0,\quad x\in\ul{N},
	\end{cases}
    \end{align}
then
\begin{align*}
    \|u\|_{P^{1, 2, \alpha}_{\mu, \nu, \Lambda}}\leq C\:\|\psi\|_{P^{0, 0, \alpha}_{\mu, \nu+2, \Lambda}}.
\end{align*}
\end{corollary}
\begin{proof}
    By the scaling property of the induced metric $g^{\vep}$ and the standard interior Schauder estimate, we have
    \begin{align*}
        \|u\|_{P^{1, 2, \alpha}_{\mu, \nu, \Lambda}}\leq C\left(\|\psi\|_{P^{0, 0, \alpha}_{\mu, \nu+2, \Lambda}}+\|u\|_{\mu, \nu, \Lambda}\right).
    \end{align*}
    Since $u \in \langle 1, w^{\vep}\rangle^{\perp}$, we may apply Theorem \ref{weighted sup norm estimate} to bound $\|u\|_{\mu, \nu, \Lambda}$ in terms of  $\|\psi\|_{P^{0, 0, \alpha}_{\mu, \nu+2, \Lambda}}$.
\end{proof}

Supposing now that $u \in P^{1, 2, \alpha}_{\mu, \nu, \Lambda} \cap \langle 1, w^{\vep}\rangle^{\perp}$ satisfies
\begin{align}
        \begin{cases}
		\partial_{t}u - \Delta_{g^\vep}u = \psi + a(t) + b(t)w^\vep \quad t\in[\Lambda, \infty),\\
		u(x, \Lambda) = 0,\quad x\in\ul{N},
	\end{cases}\label{eq-kernelequation}
\end{align}
then by the Schauder estimate above we have
\begin{align}
	\|u\|_{P^{1, 2, \alpha}_{\mu, \nu, \Lambda}}\,&\leq\, C\|\psi + a(t) + b(t)w^{\vep}\|_{P^{0, 0, \alpha}_{\mu, \nu+2, \Lambda}} \notag\\
 &\leq\, C\|\psi\|_{P^{0, 0, \alpha}_{\mu, \nu+2, \Lambda}} + C\|a(t)\|_{P^{0, 0, \alpha}_{\mu, \nu+2, \Lambda}} + C\|b(t)w^{\vep}\|_{P^{0, 0, \alpha}_{\mu, \nu+2, \Lambda}}.\label{eq-abestimate}
\end{align}
It is therefore important to estimate $a(t)$ and $b(t)$ in terms of $u$ and $\psi$.
\begin{lemma}\label{lem-a and b estimate}
Consider $u \in P^{1, 2, \alpha}_{\mu, \nu, \Lambda}\cap \langle 1, w^{\vep}\rangle^{\perp}$ and $\psi \in P^{0, 0, \alpha}_{\mu, \nu+2, \Lambda}$ satisfying (\ref{eq-kernelequation}). Then:
\begin{align*}
	&a(t) = \frac{1}{{\rm Vol}(N^{\vep})}\int_{\ul{N}}\left(\partial_t u - \psi\right)\:\dd V_{g^{\vep}}, \quad b(t) = \frac{1}{\|w^{\vep}\|^{2}_{L^{2}}}\int_{\ul{N}}\left(\partial_{t}u - \Delta_{g^{\vep}_{\ul{0}}}u - \psi\right)\cdot w^{\vep}\:\dd V_{g^{\vep}},
\end{align*}
and $a(t)$ and $b(t)$ satisfy the estimates
\begin{align*}
		\|a(t)\|_{P^{0, 0, \alpha}_{\mu, \nu+2, \Lambda}} &\leq C\|\psi\|_{P^{0, 0, \alpha}_{\mu, \nu+2, \Lambda}} + C \Lambda^{-\frac{2m-2-2\alpha -\nu\tau}{m-2}}\|u\|_{P^{1, 2, \alpha}_{\mu, \nu, \Lambda}},\\
  \|b(t)w^{\vep}\|_{P^{0, 0, \alpha}_{\mu, \nu+2, \Lambda}} &\leq C\|\psi\|_{P^{0, 0, \alpha}_{\mu, \nu+2, \Lambda}} + C \Lambda^{-\frac{m-2-2\alpha\tau -\nu\tau}{m-2}}\|u\|_{P^{1, 2, \alpha}_{\mu, \nu, \Lambda}}.
	\end{align*}
\end{lemma}

\begin{proof} The formulae for $a(t)$, $b(t)$ are obtained by integrating the differential equation against the elements of the approximate kernel $\{1, w^\vep\}$, and using the orthogonality conditions.

For the estimates, recall that by Assumption \ref{assumption on epsilon}, 
\begin{align*}
    |t_1 - t_2| < t^{-\frac{2}{m-2}} \implies |\slashed \partial^\alpha_{t_1,t_2}f| \leq C\cdot |\partial_t f (c)| \vep^{2-2\alpha}
\end{align*}
for some $c \in [t_1,t_2]$, where for convenience we use the notation $\slashed\partial^\alpha_{t_1,t_2} f := \frac{f(t_1) - f(t_2)}{|t_1-t_2|^\alpha}$ for the H\"older quotient. Differentiating the orthogonality condition and using the estimates on the volume form from Lemma \ref{lem-volumeestimates} yields 
\begin{align*}
&t^\mu \rho^{\nu + 2}_{t^{-\frac{1}{m-2}}}(x)|a(t)|
\leq C \cdot \sup_{t \in [\Lambda, \infty)} \left( \|u\|_{P^{1, 2, \alpha}_{\mu, \nu, \Lambda}} \int_{\ul N} \rho^{-\nu} |\partial_t (\dd V_{g^\vep})| + \|\psi\|_{P^{0, 0, \alpha}_{\mu, \nu+2, \Lambda}} \int_{\ul N} \rho^{-\nu -2} \:\dd V_{g^\vep} \right)\\
		%&\hspace{2.1cm}\qquad\leq C \cdot \sup_{t \in [\Lambda, \infty)} \bigg\{ \|u\|_{P^{1, 2, \alpha}_{\mu, \nu, \Lambda}}\left( \int_P \rho^{-\nu} \vep^{2m - 2}\:\dd V_L + \int_{\vep R_1}^{2\vep^\tau} \ir^{-\nu-1}\vep^{2m-2} \,\dd\ir \right) \\
		%&\hspace{2.1cm}\quad\qquad\qquad\quad\quad + \|\psi\|_{P^{0, 0, \alpha}_{\mu, \nu+2, \Lambda}}\left( \int_P \rho^{-\nu - 2} \vep^m \:\dd V_L +\int_{\vep R_1}^{R_2}\ir^{-\nu+m-3} \,\dd\ir + \int_{X^o} \dd V_{g^\vep} \right)\bigg\}\\
		&\hspace{2.1cm}\qquad\leq C \cdot \sup_{t \in [\Lambda, \infty)} \left( \|u\|_{P^{1, 2, \alpha}_{\mu, \nu, \Lambda}} \cdot \vep^{2m-2-\tau\nu}+ \|\psi\|_{P^{0, 0, \alpha}_{\mu, \nu+2, \Lambda}} \right) ~, \\
&t^\mu \rho^{\nu + 2+2\alpha}_{t^{-\frac{1}{m-2}}}(x)|\slashed\partial^\alpha_{t_1,t_2} a(t)|
\leq C \cdot \sup_{t \in [\Lambda, \infty)} \bigg(
\|u\|_{P^{1, 2, \alpha}_{\mu, \nu, \Lambda}}\int_{\ul N} \rho^{-\nu-2\alpha} |\partial_t (\dd V_{g^\vep})| + \rho^{-\nu} |\slashed \partial^\alpha_{t_1,t_2} \partial_t(\dd V_{g^\vep})|\\
    &\hspace{6cm}\quad\quad +\, \|\psi\|_{P^{0, 0, \alpha}_{\mu, \nu+2, \Lambda}} \int_{\ul N} \rho^{-\nu - 2 - 2\alpha}\: \dd V_{g^\vep} + \rho^{-\nu-2}  |\slashed \partial^\alpha_{t_1,t_2} (\dd V_{g^\vep})| \bigg)\\
    &\hspace{3cm}\qquad\leq C \cdot \sup_{t \in [\Lambda, \infty)} \left( \|u\|_{P^{1, 2, \alpha}_{\mu, \nu, \Lambda}}\: \vep^{2m-2-2\alpha-\tau\nu} + \|\psi\|_{P^{0, 0, \alpha}_{\mu, \nu+2, \Lambda}} \right) ~,
	\end{align*}
 which implies the estimate for $a(t)$. The estimate for $b(t)$ follows analogously, using the estimates from Lemma \ref{lem-approxkernelestimates} and the fact that $w^\vep$, ${\parallel}w^\vep{\parallel}_{L^2}$ are uniformly bounded.
\end{proof}

Finally, to extend the above estimates from the heat operator to our linearised operator $\mathcal{L}^\vep_{\ul 0}$, we require the following estimate on the difference between the Laplacian and the linearised operator:
\begin{lemma}\label{lem-lower order term estimate} Given $\tau<\frac{1}{m+2}$, we have the decomposition $\mathcal{L}^{\vep}_{\ul{0}} = \Delta_{g^{\vep}} + \mathcal{P}^{\vep}_{\ul{0}}$, where $\mathcal{P}^{\vep}_{\ul{0}}$ is a first order differential operator satisfying
 \begin{align*}
     |\mathcal{P}^{\vep}_{\ul{0}}[u]|\leq \begin{cases}
     C\:\vep(t)^{m-1}|\dd u|_{g^{\vep}}\quad&\mbox{on}\quad P_{j}\cup Q^{\pm}_{j},\; t\geq\Lambda,\\
     0,\quad&\mbox{otherwise}.
     \end{cases}
 \end{align*}
 In particular, there exists $C(\Lambda)>0$ with $\lim_{\Lambda\to \infty}C(\Lambda) = 0$ such that
 \begin{align*}
		\|\mathcal{P}^{\vep}_{\ul{0}}[u]\|_{P^{0, 0, \alpha}_{\mu, \nu+2, \Lambda}}\leq C(\Lambda)\|u\|_{P^{1, 2, \alpha}_{\mu, \nu, \Lambda}}.
\end{align*}
As a result, $\mathcal{L}^{\vep}_{\ul{0}}$ extends to a bounded operator $\mathcal{L}^{\vep}_{\ul{0}}:P^{1, 2, \alpha}_{\mu, \nu, \Lambda}\to P^{0, 0, \alpha}_{\mu, \nu+2, \Lambda}$.
\end{lemma}

\begin{proof}
By Proposition~\ref{prop: linearised operator} we have
    \begin{align*}
    \mathcal{L}^{\vep}_{\ul{0}}[u] = \Delta_{g^{\vep}}u-\langle\nabla\theta_{N^{\vep}}, \widehat{V}_{\ul{0}}(\dd u)\rangle_{g^{\vep}} + S^{\vep}[u],
\end{align*}
where $S^{\vep}[u]$ is a first order linear differential operator defined by
\begin{align*}
    S^{\vep}[u] &= \begin{cases}
	(2\log\varepsilon_(t))'(\dd\bt_{L})(\dd u^{V}) & \text{on }P ~, \\
	\vep'(t)\cdot \frac{\pl_\vep\kp_{\vep(t)}}{\pl_r\kp_{\vep(t)}}\cdot \pl_ru & \text{on }Q ~, \\
	0 & \text{on }\out_{\,b} ~.
\end{cases}
\end{align*}
Hence,
\begin{align*}
    \mathcal{P}^{\vep}_{\ul{0}}[u] = -\langle\nabla\theta_{N^{\vep}}, \widehat{V}_{\ul{0}}(\dd u)\rangle_{g^{\vep}} + S^{\vep}[u] ~.
\end{align*}
By Lemma~\ref{lem: angle and mean curvature in torus case},
\begin{align*}
    |\langle\nabla\theta_{N^{\vep}}, \widehat{V}_{\ul{0}}(\dd u)\rangle_{g^{\vep}}|\leq C\vep(t)^{m(1-\tau)-\tau}|\dd u|_{g^{\vep}},\quad\mbox{on}\;Q ~,
\end{align*}
By Assumption~\ref{assumption on epsilon}, we have
\begin{align*}
    |S^{\vep}[u]| \leq C\vep(t)^{m-1}|\dd u|_{g^{\vep}},\quad\mbox{on}\;P\cup Q ~.
\end{align*}
Combining these together and using the assumption $\tau<\frac{1}{m+2}$ yields
\begin{align*}
    |\mathcal{P}^{\vep}_{\ul{0}}[u]|\leq C\:\vep(t)^{m-1}|\dd u|_{g^{\vep}}.
\end{align*}
From this, we further estimate:
\begin{align*}
    t^{\mu}\rho_{\vep}^{\nu+2}|\mathcal{P}^{\vep}_{\ul{0}}[u]|\leq C\vep(t)^{m-1}\rho_{\vep}\left(t^{\mu}\rho_{\vep}^{\nu+1}|\dd u|_{g^{\vep}}\right)\leq C\Lambda^{\frac{1-m}{m-2}}\|u\|_{P^{1, 2, \alpha}_{\mu, \nu, \Lambda}}
\end{align*}
if $\tau\in(0, \frac{1}{m+2})$. The H\"older norm estimate follows similarly, by using $|\nabla \dd \theta_{N^{\vep}}| = O(\vep^{m(1-\tau)-2\tau})$ and $|\partial_{t}\nabla\theta_{N^{\vep}}| = O(\vep^{m-2+\tau})$ from Lemma \ref{lem: angle and mean curvature in torus case}.
\end{proof}

We now combine these estimates to deduce the estimates and existence theory for $\partial_t - \mathcal{L}_{\ul o}^\vep$.

\begin{theorem}\label{linear theory}
	Given $\mu>0$, $\nu\in(0, m-2)$, $\alpha\in(0, \frac{1}{2})$, $\tau\in(0, \frac{1}{m+2})$, there exists $\Lambda\gg 1$ with the following significance. Given $\psi\in P^{0, 0, \alpha}_{\mu, \nu+2, \Lambda}$, there exists a unique $u\in P^{1, 2, \alpha}_{\mu, \nu, \Lambda}\cap\langle 1, w^{\vep}\rangle^{\perp}$ and $a, b:[\Lambda, \infty)\to\mathbb{R}$ such that
	\begin{align}\label{full linear eq}
		\begin{cases}
		\partial_{t}u - \mathcal{L}^{\vep}_{\ul{0}}[u] = \psi + a(t)+b(t)w^{\vep},\quad t\in[\Lambda, \infty),\\
		u(x, \Lambda) = 0,\quad x\in\ul{N},
	\end{cases}
	\end{align}
	and $u$ satisfies the a priori estimate
	\begin{align}
	\|u\|_{P^{1, 2, \alpha}_{\mu, \nu, \Lambda}}\leq C\|\psi\|_{P^{0, 0, \alpha}_{\mu, \nu+2, \Lambda}}	\label{eq-linear theory apriori estimate}
	\end{align}
	for some $C>0$ independent of $t$.
\end{theorem}

% To prove Theorem \ref{linear theory}, we start with the sup-norm estimate for the solutions to 
% \begin{align}\label{linear parabolic eq}
% 	\begin{cases}
% 		\partial_{t}u - \Delta_{g^{\vep}}u = \psi,\quad t\in[\Lambda, \infty),\\
% 		u(x, \Lambda) = 0,\quad x\in\ul{N},
% 	\end{cases}
% \end{align}
% where $u$ satisfies the orthogonality condition
% \begin{align*}
% 	\int_{\ul{N}}u\cdot w^{\bvep}\:\dd V_{g^{\vep}} = \int_{\ul{N}}u\cdot 1\:\dd V_{g^{\vep}} = 0,\quad t\in[\Lambda, \infty).
% \end{align*}
% By Theorem \ref{weighted sup norm estimate}, we have
% \begin{align}\label{sup norm est}
% 	\|u\|_{P^{0, 0, 0}_{\mu, \nu, \Lambda}}\leq C\|\psi\|_{P^{0, 0, 0}_{\mu, \nu+2, \Lambda}}.
% \end{align}
% Combining (\ref{sup norm est}) with Corollary \ref{weighted schauder estimate}, we therefore obtain the uniform estimate
% \begin{align}
% \|u\|_{P^{1, 2, \alpha}_{\mu, \nu, \Lambda}}\leq C\|\psi\|_{P^{0, 0, \alpha}_{\mu, \nu+2, \Lambda}}
% \end{align}
% whenever $u$ is a solution to (\ref{linear parabolic eq}) and satisfies the orthogonality condition. 
\begin{proof}
First, we claim that, given $\psi\in P^{0, 0, \alpha}_{\mu, \nu+2, \Lambda}$, there exists $\Lambda_{0}\gg 1$ such that for each $\Lambda\geq \Lambda_{0}$, there exists a unique $u:\ul{N}\times[\Lambda, \infty)\to\mathbb{R}$ solving
\begin{align}\label{Brendle's trick}
	\begin{cases}
		\partial_{t}u - \Delta_{g^{\vep}}u = \psi + a(t) + b(t)w^{\vep} \quad t\in[\Lambda, \infty),\\
		u(x, \Lambda) = 0,\quad x\in\ul{N},
	\end{cases}
\end{align}
with estimate $\|u\|_{P^{1, 2, \alpha}_{\mu, \nu, \Lambda}}\leq C\|\psi\|_{P^{0, 0, \alpha}_{\mu, \nu+2, \Lambda}}$, where $C>0$ is independent of $\Lambda$.

For this purpose, define a zeroth order operator
\begin{align*}
    F^{\vep}[u] := \frac{1}{\|w^{\vep}\|^{2}_{L^{2}}}\left\{\int_{\ul{N}}u\cdot(\partial_{t}w^{\vep} + \Delta_{g^{\vep}}w^{\vep})\:\dd V_{g^{\vep}} + \int_{\ul{N}}u\cdot w^{\vep}\:\partial_{t}\dd V_{g^{\vep}}\right\} + \frac{1}{{\rm Vol}(N^{\vep})}\int_{\ul{N}}u\:\partial_{t}\dd V_{g^{\vep}}.
\end{align*}
Note that $F^{\vep}[u]$ encodes how the orthogonality condition is changed in time. Let
\begin{align*}
	\psi^{\perp} := \psi - \frac{1}{{\rm Vol}(N^{\vep})}\int_{\ul{N}}\psi\:\dd V_{g^{\vep}} - \frac{1}{\|w^{\vep}\|^{2}_{L^{2}}}\int_{\ul{N}}\psi\cdot w^{\vep}\:\dd V_{g^{\vep}}\cdot w^{\vep}.
\end{align*} 
By standard parabolic theory, there exists $u:\ul{N}\times[\Lambda, \Lambda + T]\to\mathbb{R}$ solving
\begin{align}\label{eq-fvepu problem}
	\begin{cases}
		\displaystyle\partial_{t}u - \Delta_{g^{\vep}}u + F^{\vep}[u] = \psi^{\perp},\quad t\in[\Lambda, \Lambda+T],\\
		u(x, \Lambda) = 0,\quad x\in\ul{N}.
	\end{cases}
\end{align}
% for any $T>0$. By integrating (\ref{eq-fvepu problem}) against $1$ and $w^\vep$ and integrating by parts, we deduce that the solution $u$ satisfies the orthogonality condition. 
% Hence, Corollary~\ref{weighted schauder estimate} and 
% the estimate of $E_{0}(t)$ in Lemma~\ref{lem: estimate of non-orthogonality} implies that
% \begin{align*}
%     \|u\|_{P^{1, 2, \alpha}_{\mu, \nu, \Lambda}}\leq C\|\psi^{\perp}\|_{P^{0, 0, \alpha}_{\mu, \nu+2, \Lambda}} + \|F^{\vep}[u]\|_{P^{0, 0, \alpha}_{\mu, \nu, \Lambda}}\leq C\|\psi\|_{P^{0, 0, \alpha}_{\mu, \nu, \Lambda}} + C(\Lambda)\|u\|_{P^{1, 2, \alpha}_{\mu, \nu, \Lambda}},
% \end{align*}
% where $\lim_{\Lambda\to \infty}C(\Lambda) = 0$. 
 Letting
\begin{align*}
	&a(t) = \frac{1}{{\rm Vol}(N^{\vep})}\int_{\ul{N}}\left(\partial_t u - \psi\right)\:\dd V_{g^{\vep}},\\
	&b(t) = \frac{1}{\|w^{\vep}\|^{2}_{L^{2}}}\int_{\ul{N}}\left(\partial_{t}u - \Delta_{g^{\vep}}u - \psi\right)\cdot w^{\vep}\:\dd V_{g^{\vep}},
\end{align*}
it follows that the triple $(u, a, b)$ solves
\begin{align*}
	\partial_{t}u - \Delta_{g^{\vep}}u = \psi + a(t) + b(t)w^{\vep},
\end{align*}
and by Corollary \ref{weighted schauder estimate} and Lemma \ref{lem-a and b estimate}, for $\Lambda>0$ large enough, the estimate $\|u\|_{P^{1, 2, \alpha}_{\mu, \nu, \Lambda}}\leq C\|\psi\|_{P^{0, 0, \alpha}_{\mu, \nu+2, \Lambda}}$ holds, for any $T>0$. It follows that the operator
\begin{align*}
	\mathcal{A}^{\vep}_{\ul{0}}:(u, a(t), b(t))\longmapsto \partial_{t}u - \Delta_{g^{\vep}}u - a(t) - b(t)w^{\vep}
\end{align*}
as a bounded operator from $P^{1, 2, \alpha}_{\mu, \nu, \Lambda}\cap\langle 1, w^{\vep}\rangle^\perp\times C^{0, \alpha}_{\mu, \Lambda}\times C^{0, \alpha}_{\mu, \Lambda}$ to $P^{0, 0, \alpha}_{\mu, \nu+2, \Lambda}$ is a linear isomorphism whose inverse is bounded by $C$, which is independent of $\Lambda\geq \Lambda_{0}$. This proves the claim.

Now, our goal is to show that $\mathcal{L}^{\vep}_{\ul{0}} = \mathcal{A}^{\vep}_{\ul{0}} - \mathcal{P}^{\vep}_{\ul{0}}$ is invertible. Write
\begin{align*}
	\mathcal{A}^{\vep}_{\ul{0}} - \mathcal{P}^{\vep}_{\ul{0}} = \mathcal{A}^{\vep}_{\ul{0}}\left(\mathcal{I} - (\mathcal{A}^{\vep}_{\ul{0}})^{-1}\mathcal{P}^{\vep}_{\ul{0}}\right),
\end{align*}
where $\mathcal{I}$ is the identity operator in $P^{1, 2, \alpha}_{\mu, \nu, \Lambda}\cap\langle 1, w^{\vep}\rangle\times C^{0, \alpha}_{\mu, \Lambda}\times C^{0, \alpha}_{\mu, \Lambda}$. Since by Lemma \ref{lem-lower order term estimate} $\|(\mathcal{A}^{\vep}_{\ul{0}})^{-1}\mathcal{P}^{\vep}_{\ul{0}}\|\to 0$ as $\Lambda\to\infty$, it follows that $\mathcal{I} - (\mathcal{A}^{\vep}_{\ul{0}})^{-1}\mathcal{P}^{\vep}_{\ul{0}}$ is invertible for large $\Lambda>0$. Hence, $\mathcal{A}^{\vep}_{\ul{0}} - \mathcal{P}^{\vep}_{\ul{0}}$ is invertible for large $\Lambda>0$.
\end{proof}

%%%%%%%%%%%%%%%%%%%%%%%%%%%%%%%%
\section{Estimates for the Error Terms} \label{sec-estimate-error}

In this section, we provide pointwise estimates for the zeroth order term, $\theta_{N^\vep} + \xi(0)$ and the quadratic term $Q^\vep[\dd u]$, which will be utilised in the iteration scheme of section \ref{sec-iteration}. We also estimate the projection of the zeroth order term onto the approximate kernel, whose dominant term provides the approximate ODE that $\vep(t)$ should satisfy.

\subsection{The Zeroth Order Error}

The main zeroth order error estimate is the following.
\begin{proposition}\label{prop: zeroth order error}
Assume that the constants $\mu>0$, $\nu\in(0, m-2)$, $\alpha\in(0, \frac{1}{2})$ and $\tau\in(0, \frac{1}{2})$ satisfy the relation
\begin{align}\label{equation: assumption on constants}
    \tau>\frac{2\alpha}{m+1+2\alpha},\quad\frac{\nu+2}{m-2}<\mu<\frac{1}{m-2}(\tau(\nu+2)+(1-\tau)m).
\end{align}
Then, we have
\begin{align}
    \lim_{\Lambda\to\infty}\|\theta_{N^{\vep}} + \xi(0)\|_{P^{0, 0, \alpha}_{\mu, \nu+2, \Lambda}} = 0.
\end{align}
Precisely, we have the following bounds in terms of $\Lambda$:
	\begin{align}
		\|\theta_{N^\vep}\|_{P^{0,0,\alpha}_{\mu,\nu+2,\Lambda}} \leq C\Lambda^{\mu - \frac{1}{m-2}\left( \tau(\nu+2)+(1-\tau)m\right)}, \quad \|\xi(0)\|_{P^{0,0,\alpha}_{\mu,\nu+2,\Lambda}} \leq C\Lambda^{\mu-\frac{1}{m-2}\left( m - 2\alpha \right)},
	\end{align}
 for some $C>0$ independent of $\Lambda$.
\end{proposition}

\begin{remark}
    Notice that $\tau(\nu+2)+(1-\tau)m - (\nu+2) = (1-\tau)(m-2-\nu) > 0$, and the ranges for the constants in (\ref{equation: assumption on constants}) are non-empty.
\end{remark}

\begin{proof}
We first estimate the H\"older norms of $\theta_{N^{\vep}}$. By construction, it suffices to consider the transition region $\overline{\kappa_{\vep}}^{-1}(\Sigma\times(\vep^{\tau}, 2\vep^{\tau}))$.

By Lemma~\ref{lem: angle and mean curvature in torus case}, we have for $x, x'\in \overline{\kappa}_{\vep}^{-1}(\Sigma\times(\vep^{\tau}, 2\vep^{\tau}))$,
	\begin{align*}
		t^{\mu}\rho_{\vep}^{\nu+2}(x, t)|\theta_{N^{\vep}}|(x, t)\leq C\:t^{\mu}\vep(t)^{\tau(\nu+2)}\vep(t)^{(1-\tau)m} \leq C\:t^{\mu - \frac{1}{m-2}(\tau(\nu+2)+(1-\tau)m)},
	\end{align*}
 and similarly, using $|\dd\theta_{N^{\vep}}|\leq C\vep(t)^{(1-\tau)m-\tau}$,
 \begin{align*}
     t^{\mu}\rho_{\vep}^{\nu+2+2\alpha}(x, t)\frac{|\theta_{N^{\vep}}(x, t) - \theta_{N^{\vep}}(x', t)|}{d_{g^{\vep}}(x, x')^{2\alpha}}&\leq C\:t^{\mu}\vep(t)^{\tau(\nu+2+2\alpha)}\cdot\vep(t)^{(1-\tau)m-\tau}\cdot \vep(t)^{\tau(1-2\alpha)}\\
     &\leq C\:t^{\mu - \frac{1}{m-2}(\tau(\nu+2)+(1-\tau)m)}.
 \end{align*}
On the other hand, for $t_2>t_1$, $t_1, t_2\in[t, 2t]$, $0<|t_1 - t_2|<t^{-\frac{2}{m-1}}$, and $x = (\sigma, r)\in \bigcup_{t_1<t<t_2}\overline{\kappa}_{\vep}^{-1}(\Sigma\times(\vep^{\tau}(t), 2\vep^{\tau}(t)))$,
\begin{align}\label{equation: time-derivative of angle}
    &\theta_{N^{\vep}}(x, t_1) - \theta_{N^{\vep}}(x, t_2)\nonumber\\
    &= (\theta\circ\Phi_{C})((\sigma, \kappa_{\vep(t_1)}(r)), \dd\mathfrak{Q}_{\vep(t_1)}(\sigma, \kappa_{\vep(t_1)}(r))) - (\theta\circ\Phi_{C})((\sigma, \kappa_{\vep(t_2)}(r)), \dd\mathfrak{Q}_{\vep(t_2)}(\sigma, \kappa_{\vep(t_2)}(r)))\nonumber\\
    &=\int_{s=0}^{s=1}\frac{\dd}{\dd s}(\theta\circ\Phi_{C})((\sigma, \kappa_{\vep(s)}(r)), \dd\mathfrak{Q}_{\vep(s)}(\sigma, \kappa_{\vep(s)}(r)))\:\dd s,
\end{align}
where $\vep(s) := \vep(s t_2 + (1-s) t_1)$. Using $\vep'(s) \leq C\vep(t)^{m-1}|t_1 - t_2|$ and $\partial_{\vep}\kappa_{\vep}(r) \leq \frac{\kappa_{\vep}(r)}{\vep}\leq\vep^{\tau-1}$, we deduce that
\begin{align*}
    \partial_{s}\dd\mathfrak{Q}_{\vep(s)}(\sigma, \kappa_{\vep(s)}(r)) = O(\vep(t)^{m+(1-\tau)(m-2)})\cdot|t_1 - t_2|,
\end{align*}
measuring by the induced metric $g^{\vep}_{\ul{0}}$. Inserting this into (\ref{equation: time-derivative of angle}) gives
\begin{align*}
    |\theta_{N^{\vep}}(x, t_1) - \theta_{N^{\vep}}(x, t_2)|\leq C\:\vep(t)^{m-2+\tau}|t_1 - t_2|.
\end{align*}
Thus,
\begin{align*}
    t^{\mu}\rho_{\vep(t)}^{\nu+2+2\alpha}(x)\frac{|\theta_{N^{\vep}}(x, t_1) - \theta_{N^{\vep}}(x, t_2)|}{|t_1 - t_2|^{\alpha}}\leq C\:t^{\mu-\frac{1}{m-2}(\tau(\nu+2)+m+\tau-2\alpha(1-\tau))}.
\end{align*}
 
Putting these together, we obtain
\begin{align*}
     \|\theta_{N^{\vep}}\|_{P^{0, 0, \alpha}_{\mu, \nu+2, \Lambda}}\leq C\:\Lambda^{\mu-\frac{\tau(\nu+2)+(1-\tau)m}{m-2}} = o(1),\quad\mbox{as}\;\Lambda\to\infty.
\end{align*}
if $\tau>\frac{2\alpha}{m+1+2\alpha}$ and  $\mu<\frac{\tau(\nu+2)+(1-\tau)m}{m-2}$.

For $\xi(0)$, it is not hard to deduce from \eqref{choice of constants in l.o.t} and \eqref{l.o.t} that $\|\xi(0)\|_{P^{0, 0, \alpha}_{\mu, \nu+2, \Lambda}}\leq C\:\Lambda^{\mu-\frac{m-2\alpha}{m-2}}$.
This finishes the proof of this proposition.
\end{proof}

\subsection{The Quadratic Error}
Let $Q^{\vep}[\dd u] := -\theta_{N^{\vep}} - {\xi}(0) + \partial_{t}u - \mathcal{L}^{\vep}_{\ul{0}}[u]$ be the quadratic error term. We now estimtate $Q^{\vep}[\dd u]$ in terms of weighted norms of $u$.
\begin{proposition}\label{Prop: quadratic error}
    There is $C>0$ and $\Lambda\gg 1$ such that if $u\in P^{1, 2, \alpha}_{\mu, \nu, \Lambda}$, $t, t_1, t_2\geq \Lambda$ with $0<|t_1 - t_2|<t^{\frac{-2}{m-2}}$, and $x, x_1, x_2 \in\ul{N}$ with $0<d_{g^{\vep}}(x_1, x_2)<\rho_{\vep(t)}(x_1)$, then
    \begin{align}
        |Q^{\vep}[\dd u]|(x, t)&\leq C\|u\|_{P^{1, 2, \alpha}_{\mu, \nu, \Lambda}}\cdot t^{-2\mu}\rho_{\vep(t)}^{-2\nu-4}(x),\label{C0 of Q}\\
        \frac{|Q^{\vep}[\dd u](x_1, t) - Q^{\vep}[\dd u](x_2, t)|}{d_{g^{\vep}}(x_1, x_2)^{2\alpha}}&\leq C\|u\|_{P^{1, 2, \alpha}_{\mu, \nu, \Lambda}}\cdot t^{-2\mu}\rho_{\vep(t)}^{-2\nu-4-2\alpha}(x_1),\label{C^alpha of Q in space}\\
        \frac{|Q^{\vep}[\dd u](x, t_1) - Q^{\vep}[\dd u](x, t_2)|}{|t_1 - t_2|^{\alpha}}&\leq C\|u\|_{P^{1, 2, \alpha}_{\mu, \nu, \Lambda}}\cdot t^{-2\mu}\rho_{\vep(t)}^{-2\nu-4-2\alpha}(x) + Ct^{-\frac{m+2-2\alpha}{m-1}}.\label{C_alpha of Q in time}
    \end{align}
\end{proposition}
\begin{proof}
    Write $Q^{\vep}[\dd u] = Q^{\vep}_{\theta}[\dd u] + Q^{\vep}_{\xi}[\dd u]$, where
    \begin{align*}
        &Q^{\vep}_{\theta}[\dd u] = \theta_{N^{\vep}}[\dd u] - \theta_{N^{\vep}} - \Delta_{N^{\vep}}u + \langle\nabla\theta_{N^{\vep}}, \hat{V}_{\ul{0}}(\dd u)\rangle,\\
        &Q^{\vep}_{\xi}[\dd u] = \xi[\dd u] - \xi(0) - S^{\vep}[u].
    \end{align*}
    We first estimate $Q^{\vep}_{\theta}$. In the tip region, the induced metric is uniformly equivalent to the metric $\vep_{j}^{2}g_{L_{j}}$. Using the scale-invariant property of Lagrangian angle we have
    \begin{align*}
        |Q^{\vep}_{\theta}(x, t, \dd u(x, t), \nabla\dd u(x, t))|\leq C\left(\vep_{j}^{-2}|\dd u|^{2}_{g^{\vep}}+|\nabla\dd u|^{2}_{g^{\vep}}\right),\quad x\in P_{j},
    \end{align*}
    for some $C>0$ independent of $\vep$. Similarly, in the intermediate region, the metric is uniformly equivalent to the cone metric, and the scale-invariant property of Lagrangian angle implies
    \begin{align*}
         |Q^{\vep}_{\theta}((\mathfrak{r}, \sigma), t, \dd u((\mathfrak{r}, \sigma), t), \nabla\dd u((\mathfrak{r}, \sigma), t))|\leq C\left(\mathfrak{r}^{-2}|\dd u|^{2}_{g^{\vep}}+|\nabla\dd u|^{2}_{g^{\vep}}\right),\quad (\mathfrak{r}, \sigma))\in \kappa_{\vep}Q^{\pm}_{j},
    \end{align*}
    for some $C>0$ independent of $\vep$. Combining these estimates yields
    \begin{align*}
        |Q^{\vep}_{\theta}(x, t, \dd u(x, t), \nabla\dd u(x, t))|\leq C\left(\rho_{\vep}^{-2}(x, t)|\dd u(x, t)|^{2}_{g^{\vep}}+|\nabla\dd u(x, t)|^{2}_{g^{\vep}}\right),\quad(x, t)\in \ul{N}\times[\Lambda, \infty).
    \end{align*}
    Multiply both sides by $t^{\mu}\rho_{\vep}^{\nu+2}$ gives
    \begin{align*}
        t^{\mu}\rho_{\vep}^{\nu+2}|Q^{\vep}_{\theta}[\dd u]|&\leq Ct^{\mu}\left(\rho_{\vep}^{\nu}|\dd u|_{g^{\vep}}^{2} + \rho_{\vep}^{\nu+2}|\nabla\dd u|^{2}_{g^{\vep}}\right)\\
        %&= C\left[t^{-\mu}\rho_{\vep}^{-\nu-2}\left(t^{\mu}\rho_{\vep}^{\nu+1}|\dd u|_{g^{\vep}}\right)^{2} + t^{-\mu}\rho_{\vep}^{-\nu-2}\left(t^{\mu}\rho_{\vep}^{\nu+2}|\nabla\dd u|_{g^{\vep}}\right)^{2}\right]\\
        &\leq 2Ct^{-\mu}\rho_{\vep}^{-\nu-2}\|u\|^{2}_{P^{1, 2, \alpha}_{\mu, \nu, \Lambda}} ~.
    \end{align*}

    The estimate for $Q^{\vep}_{\xi}[\dd u]$ follows similarly. From the explicit expression (\ref{l.o.t}), we only need to consider the tip region. By Taylor theorem,
    \begin{align*}
        \xi_{N^{\vep}}[\dd u] = (\vep_{j}^{2})'\left[\alpha_{L}(x, 0) + \partial_{y}\alpha_{L}(x, 0)\cdot\vep_{j}^{-2}\dd u + O(\vep_{j}^{-4}|\dd u|^{2}_{g_{L_{j}}})\right] ~.
    \end{align*}
    It follows that, using $|(\vep_{j}^{2})'|\ll 1$,
    \begin{align*}
        |Q^{\vep}_{\xi}[\dd u](x, t)|\leq C\vep_{j}^{-2}|\dd u(x, t)|^{2}_{g^{\vep}}\leq C\rho_{\vep}^{-2}(x, t)|\dd u(x, t)|^{2}_{g^{\vep}} ~.
    \end{align*}
    Multiply both sides by $t^{\mu}\rho_{\vep}^{\nu+2}$ and estimate as above gives
    \begin{align*}
        t^{\mu}\rho_{\vep}^{\nu+2}|Q^{\vep}_{\xi}[\dd u]|\leq Ct^{-\mu}\rho_{\vep}^{-\nu-2}\|u\|^{2}_{P^{1, 2, \alpha}_{\mu, \nu, \Lambda}} ~.
    \end{align*}
    Combining everything together yields
    \begin{align*}
        |Q^{\vep}[\dd u]|\leq |Q^{\vep}_{\theta}[\dd u]|+|Q^{\vep}_{\xi}[\dd u]|\leq Ct^{-2\mu}\rho_{\vep}^{-2\nu-4} ~.
    \end{align*}
    This proves (\ref{C0 of Q}).

    To prove (\ref{C^alpha of Q in space}), we fix $t\in[\Lambda, \infty)$, and view $\theta_{N^{\vep}}[\dd u](\cdot, t)$ and $\xi[\dd u](\cdot, t)$ as coming from restricting smooth functions
    \begin{align*}
        \Theta(x, y, z),\quad\Xi(x, y),\quad x\in\ul{N},\;y\in T^{*}_{x}\ul{N},\;z\in \textstyle\bigotimes^{2}T^{*}_{x}\ul{N} ~,
    \end{align*}
    to the graph $\{(x, \dd u(x, t), \nabla\dd u(x, t))\::\:x\in\ul{N}\}$, namely, we have
    \begin{align*}
        \theta_{N^{\vep}}[\dd u](x, t) = \Theta(x, \dd u(x, t), \nabla\dd u(x, t)), \quad\xi[\dd u](x, t) = \Xi(x, \dd u(x, t)) ~.
    \end{align*}
    Note that by scale-invariant property we have
    \begin{align*}|\partial_{x}^{a}\partial_{y}^{b}\partial_{z}^{c}\Theta|\leq C\rho_{\vep}^{-a-b},\quad|\partial_{x}^{a}\partial_{y}^{b}\Xi|\leq C\rho_{\vep}^{-a-b},\quad a, b, c\in\mathbb{N}\cup\{0\} ~.
    \end{align*}
    Then a long but straightforward computation using mean value theorem shows that, for $x_1\neq x_2\in\ul{N}$ with $d_{g^{\vep}}(x_1, x_2)<\rho_{\vep}(x_1, t)$, 
    \begin{align*}
        t^{\mu}\rho_{\vep}^{\nu+2+2\alpha}(x_1)\frac{|Q^{\vep}_{\theta}[\dd u](x_1, t) - Q^{\vep}_{\theta}[\dd u](x_2, t)|}{d_{g^{\vep}}(x_1, x_2)^{2\alpha}}\leq Ct^{-\mu}\rho_{\vep}^{-\nu-2}(x_1)\|u\|_{P^{1, 2, \alpha}_{\mu, \nu, \Lambda}}^{2} ~.
    \end{align*}
    Similarly,
    \begin{align*}
        t^{\mu}\rho_{\vep}^{\nu+2+2\alpha}(x_1)\frac{|Q^{\vep}_{\xi}[\dd u](x_1, t) - Q^{\vep}_{\xi}[\dd u](x_2, t)|}{d_{g^{\vep}}(x_1, x_2)^{2\alpha}}\leq Ct^{-\mu}\rho_{\vep}^{-\nu-2}(x_1)\|u\|_{P^{1, 2, \alpha}_{\mu, \nu, \Lambda}}^{2} ~.
    \end{align*}
    Combining these estimates yields (\ref{C^alpha of Q in space}).

    Finally, we prove (\ref{C_alpha of Q in time}). A similar computation to those of Joyce \cite{Joyce2004SLCS3}*{Proposition~5.8} and Pacini \cite{Pacini2013b}*{Proposition~5.6} shows that if $\alpha, \beta$ are small closed $1$-forms on $\ul{N}$, then for each fixed $t$,
    \begin{align*}
          |Q^{\vep}[\alpha] - Q^{\vep}[\beta]|&\leq C\left(\rho_{\vep}^{-1}|\alpha - \beta| + |\nabla(\alpha - \beta)|\right)\left(\rho_{\vep}^{-1}|\alpha|+ \rho_{\vep}^{-1}|\beta| + |\nabla\alpha| + |\nabla\beta|\right).
    \end{align*}
    Letting $\alpha = \dd u(\cdot, t_1)$, $\beta = \dd u(\cdot, t_2)$ yields
    \begin{align*}
        t^{\mu}\rho_{\vep}^{\nu+2+2\alpha}(x, t)\frac{|Q^{\vep(t)}[\dd u(x, t_1)] - Q^{\vep(t)}[\dd u(x, t_2)]|}{|t_1 - t_2|^{\alpha}}\leq Ct^{-\mu}\rho_{\vep}^{-\nu-2}(x, t)\|u\|^{2}_{P^{1, 2, \alpha}_{\mu, \nu, \Lambda}} ~.
    \end{align*}
    On the other hand, by the assumption on $\vep$ we have
    \begin{align*}
        \frac{|Q^{\vep(t_1)}[\dd u(x, t_2)] - Q^{\vep(t_2)}[\dd u(x, t_2)]|}{|t_1 - t_2|}\leq C\sup_{j}|\vep_{j}'(t)\vep_{j}(t)|\leq C t^{-\frac{m}{m-1}} ~.
    \end{align*}
    Combining these estimates, we conclude that for $t_1, t_2\geq\Lambda$ with $0<|t_1 - t_2|<t^{-\frac{2}{m-2}}$,
    \begin{align*}
        \frac{|Q^{\vep}[\dd u](x, t_1) - Q^{\vep}[\dd u](x, t_2)|}{|t_1 - t_2|^{\alpha}}\leq Ct^{-2\mu}\rho_{\vep}^{-2\nu-4-2\alpha}\|u\|^{2}_{P^{1, 2, \alpha}_{\mu, \nu, \Lambda}} + Ct^{-\frac{m+2-2\alpha}{m-1}} ~.
    \end{align*}
    This proves (\ref{C_alpha of Q in time}).
\end{proof}

\subsection{Projection onto the Approximate Kernel}

Finally, we will require the following integral estimates, which are the projection of the zeroth order terms onto the approximate kernel.

\begin{lemma}\label{projection of theta onto 1}
    We have
    %%%% OLD, STRONGER ESTIMATE %%%%
    % \begin{align}
    %     \int_{\Sigma^{\pm}\times(\vep^{\tau}, 2\vep^{\tau})}\theta_{N^{\vep}}\:\dd V_{N^{\vep}} = \pm \vep^{m} A + O(\vep^{(2-\tau)m}).
    % \end{align}
    \begin{align}
        \int_{\Sigma^{\pm}\times(\vep^{\tau}, 2\vep^{\tau})}\theta_{N^{\vep}}\:\dd V_{N^{\vep}} = \pm \vep^{m} A + O(\vep^{(1+\tau)m}).
    \end{align}
\end{lemma}
\begin{proof}

This follows from the proof of \cite{Joyce2004SLCS3}*{Proposition~7.5} by estimating $\theta_{N^\vep}$ by $\sin(\theta_{N^\vep}) + O(\theta^3)$, and using Lemma \ref{Lawlor_Z} and Lemma \ref{lem: angle and mean curvature in torus case}.
\end{proof}
Applying this Lemma, we have the following projection formula for the zeroth order term.
\begin{proposition} The $L^{2}$ projection of the zeroth order error $\theta_{N^{\vep}}+\xi(0)$ onto the approximate kernel  ${\rm span}_{\mathbb{R}}\{1, w^{\vep}_{(0, 1)}\}$ is given by
    \begin{align}
    \int_{\ul{N}}\left[\theta_{N^{\vep}}+\xi(0)\right]\cdot w^{\vep}_{(0, 1)}\:\dd V_{g^{\vep}} &= c_L\frac{V_{1}V_{2}}{V_{1}+V_{2}}\left\{\frac{\dd\vep^{2}(t)}{\dd t} + \frac{A}{c_L}\frac{V_1 + V_2}{V_1V_2}\vep^{m}(t)\right\} + O(\vep^{(1+\tau)m})~,\label{full projection on to approx kernel}
    \end{align}
    and
    \begin{align}
        \int_{\ul{N}}\left[\theta_{N^{\vep}}+\xi(0)\right]\cdot 1\:\dd V_{g^{\vep}} &= O(\vep^{(1+\tau)m})~. \label{full projection onto 1}
    \end{align}
\end{proposition}
\begin{proof}
     It follows from these choices that
    \begin{align*}
        \int_{\ul{N}}\xi(0)\cdot w^{\vep}_{(0, 1)}\:\dd V_{g^{\vep}} &= \int_{P}\frac{\dd\vep^{2}(t)}{\dd t}\beta_{L} + \int_{Q^{\pm}}\xi(0)\cdot w^{\vep}_{(0, 1)} + \int_{X^{o}_{2}}c_L\frac{\dd\vep^{2}(t)}{\dd t}\\
        & = V_{2}\cdot c_L\cdot\frac{\dd\vep^{2}(t)}{\dd t} + O(\vep^{(1+\tau)m}(t)) ~,
    \end{align*}
    where in the second line we used the assumption that $\frac{\dd \vep(t)}{\dd t} = O(\vep^{m-1}(t))$.

    Combining with Lemma~\ref{projection of theta onto 1} and using the fact that the volume of the interpolating region is $O(\vep^{\tau m})$ yield
    \begin{align*}
        \int_{\ul{N}}&\left[\theta_{N^{\vep}}+\xi(0)\right]\cdot w^{\vep}_{(0, 1)}\:\dd V_{g^{\vep}}\\
        &= \vep^{m}(t)A + V_{2}\cdot c_L\frac{\dd\vep^{2}(t)}{\dd t} - \frac{c_LV_{2}}{V_{1}+V_{2}}\frac{\dd\vep^{2}(t)}{\dd t}(V_{2} - O(\vep^{\tau m})) + O(\vep^{(\tau+1)m} + \vep^{(2-\tau)m})\\
        & = c_L\frac{V_{1}V_{2}}{V_{1}+V_{2}}\left\{\frac{\dd\vep^{2}(t)}{\dd t} + \frac{A}{c_L}\frac{V_1 + V_2}{V_1V_2}\vep^{m}(t)\right\} + O(\vep^{(\tau+1)m}) ~.
    \end{align*}
    This proves (\ref{full projection on to approx kernel}). Equation (\ref{full projection onto 1}) follows from a similar computation.
\end{proof}

By (\ref{full projection on to approx kernel}), (\ref{full projection onto 1}), the projection onto the normalised approximate kernel element $w^{\vep}$ (as defined in (\ref{eq-normalisedapproxkernel})\,) takes the same form:
\begin{align}\label{full projection onto normalised approx kernel}
    \int_{\ul{N}}&\left[\theta_{N^{\vep}}+\xi(0)\right]\cdot w^{\vep}\:\dd V_{g^{\vep}}= c_L\frac{V_{1}V_{2}}{V_{1}+V_{2}}\left\{\frac{\dd\vep^{2}(t)}{\dd t} + \frac{A}{c_L}\frac{V_1 + V_2}{V_1V_2}\vep^{m}(t)\right\} + O(\vep^{(\tau+1)m}) ~.
\end{align}
In section \ref{sec-iteration}, the above integrals will appear as error terms that we wish to minimise.  Therefore, $\vep(t)$ shall be a small perturbation of a solution to the ODE given by the leading order term on the right hand side.  This is the reason why Assumption \ref{assumption on epsilon} is imposed.

%\begin{remark}\label{rmk-heuristicODE}
%In the iteration scheme of Section \ref{sec-iteration}, the above integrals will appear as error terms that we wish to minimise. We will therefore define $\vep(t)$ to be a small perturbation of a solution of the following ODE:

%\begin{equation}\label{eq-heuristic_ode_in_torus_case}
%\frac{\dd\vep^2(t)}{\dd t} + \frac{A}{c_L}\left( \frac{1}{V_1} + \frac{1}{V_2}\right) \vep^m(t) \, = \, 0.
%\end{equation}
%\end{remark}

%We note that any solution $\vep(t)$ to this ODE satisfies Assumption \ref{assumption on epsilon}.

\section{Solving the Nonlinear Equation} \label{sec-iteration}

We are now ready to state and prove our main theorem precisely. For the remainder of the paper, we make the following assumptions on $\nu, \alpha, \tau, \mu, \zeta$, which imply all previously made assumptions on these constants:
\begin{align}
    \nu &\in \left(\max\left\{\frac{m}{2}-2, 0\right\}, m-2\right),\quad \alpha \in \left(0, \frac{1}{2}\right), \quad \tau \in \left(\frac{2\alpha}{m+1+2\alpha}, \frac{1}{m+2}\right), \label{assumption on constants} \\
    \mu &\in \left(\frac{\nu + 2 + 2\alpha}{m-2}, \frac{1}{m-2}(\tau(\nu+2) + (1-\tau)m)\right), \quad \zeta \in \left(0, \min\left\{ \frac{\tau m}{m-2}, \mu - \frac{\nu + 2 + 2\alpha}{m-2}\right\}\right).\notag
\end{align}
(For example, $(\nu, \alpha, \tau, \mu) = (\frac{3m - 8}{4}, \frac{1}{100}, \frac{1}{2(m+2)}, \frac{7(1-\tau)m}{8(m-2)})$ and $\zeta$ sufficiently small).

\begin{theorem}\label{thm-main_precise}
    Let $m \geq 3$, let $\iota: X \to M$ be a special Lagrangian immersion in a flat complex torus $(M = T^{2m},g, J, \omega, \Omega)$ satisfying Assumption \ref{assumption torus}, and let $\nu, \alpha, \tau, \mu, \zeta$ be real constants satisfying (\ref{assumption on constants}). Let $\stm$ be the corresponding abstract manifold as defined in Definition \ref{static_mfd}, and let $P^{1, 2, \alpha}_{\mu, \nu, \Lambda}, C^{0, \alpha}_{\zeta, \Lambda}$ be the Banach spaces on $\stm \times [\Lambda, \infty)$ and $[\Lambda, \infty)$ respectively as defined in Definition \ref{def-parabolicholdernorms}.
    
    Then there exist $u \in P^{1, 2, \alpha}_{\mu, \nu, \Lambda}$, $\vep:[\Lambda, \infty) \to (0, \infty)$ satisfying Assumption \ref{assumption on epsilon}, and $a:[\Lambda, \infty)\to\mathbb{R}$ such that 
    \begin{align}\label{PDE in simplest case}
    \begin{cases}
        \partial_{t}u = \theta(\dd u) + \xi(\dd u) + a(t) &\text{ for }t > \Lambda\\
        u(x, \Lambda) = 0 &\text{ on }\stm \times \{0\},
    \end{cases}
    \end{align}
    where $\theta(\dd u)$ is the Lagrangian angle of the Lagrangian embedding $\Psi_{N^{\vep(t)}} \circ \dd u:\stm \to M$ as in section \ref{sec-theequation}, and $\xi(\dd u)$ is defined in (\ref{l.o.t}) with constants $C_{P}$, $C_{Q^{\pm}}$ defined by (\ref{choice of constants in l.o.t}).

    The family $\Psi_{N^{\vep(t)}} \circ \dd u : \stm \to M$ of Lagrangian submanifolds satisfies mean curvature flow, and forms an infinite-time singularity. As $t \to \infty$ we have smooth convergence $N^\vep \to \iota(X)$ away from the transverse self-intersection point.
\end{theorem}

To prove Theorem \ref{thm-main_precise}, we first carefully define an iteration map $\mathscr{I}$ on the Banach space $P^{1, 2, \alpha}_{\mu, \nu, \Lambda} \times C^{0, \alpha}_{\zeta, \Lambda}$ for which a fixed point $(u,h)$ corresponds to a solution $(u, \vep)$ of (\ref{PDE in simplest case}). We then show that $\mathscr{I}$ maps a compact subset of $P^{1, 2, \alpha}_{\mu, \nu, \Lambda} \times C^{0, \alpha}_{\zeta, \Lambda}$ continuously into itself, and apply the Schauder fixed point theorem to conclude that a fixed point exists.

\subsection{Definition of the Iteration Map}\label{sec-iteration definition}
Denote the unit balls of $P^{1, 2, \alpha}_{\mu, \nu, \Lambda}$, $C^{0, \alpha}_{\zeta, \Lambda}$ by
\begin{align*}
    &\mathcal{B}^{\alpha}_{\mu, \nu, \Lambda} := \Big\{u\in P^{1, 2, \alpha}_{\mu, \nu, \Lambda}\::\:\|u\|_{P^{1, 2, \alpha}_{\mu, \nu, \Lambda}}\leq 1\Big\},\quad \mathcal{I}^{k,\alpha}_{\zeta, \Lambda} := \Big\{h\in C^{k, \alpha}_{\zeta, \Lambda}\::\:\|h\|_{P^{k, \alpha}_{\zeta, \Lambda}}\leq 1\Big\}.
\end{align*}

We now define the iteration map $\mathscr{I}:\mathcal{B}^{\alpha}_{\mu, \nu, \Lambda}\times\mathcal{I}^{0,\alpha}_{\zeta, \Lambda}\to \mathcal{B}^{\alpha}_{\mu, \nu, \Lambda}\times\mathcal{I}^{0,\alpha}_{\zeta, \Lambda}$. Given a pair $(u,h) \in \mathcal{B}^{\alpha}_{\mu,\nu,\Lambda} \times \mathcal{I}^{0,\alpha}_{\zeta, \Lambda}$, the pair $(v,k) = \mathscr{I}(u,h)\in P^{1,2,\alpha}_{\mu,\nu,\Lambda} \times C^{0,\alpha}_{\zeta, \Lambda}$ is defined as follows:

\ 

\noindent{\it Step~1. (Ansatz for $\vep(t)$)}:\quad First, we define 
\begin{equation}
    \vep(t) := \left[\frac{m-2}{2} \frac{A}{c_L}\frac{V_1 + V_2}{V_1V_2}\cdot t \, + \, \int_{\Lambda}^t h(s)\,\dd s \right]^{-\frac{1}{m-2}}, \label{def-iteration epsilon}
\end{equation}
and use $\vep(t)$ to construct the Lagrangian embedding $\iota^{\vep(t)}$ and related quantities and functions that depend on $\vep(t)$ as in section \ref{sec-desingularisation}. By definition, $\vep(t)$ satisfies the ODE:
\begin{align*}
		\frac{\dd\vep^{2}}{\dd t} + \frac{A}{c_L}\frac{V_1 + V_2}{V_1V_2}\vep^{m} = -\frac{2}{m-2}\vep^{m}h(t).
\end{align*}
It is easy to check that $\vep(t)$ satisfies Assumption \ref{assumption on epsilon}. We then construct the desingularisation $N^{\vep(t)}$ using $\vep(t)$.

\begin{comment}
\begin{align*}
    |\vep(t)| &= O(t^{-\frac{1}{m-2}}) ~,\\
    \vep'(t) &= -(\vep(t))^{m-1}\left[ \frac{A}{2c_L}\frac{V_1 + V_2}{V_1V_2} + \frac{1}{m-2}h(t) \right]\\
	&\implies |\vep'(t)| = O(t^{-\frac{m-1}{m-2}}) ~,\\
    \frac{\left|\vep'(t_2)-\vep'(t_1)\right|}{|t_2-t_1|^\af} &\lesssim  \frac{\left|(\vep(t_2))^{m-1}-(\vep(t_1))^{m-1}\right|}{|t_2-t_1|^\af} + \frac{\left|(\vep(t_2))^{m-1}h(t_2)-(\vep(t_1))^{m-1}h(t_1)\right|}{|t_2-t_1|^\af} \\
		&\leq \frac{\left|(\vep(t_2))^{m-1}-(\vep(t_1))^{m-1}\right|}{|t_2-t_1|^\af} \\
		&\quad + \frac{\left|(\vep(t_2))^{m-1}-(\vep(t_1))^{m-1}\right|\,|h(t_2)|}{|t_2-t_1|^\af} + \frac{(\vep(t_1))^{m-1}\,|h(t_2)-h(t_1)|}{|t_2-t_1|^\af} \\
		&\lesssim t^{-1}t^{-\frac{m-1}{m-2}}|t_2-t_1|^{1-\af} + t^{-\frac{m-1}{m-2}}t^{-\zeta} \\
		&\lesssim t^{-\frac{m-1}{m-2}}\left( t^{-1-\frac{2}{m-2}(1-\af)} + t^{-\zeta} \right) ~.
\end{align*}
\end{comment}

\ 

\noindent{\it Step~2. ($u\leadsto v$)}:\quad Next, we define $v \in P^{1, 2, \alpha}_{\mu, \nu, \Lambda}$. Define $\psi := \theta_{N^\vep} + \xi(0) + Q^\vep[\dd u]$. By Proposition~\ref{prop: zeroth order error} and Proposition~\ref{Prop: quadratic error}, we see that $\psi \in {P^{0,0,\alpha}_{\mu, \nu+2, \Lambda}}$. We may therefore apply Theorem \ref{linear theory}, to show that there exist $v \in P^{1,2,\alpha}_{\mu, \nu, \Lambda} \cap \langle 1, w^\vep\rangle^\perp$, $a:[\Lambda, \infty) \rightarrow \mathbb{R}$ and $b:[\Lambda, \infty) \rightarrow \mathbb{R}$ satisfying:
\begin{align}
		\begin{cases}
			\partial_{t}v - \mathcal{L}^{\vep}_{\ul{0}}[v] = \theta_{N^{\vep}} + \xi(0) + Q^{\vep}[\dd u] + a(t) + b(t)w^{\vep},\quad t\in[\Lambda, \infty),\\
			v(x, \Lambda) = 0,\quad x\in\ul{N},
		\end{cases} \label{iteration step}
\end{align}
and
\begin{align}
        \|v\|_{P^{1,2,\alpha}_{\mu,\nu,\Lambda}} \leq C\cdot \|\psi\|_{P^{0,0,\alpha}_{\mu, \nu + 2, \Lambda}}. \label{v estimate}
\end{align}

\ 

\noindent{\it Step~3. ($h\leadsto k$)}:\quad Finally, we define $k \in C^{0,\alpha}_{\zeta, \Lambda}$. Integrating (\ref{iteration step}) against the functions $1$ and $w^\vep$ respectively, and using the projection formulae (\ref{full projection onto 1}) and (\ref{full projection onto normalised approx kernel}), we obtain the following expressions for $a(t)$ and $b(t)$:
\begin{align}
    a(t) %&= \frac{1}{{\rm Vol}(N^\vep)} \int_{\ul N} \left( \partial_t v - \mathcal{L}_{\ul 0}^\vep[v] - \psi\right) \:\dd V_{g^\vep}\notag\\
    &= \frac{1}{{\rm Vol}(N^\vep)} \left(  \int_{\ul N}(\partial_t v - \mathcal{L}_{\ul 0}^\vep[v] - Q^\vep[\dd u])\:\dd V_{g^\vep} + O(\vep^{(1+\tau)m}) \right) ~, \label{a estimate}\\
    b(t) %&=  \frac{1}{\|w^\vep\|_{L^2}^2} \int_{\ul N} \left( \partial_t v - \mathcal{L}_{\ul 0}^\vep[v] - \psi\right) w^\vep \:\dd V_{g^\vep}\notag\\
    %&= \frac{1}{\|w^\vep\|_{L^2}^2} \left(  \int_{\ul N}(\partial_t v - \mathcal{L}_{\ul 0}^\vep[v] - Q^\vep[\dd u])w^\vep \: \dd V_{g^\vep}\right.\notag\\
    %&\left.\qquad\qquad\qquad - c_L\frac{V_{1}V_{2}}{V_{1}+V_{2}}\left[\frac{\dd\vep^{2}}{\dd t} + \frac{\vep^{m}A}{c_L}\frac{V_{1}+V_{2}}{V_{1}V_{2}}\right] + O(\vep^{(1+\tau)m}) \right)\notag\\
    &= \frac{1}{\|w^\vep\|_{L^2}^2} \left(  \int_{\ul N}(\partial_t v - \mathcal{L}_{\ul 0}^\vep[v] - Q^\vep[\dd u])w^\vep \:\dd V_{g^\vep} + \frac{c_LV_{1}V_{2}}{V_1 + V_2} \frac{2 \vep^m h(t)}{m-2} + O(\vep^{(1+\tau)m}) \right) ~. \label{b estimate}
\end{align}
It is therefore natural to define $k(t)$ as follows, in order to cancel out the dominant term from this expansion of $b(t)$: 
    \begin{align}
		k(t) := h(t) - \frac{m-2}{2c_L}\frac{V_{1}+V_{2}}{V_{1}V_{2}}\vep^{-m}\|w^{\vep}\|^{2}_{L^{2}}\cdot b(t) ~.	 \label{k definition}
	\end{align}

\subsection{Estimates for the Iteration Map} In order to apply the Schauder fixed point theorem, we now aim to prove the following proposition regarding the iteration map $\mathscr{I}$:

\begin{proposition}\label{prop-iteration estimate}
    For any $\mu' < \mu$, $\alpha' < \alpha$, $\zeta' < \zeta$, the iteration map $\mathscr{I}: \mathcal{B}^\alpha_{\mu, \nu, \Lambda} \times \mathcal{I}^{0,\alpha}_{\zeta, \Lambda} \to P^{1, 2, \alpha}_{\mu, \nu, \Lambda} \times C^{0, \alpha}_{\zeta, \Lambda}$ defined in section \ref{sec-iteration definition} is continuous with respect to the norm on $P^{1,2,\alpha'}_{\mu', \nu, \Lambda} \times C^{0, \alpha'}_{\zeta', \Lambda}$, and has image lying in $\mathcal{B}^\alpha_{\mu, \nu, \Lambda} \times \mathcal{I}^{0,\alpha}_{\zeta, \Lambda}$.
\end{proposition}

We first estimate the projection of the inhomogeneous term $\psi$ onto the approximate kernel.
\begin{lemma}\label{lem: balancing condition}
    Let
    \begin{align}\label{definition of G and H}
        G(t) &:= c_L\frac{V_{1}V_{2}}{V_{1}+V_{2}}\left\{\frac{\dd\vep^{2}(t)}{\dd t} + \frac{\vep^{m}(t)A}{c_L}\frac{V_1 + V_2}{V_1V_2}\right\} - \int_{\ul{N}}\psi\cdot w^{\vep}\:\dd V_{g^{\vep}} ~.
    \end{align}
    %%% DEFINITION OF H, CURRENTLY NOT USED %%%
    % \begin{align}
    %     H(t) &:= \int_{\ul{N}}\psi\:\dd V_{g^{\vep}}.
    % \end{align}
    Then, if $\|u\|_{P^{1, 2, \alpha}_{\mu, \nu, \Lambda}}\leq 1$, $\nu\in(\max\{\frac{m}{2}-2, 0\}, m-2)$ and $\mu>\frac{\nu+2}{m-2}$, then for any $\overline \zeta > 0$ satisfying
    \begin{align*}
        \overline\zeta < \min\{\textstyle\frac{\tau m}{m-2},  2(\mu-\frac{\nu+2}{m-2})\}
    \end{align*}
    and $0<|t_1 - t_2|<t^{-\frac{2}{m-2}}$, it follows that
    %%% OLD VERSION, INCLUDES G AND H %%%
    % \begin{align}\label{C0 alpha of G and H}
    %     |G(t)| + |H(t)|\leq C\vep(t)^{m}\cdot t^{-\zeta}, \quad \frac{|G(t) - G(t')|}{|t - t'|^{\alpha}} + \frac{|H(t) - H(t')|}{|t - t'|^{\alpha}}\leq C\vep(t)^{m}\cdot t^{-\zeta}
    % \end{align}
    \begin{align}\label{C0 alpha of G and H}
        |G(t)| + \frac{|G(t_1) - G(t_2)|}{|t_1 - t_2|^{\alpha}} \,\leq\, C\vep(t)^{m}\cdot t^{-\overline\zeta} ~.
    \end{align}
\end{lemma}
\begin{proof}
    By projection formula (\ref{full projection onto normalised approx kernel}) we have
    \begin{align*}
        |G(t)|\leq C\vep(t)^{(\tau+1)m} + \int_{\nu_{N}}|Q^{\vep}[\dd u]|\cdot w^{\vep}\:\dd V_{g^{\vep}} ~.
    \end{align*}
    Since $\|u\|_{P^{1, 2, \alpha}_{\mu, \nu, \Lambda}}\leq 1$, by (\ref{C0 of Q}) we have
    \begin{align*}
        \int_{\nu_{N}}|Q^{\vep}[\dd u]|\cdot w^{\vep}\:\dd V_{g^{\vep}}\leq C\: t^{-2\mu}\int_{\ul{N}}\rho^{-2\nu-4}_{\vep}\:\dd V_{g^{\vep}}  ~.
    \end{align*}
    Note that for any small $a>0$, by assumption we have $-2\nu - 4+m-a<0$. Hence,
    \begin{align*}
        \int_{\ul{N}}\rho^{-2\nu-4}_{\vep}\:\dd V_{g^{\vep}} = \int_{\ul{N}}\rho^{-2\nu-4+m-a}_{\vep}\rho_{\vep}^{-m+a}\:\dd V_{g^{\vep}}\leq C\vep(t)^{-2\nu -4 +m - a}\int_{\ul{N}}\rho_{\vep}^{-m+a}\:\dd V_{g^{\vep}} ~.
    \end{align*}
    As $\rho_{\vep}\to \hat{r}$ and $\rho_{\vep}^{-m+a}\leq C\hat{r}^{-m+a}$ on $\ul{N}$, where $\hat{r}$ is the intrinsic distance to the intersection point on $X_{1}\cup X_{2}$, the dominated convergence theorem implies
    \begin{align*}
        \int_{\ul{N}}\rho_{\vep}^{-m+a}\:\dd V_{g^{\vep}}\leq \int_{X_{1}\cup X_{2}}\hat{r}^{-m+a}\:\dd V_{X'}\leq C(a)<\infty
    \end{align*}
    for all $t\geq \Lambda$. It follows that, by choosing $a$ sufficiently small such that $\mu>\frac{\nu+2}{m-2}+\frac{a}{2}$,
    \begin{align*}
        \int_{\ul{N}}|Q^{\vep}(\dd u)|\cdot w^{\vep}\:\dd V_{g^{\vep}}\leq C(a) \vep(t)^{m}t^{-2\mu + \frac{2}{m-2}(\nu+2+\frac{a}{2})}\leq C(a)\vep(t)^{m}t^{-2(\mu - \frac{\nu+2}{m-2} - \frac{a}{2})} ~,
    \end{align*}
    which shows 
    \begin{align*}
        |G(t)|\leq C\vep(t)^{m}\cdot t^{-\overline\zeta}.
    \end{align*}
    A similar argument using (\ref{C^alpha of Q in space}) gives
    \begin{align*}
        \frac{|G(t_1) - G(t_2)|}{|t_1 -t_2|^{\alpha}}\leq C\vep(t)^{m}\cdot t^{-\overline\zeta}
    \end{align*}
    providing $0<|t_1-t_2|<t^{-\frac{2}{m-2}}$. 
    % The estimate for $H(t)$ follows the same computation, using (\ref{full projection onto 1}) instead of (\ref{full projection onto normalised approx kernel}).
\end{proof}

\begin{lemma}\label{lem: estimate of non-orthogonality}
    Given $\psi\in P^{0, 0, \alpha}_{\mu, \nu+2, \Lambda}$ with $\|\psi\|_{P^{0, 0, \alpha}_{\mu, \nu+2, \Lambda}}\leq 1$, let the triple $(u, a(t), b(t))$ be the solution to the Cauchy problem
    \begin{align*}
    \begin{cases}
        \partial_{t}u - \mathcal{L}^{\vep}_{\ul{0}}[u] = \psi + a(t) + b(t)w^{\vep}, &\mbox{on}\;\ul{N}\times(\Lambda, \infty) ~, \\
        u(\cdot, \Lambda) = 0, &\mbox{on}\;\ul{N} ~,
    \end{cases}
    \end{align*}
    with orthogonality condition $u \in \langle 1, w^\vep \rangle^\perp$. Define 
    \begin{align}
        E(t) := \int_{\ul{N}}(\psi + b(t)w^{\vep})\cdot w^{\vep}\:\dd V_{g^{\vep}} ~.
        % \quad F(t) := \int_{\ul{N}}(\psi + a(t))\:\dd V_{g^{\vep}}.
    \end{align}
    Then if $\mu > \frac{\nu + 2 + 2\alpha}{m-2}$, then for $0<|t_1 -t_2|<t^{-\frac{2}{m-2}}$ and $\overline\zeta > 0$ satisfying $\overline\zeta < \mu - \frac{\nu + 2 + 2\alpha}{m-2}$,
    \begin{align}
        |E(t)| + \frac{|E(t_1) - E(t_2)|}{|t_1 -t_2|^{\alpha}} \leq C\:\vep(t)^{m}t^{-\zeta} ~.
    \end{align}
    % and
    % \begin{align}
    %     |F(t)| + \frac{|F(t) - F(t')|}{|t-t'|^{\alpha}}\leq C\:\vep(t)^{m}t^{-\zeta}.
    % \end{align}
\end{lemma}

\begin{proof}
   Write $E(t) = E_{0}(t)- E_{1}(t)$, where
    \begin{align*}
        E_{0}(t) = \int_{\ul{N}}\partial_{t}u\cdot w^{\vep}\:\dd V_{g^{\vep}}, \quad E_{1}(t) = \int_{\ul{N}}\mathcal{L}^{\vep}_{\ul{0}}[u]\cdot w^{\vep}\:\dd V_{g^{\vep}} ~.
    \end{align*}
    We first estimate $E_{0}(t)$. By differentiating orthogonality condition in time, we have
    \begin{align*}
        E_{0}(t) = -\int_{\ul{N}}u\cdot\partial_{t}w^{\vep}\:\dd V_{g^{\vep}} - \int_{\ul{N}}u\cdot w^{\vep}\:\partial_{t}\dd V_{g^{\vep}} ~.
    \end{align*}
    It follows from Lemma \ref{lem-volumeestimates} and Lemma~\ref{lem-approxkernelestimates}, and the assumption on $\vep(t)$ that
    \begin{align*}
        |E_{0}(t)| + \frac{|E_{0}(t_1) - E_{0}(t_2)|}{|t_1 - t_2|^{\alpha}} \leq C\:\vep(t)^{m}\: t^{-\mu+\frac{1}{m-2}(\nu+2+2\alpha)}\|u\|_{P^{0, 0, \alpha}_{\mu, \nu, \Lambda}}
    \end{align*}
    for $0<|t_1 - t_2|<t^{-\frac{2}{m-2}}$.

    To estimate $E_{1}(t)$, write
    \begin{align*}
        E_{1}(t) = \int_{\ul{N}}\Delta_{g^{\vep}}u\cdot w^{\vep}\:\dd V_{g^{\vep}} + \int_{\ul{N}}\mathcal{P}^{\vep}_{\ul{0}}[u]\cdot w^{\vep}\:\dd V_{g^{\vep}} ~,
    \end{align*}
    where $\mathcal{P}^{\vep}_{\ul{0}} := \mathcal{L}^{\vep}_{\ul{0}} - \Delta_{g^{\vep}}$. Then by Lemma~\ref{lem-approxkernelestimates} and Lemma~\ref{lem-lower order term estimate} we obtain
    \begin{align*}
        |E_{1}(t)|&\leq\int_{\ul{N}}|u|\cdot|\Delta w^{\vep}|\:\dd V_{g^{\vep}} + \int_{\ul{N}}|\mathcal{P}^{\vep}_{\ul{0}}[u]|\cdot|w^{\vep}|\:\dd V_{g^{\vep}}\\
        &\leq C\:\|u\|_{P^{1, 2, \alpha}_{\mu, \nu, \Lambda}}\:t^{-\mu+\frac{\nu+2}{m-2}}\:\vep(t)^{m} ~.
    \end{align*}
    Similar estimate using the H\"older estimates in Lemma~\ref{lem-approxkernelestimates}, Lemma~\ref{lem-volumeestimates} and Lemma~\ref{lem-lower order term estimate} yields
    \begin{align*}
        \frac{|E_{1}(t_1) - E_{1}(t_2)|}{|t_1 - t_2|^{\alpha}}\leq C\:\|u\|_{P^{1, 2, \alpha}_{\mu, \nu, \Lambda}}\:t^{-\mu+\frac{\nu+2+2\alpha}{m-2}}\:\vep(t)^{m}
    \end{align*}
    for $0<|t_1 - t_2|<t^{-\frac{2}{m-2}}$.
    Hence,
    \begin{align*}
        |E_1(t)| + \frac{|E_1(t_1) - E_{1}(t_2)|}{|t_1 - t_2|^{\alpha}} \leq C\:\vep(t)^{m}\:t^{-\mu+\frac{\nu+2+2\alpha}{m-2}}\|u\|_{P^{1, 2, \alpha}_{\mu, \nu, \Lambda}}
    \end{align*}
    for $0<|t_1 - t_2|<t^{-\frac{2}{m-2}}$. 
    Combining these estimates, along with Theorem \ref{linear theory}, we obtain
    \begin{align*}
        |E(t)|+\frac{|E(t_1) - E(t_2)|}{|t_1 - t_2|^{\alpha}}\leq C\:\vep(t)^{m}\:t^{-\mu+\frac{\nu+2+2\alpha}{m-2}}
    \end{align*}
    for $0<|t_1 - t_2|<t^{-\frac{2}{m-2}}$. 
    %%% PROOF OF ESTIMATE FOR F, CURRENTLY UNUSED %%%
    % Similarly, we write $F(t) = F_{0}(t) - F_{1}(t)$, where
    % \begin{align*}
    %     F_{0}(t) = \int_{\ul{N}}\partial_{t}u\:\dd V_{g^{\vep}},\quad F_{1}(t) = \int_{\ul{N}}\mathcal{L}^{\bvep}_{\ul{0}}[u]\:\dd V_{g^{\vep}}.
    % \end{align*}
    % Differentiating the orthogonality condition and using Lemma~\ref{lem-volumeestimates} we have
    % \begin{align*}
    %     |F_{0}(t)|\leq \int_{\ul{N}}|u|\cdot|\partial_{t}\dd V_{g^{\vep}}|\leq\int_{\ul{N}}t^{-\mu}\rho_{\bvep}^{-\nu}\|u\|_{P^{0, 0, \alpha}_{\mu, \nu, \Lambda}}\:\vep^{2m-3}\:\dd V_{g^{\vep}}\leq C\:\vep(t)^{m}\:t^{-\mu+\frac{\nu+3-m}{m-2}}\|u\|_{P^{0, 0, \alpha}_{\mu, \nu, \Lambda}}.
    % \end{align*}
    % We also have
    % \begin{align*}
    %     \frac{|F_{0}(t_1) - F_{0}(t_2)|}{|t_1 - t_2|^{\alpha}}\leq C\:\vep(t)^{m}\:t^{-\mu+\frac{\nu+3-m+2\alpha}{m-2}}\|u\|_{P^{0, 0, \alpha}_{\mu, \nu, \Lambda}}.
    % \end{align*}
    % For $F_{1}(t)$, by divergence theorem and Lemma~\ref{lem-lower order term estimate} we have
    % \begin{align*}
    %     |F_1(t)|\leq\int_{\ul{N}}|\mathcal{P}^{\bvep}_{\ul{0}}[u]|\:\dd V_{g^{\vep}}\leq\int_{P\cup Q^{\pm}}C\:\vep(t)^{m-1}|\dd u|_{g^{\vep}}\:\dd V_{g^{\vep}}\leq C\:\vep(t)^{m}\:t^{-\mu+\frac{\nu+2}{m-2}}\|u\|_{P^{1,2,\alpha}_{\mu, \nu, \Lambda}},
    % \end{align*}
    % and ... (UNFINISHED)
\end{proof}

\begin{proof}[Proof of Proposition \ref{prop-iteration estimate}.] To show that $\mathscr{I}$ is continuous with respect to the norm of $P^{1,2,\alpha'}_{\mu', \nu, \Lambda} \times C^{0, \alpha'}_{\zeta', \Lambda}$, one may use a contradiction argument as in the proof of \cite{BrendleK2017}*{Proposition~5.3}.

For the estimate on $k(t)$, note that by definition we have
\begin{align*}
    k(t) &= h(t) - \frac{m-2}{2c_L}\frac{V_{1}+V_{2}}{V_{1}V_{2}}\vep^{-m}\|w^{\vep}\|^{2}_{L^{2}}\cdot b(t)\\
    %&= -\vep(t)^{-m}\frac{m-2}{2}\left\{\frac{\dd\vep^{2}(t)}{\dd t} + \frac{\vep^{m}(t)A}{c_L}\frac{V_1 + V_2}{V_1V_2}+ \frac{1}{c_L}\frac{V_{1}+V_{2}}{V_{1}V_{2}}\|w^{\vep}\|^{2}_{L^{2}}b(t)\right\}\\
    &= -\vep(t)^{-m}\frac{(m-2)(V_1 + V_{2})}{2c_LV_1 V_{2}}(G(t) + E(t)) ~.
\end{align*}
It follows from Lemma~\ref{lem: estimate of non-orthogonality} and Lemma~\ref{lem: balancing condition} that we may choose $\overline \zeta > \zeta$ such that 
\begin{align}
    |k(t)| + \frac{|k(t_1) - k(t_2)|}{|t_1 -t_2|^{\alpha}} &\leq C t^{-\overline \zeta},\quad 0<|t_1 -t_2|<t^{-\frac{2}{m-2}}.\\
    \implies \|k \|_{P^{0, \alpha}_{\zeta, \Lambda}} \, &\leq \, C\Lambda^{-(\overline\zeta - \zeta)} ~.
\end{align}
Finally, we may estimate $v$ using (\ref{v estimate}), Proposition~\ref{prop: zeroth order error}, and Proposition~\ref{Prop: quadratic error}:
\begin{align*}
    \|v\|_{P^{1,2,\alpha}_{\mu,\nu,\Lambda}} &\leq C\left(\|\theta_N^\vep\|_{P^{0,0,\alpha}_{\mu,\nu+2,\Lambda}} + \|\xi(0)\|_{P^{0,0,\alpha}_{\mu,\nu+2,\Lambda}} + \|Q^\vep[\dd u]\|_{P^{0,0,\alpha}_{\mu,\nu+2,\Lambda}}\right)\\
    &\leq C\left(\Lambda^{\mu - \frac{1}{m-2}\left( \tau(\nu+2) + (1-\tau)m \right)} + \Lambda^{\mu-\frac{1}{m-2}\left( m - 2\alpha \right)} + \Lambda^{-\mu+\frac{\nu+2}{m-2}}\right) ~.
\end{align*}
Taking $\Lambda$ sufficiently large therefore ensures that $\mathscr{I}$ maps $\mathcal{B}^{\alpha}_{\mu, \nu, \Lambda}\times\mathcal{I}^{0, \alpha}_{\zeta, \Lambda}$ to itself, as required.
\end{proof}

\subsection{Proof of Theorem \ref{thm-main_precise}}

Consider the iteration map $\mathscr{I}$ defined in section \ref{sec-iteration definition}. By Proposition \ref{prop-iteration estimate}, it may be viewed as a function on the product of unit balls, $\mathscr{I}:\mathcal{B}^{\alpha}_{\mu, \nu, \Lambda}\times \mathcal{I}^{0, \alpha}_{\zeta, \Lambda} \, \to \, \mathcal{B}^{\alpha}_{\mu, \nu, \Lambda}\times \mathcal{I}^{0, \alpha}_{\zeta, \Lambda}$.

By Lemma \ref{lem-compactness}, $\mathcal{B}^{\alpha}_{\mu, \nu, \Lambda} \times \mathcal{I}^{0, \alpha}_{\zeta, \Lambda}$ is a compact subset of $P^{1,2,\alpha'}_{\mu', \nu, \Lambda} \times C^{0, \alpha'}_{\zeta', \Lambda}$ for any $\mu' < \mu$, $\alpha' < \alpha$, $\zeta' < \zeta$. Since $\mathscr{I}$ is a continuous map by Proposition \ref{prop-iteration estimate}, we may therefore apply the Schauder fixed point theorem to conclude that there exist $(u, h) \in \mathcal{B}^{\alpha}_{\mu,\nu,\Lambda} \times \mathcal{I}^{0, \alpha}_{\zeta, \Lambda}$ such that $(v, k) := \mathscr{I}(u,h) =  (u,h)$. Define $\vep(t)$ and the Lagrangian embedding $\iota^\vep$ using the function $h$ as in (\ref{def-iteration epsilon}). Since $h\in\mathcal{I}^{0, \alpha}_{\zeta, \Lambda}$, $\vep(t)$ satisfies Assumption~\ref{assumption on epsilon}. By (\ref{iteration step}) and (\ref{k definition}), the fixed point $(u,h)$ satisfies
\begin{align*}
h(t) = k(t) &\implies b(t) = 0 ~,\\
u(t) = v(t) &\implies \partial_{t}u - \mathcal{L}^{\vep}_{\ul{0}}[u] = \theta_{N^{\vep}} + \xi(0) + Q^{\vep}[\dd u] + a(t)\\
&\implies \partial_t u = \theta(\dd u) \, + \, \xi(\dd u) + a(t),
\end{align*}
as required. 

Finally, since $u\in\mathcal{B}^{\alpha}_{\mu, \nu, \Lambda}$, we have
\begin{align}\label{C1 estimate}
    |\dd u(x, t)|_{g^{\vep}}\leq \|u\|_{P^{1, 2, \alpha}_{\mu, \nu, \Lambda}}\:t^{-\mu}\rho_{t^{-\frac{1}{m-2}}}^{-\nu-1}(x)\leq C\:t^{-\mu+\frac{\nu+2}{m-2}}\:\vep(t),\quad \mbox{ for all }(x, t)\in\ul{N}\times[\Lambda, \infty).
\end{align}
Since $\mu>\frac{\nu+2}{m-2}$, this shows that the time-dependent $1$-form $\dd u(\cdot, t)$ is contained in the Lagrangian neighborhood $U_{N^{\vep(t)}}$ for all $t\geq\Lambda$ for $\Lambda$ sufficiently large. We may then apply Proposition~\ref{dMCF} to obtain a solution to the mean curvature flow given by $\Psi_{N^{\vep(t)}}\circ\dd u(\cdot, t)$. The estimate (\ref{C1 estimate}) implies that $\Psi_{N^{\vep(t)}}\circ\dd u$ converges to the immersion $\iota:X\to M$.
\qed

We end this section by studying the convergence of the mean curvature flow solution $N_{t} := \Psi_{N^{\vep(t)}}\circ\dd u (\ul{N}, t)$ as $t \to\infty$.

\begin{proposition}\label{prop: tip convergence} We have the following locally smooth convergence of submanifolds in $\mathbb{C}^m$:
    \begin{align}
        \vep(t)^{-1}\left[\Upsilon^{-1}(N_{t})\cap B_{\vep(t)^{\tau}}\right]\to L\quad\mbox{as}\;\;t\to\infty.
    \end{align}
\end{proposition}

\begin{proof}
    By construction we know that $\vep(t)^{-1}\Upsilon^{-1}(N_{t})\cap B_{\vep(t)^{\tau-1}}$ can be written as a graph $\Phi_{L}\circ(\vep(t)^{-2}\dd u)$ over $L$. Then $u\in\mathcal{B}^{\alpha}_{\mu, \nu, \Lambda}$ implies that $|\dd u|_{g_{L}} + |\nabla \dd u|_{g_{L}} = o(\vep(t)^{2})$ as $t\to\infty$. Hence, the graph $\Phi_{L}\circ(\vep(t)^{-2}\dd u)$ converges to $L$ as $t\to\infty$ on $B_{\vep(t)^{\tau-1}}$, locally in the $C^1$-sense. Since $u$ is a solution to the parabolic equation (\ref{MCF1_tip}) on $L$, by parabolic regularity the convergence is locally smooth.
\end{proof}

\begin{corollary}\label{cor: curvature blow up}
    The second fundamental form $|A_{N_{t}}|$ blows up at a rate $O\left(t^{\frac{1}{m-2}}\right)$ as $t\to\infty$.
\end{corollary}
\begin{proof}
    By construction and the fact that $u\in\mathcal{B}^{\alpha}_{\mu, \nu, \Lambda}$, $|A_{N_{t}}|$ remains bounded away from the region $N_{t}\cap\Upsilon(B_{\vep(t)^{\tau}})$. The blow-up rate now follows from Proposition~\ref{prop: tip convergence}.
\end{proof}

\end{document}